\documentclass[a4paper,12pt,leqno]{amsart}
\usepackage{latexsym}
\usepackage[all]{xy}

\usepackage{amssymb} 
\usepackage{amsmath} 
\usepackage{color}
\usepackage{comment}
\usepackage{mathtools}

\usepackage{graphicx}
\usepackage{stmaryrd}

\definecolor{gray}{gray}{0.7}
\definecolor{Gray}{gray}{0.3}

\usepackage{amscd}
\usepackage{mathdots}

\numberwithin{equation}{section}

\theoremstyle{break}
 \newtheorem{theorem}{Theorem}[section]
 \newtheorem{proposition}[theorem]{Proposition}
 \newtheorem{corollary}[theorem]{Corollary}
 \newtheorem{lemma}[theorem]{Lemma}

 \theoremstyle{definition}
 \newtheorem{definition}[theorem]{Definition}
 \newtheorem{remark}[theorem]{Remark}
 \newtheorem{example}[theorem]{Example}

\allowdisplaybreaks[3]

\def\C{\mathbb C}

\def\Q{\mathbb Q}
\def\Z{\mathbb Z}
\def\YY{\mathcal{Y}}
\def\ll{\mathbf{l}}
\def\rr{\mathbf{r}}
\def\Plucker{\mathsf{P}}

\def\q{\mathbf{q}}
\def\hh{\mathbf{h}}
\def\SS{\mathfrak{S}}

\DeclareMathOperator{\GL}{GL}
\DeclareMathOperator{\Hess}{Hess}
\DeclareMathOperator{\Fl}{Fl}

\begin{document}

\title[Symmetric group actions on coordinate ring and polynomial ring]{Symmetric group actions on the coordinate ring of the lower unipotent group and the polynomial ring with quantum parameters}
\author [T. Horiguchi]{Tatsuya Horiguchi}
\address{National Institute of Technology, Akashi College, 679-3, Nishioka, Uozumi-cho, Akashi, Hyogo 674-8501, Japan}
\email{tatsuya.horiguchi0103@gmail.com}



\begin{abstract}
Givental–Kim and Ciocan–Fontanine gave an explicit presentation of the quantum cohomology ring of the flag variety. 
The author and Shirato introduced an algebraic generalization of their presentation in the context of the coordinate rings of regular nilpotent Hessenberg varieties. 
In particular, they connect the coordinate ring of the lower unipotent group in a general linear group and the polynomial ring with quantum parameters. 
In this paper we consider the dual of the connection and we see that the Pl\"{u}cker coordinates correspond to the quantizations of Schur polynomials.
As an application of the connection, we construct an action of the symmetric group on the polynomial ring with quantum parameters. 
Using the symmetric group action, one can define the divided difference operators on the polynomial ring with quantum parameters. 
We study the quantizations of Schubert polynomials in relation to the divided difference operators. 
\end{abstract}

\maketitle

\setcounter{tocdepth}{1}

\tableofcontents

\section{Introduction}
\label{section:introduction}

In this paper we construct an action of a symmetric group on the polynomial ring with quantum parameters via the connection with the coordinate ring of the lower unipotent group in a genaral linear group. 

\subsection{Motivation}

Let $n$ be a fixed positive integer and $S_n$ the symmetric group on $n$ letters $\{1,2,\ldots,n\}$ acting on a polynomial ring $\Z[x_1,\ldots,x_n]$ by permuting its variables. 
Bernstein--Gelfand--Gelfand \cite{BGG} and Demazure \cite{Dem74} independently introduced the divided difference operator $\partial_k$ for each $1 \leq k \leq n-1$ by the following formula
\begin{align*} 
\partial_k(f) = \frac{f-s_k(f)}{x_k-x_{k+1}} \ \ \ \textrm{for} \ f \in \Z[x_1,\ldots,x_n],
\end{align*}
where $s_k$ denotes the adjacent transposition of $k$ and $k+1$. 
The importance in geometry was motivated by the context of Schubert calculus, as we brieﬂy explain now. 
Let $\Fl(\C^n)$ be the (full) flag variety in $\C^n$ consisting of all nested complex linear subspaces $V_\bullet \coloneqq (V_1 \subset V_2 \subset \dots \subset V_n = \C^n)$ of $\C^n$ where each $V_i$ is complex $i$-dimension for $1 \leq i \leq n$. 
It is well-known that the cohomology ring of $\Fl(\C^n)$ with $\Z$-coefficients is isomorphic to
\begin{align*} 
H^*(\Fl(\C^n);\Z) \cong \Z[x_1,\ldots,x_n]/(e_1^{(n)},\ldots,e_n^{(n)})
\end{align*}
as graded rings (\cite{Bor53}).
Here, we take $\deg x_i = 2$ for all $1 \leq i \leq n$, and $e_i^{(n)}$ is the $i$-th elementary symmetric polynomial in the variables $x_1,\ldots,x_n$. 
There exists an additive basis for the cohomology of $\Fl(\C^n)$, called Schubert classes $\{\sigma_w \}_{w \in S_n}$ which are the Poincar\'{e} dual to the (opposite) Schubert varieties. 
The Schubert polynomials $\SS_w$ give good polynomial representatives for the Schubert classes $\sigma_w$ in $H^*(\Fl(\C^n);\Z)$ satisfying 
\begin{align} \label{eq:Intro_divided_difference_Schubert}
\partial_k (\SS_w) = \begin{cases}
\SS_{ws_k} \ & \textrm{if} \ w(k) > w(k+1); \\
0 \ & \textrm{if} \ w(k) < w(k+1). 
\end{cases}
\end{align}

We shall consider a quantum analogue of the property above. 
Set quantum parameters $\q_o = \{q_1,\ldots,q_{n-1} \}$ and the matrix 
\begin{align*}
M_n^{\q_o} = \left(
 \begin{array}{@{\,}ccccc@{\,}}
     x_1 & q_1 & 0 & \cdots & 0 \\
     -1 & x_2 & q_2 & \ddots & \vdots \\ 
      0 & \ddots & \ddots & \ddots & 0 \\ 
      \vdots & \ddots & -1 & x_{n-1} & q_{n-1} \\
      0 & \cdots & 0 & -1 & x_n 
 \end{array}
 \right).
\end{align*}
Then the \emph{quantized elementary symmetric polynomials} $E_1^{(n) \, \q_o}, \ldots, E_n^{(n) \, \q_o}$ are defined by 
\begin{align*}
\det(t I_n - M_n^{\q_o}) = t^n - E_1^{(n) \, \q_o} \, t^{n-1} + E_2^{(n) \, \q_o} \, t^{n-2} + \cdots + (-1)^n E_n^{(n) \, \q_o},
\end{align*}
where $I_n$ is the identity matrix of order $n$. 
In the classical limit $q_i= 0$ for all $1 \leq i \leq n-1$, the quantized elementary symmetric polynomial $E_i^{(n) \, \q_o}$ specializes to the $i$-th elementary symmetric polynomial $e_i^{(n)}$.
Givental–Kim \cite{GivKim} and Ciocan-Fontanine \cite{Font95} gave an explicit presentation of the quantum cohomology ring of $\Fl(\C^n)$ as follows:
\begin{align} \label{eq:Intro_quantum_cohomology_flag}
QH^*(\Fl(\C^n);\Z) \cong \Z[x_1,\ldots,x_n, q_1,\ldots,q_{n-1}]/(E_1^{(n) \, \q_o},\ldots,E_n^{(n) \, \q_o}).
\end{align}
Note that this is an isomorphism of graded rings where we take $\deg q_i =4$ for any $1 \leq i \leq n-1$.
Fomin--Gelfand--Postnikov defined in \cite{FGP} the quantum Schubert polynomials $\SS_w^{\q_o}$ by introducing and using the quantization map. 
Here and below, we explain the combinatorial definition of the quantization in \cite{FGP}.
For $i_1,\ldots,i_m$ with $0 \leq i_k \leq k$, the standard elementary monomial $e_{i_1,\ldots,i_m}$ and quantum standard elementary monomial $E_{i_1,\ldots,i_m}^{\q_o}$ are defined by 
\begin{align*}
e_{i_1,\ldots,i_m} = e_{i_1}^{(1)} \cdots e_{i_m}^{(m)} \ \textrm{and} \ E_{i_1,\ldots,i_m}^{\q_o} = E_{i_1}^{(1) \, \q_o} \cdots E_{i_m}^{(m) \, \q_o}. 
\end{align*}
Since the set of the standard elementary monomials $\{e_{i_1,\ldots,i_m}\}$ forms an additive basis of the polynomial ring $\Z[x_1,x_2,\ldots]$ in infinitely many variables, arbitrary polynomial $f \in \Z[x_1,x_2,\ldots]$ can be uniquely written as $f = \sum c_{i_1,\ldots,i_m} e_{i_1,\ldots,i_m}$ for some $c_{i_1,\ldots,i_m} \in \Z$.
Then the $\q_o$-quantization $F$ of the polynomial $f$ is defined by 
\begin{align*}
F = \sum c_{i_1,\ldots,i_m} E_{i_1,\ldots,i_m}^{\q_o}.
\end{align*}
The quantum Schubert polynomial $\SS_w^{\q_o}$ is the $\q_o$-quantization of the Schubert polynomial $\SS_w$.
It is natural to ask whether there exists the divided difference operators on the polynomial ring $\Z[x_1,\ldots,x_n,q_1,\ldots,q_{n-1}]$.
For this purpose, it is natural to consider an action of the symmetric group $S_n$ on $\Z[x_1,\ldots,x_n,q_1,\ldots,q_{n-1}]$. 
Unfortunately, we could \emph{not} find such an action. 
However, it is successful by adding more quantum parameters to construct an $S_n$-action on the polynomial ring $\Z[x_1,\ldots,x_n,q_{ij} \mid 1 \leq i < j \leq n]$. 
In this paper we explain how to construct the $S_n$-action on $\Z[x_1,\ldots,x_n,q_{ij} \mid 1 \leq i < j \leq n]$.

\subsection{Polynomial ring and coordinate ring}

From now on, we consider the polynomial ring 
\begin{align*}
\Z[x_1,\ldots,x_n,q_{ij} \mid 1 \leq i < j \leq n]
\end{align*} 
with $\deg q_{ij} = 2(j-i+1)$ for $1 \leq i < j \leq n$. 
Note that it is introduced in \cite{HorShi} in the context of the coordinate rings of regular nilpotent Hessenberg varieties, which are subvarieties of the flag variety $\Fl(\C^n)$.  
We write the quantum parameters $\q=(q_{ij} \mid 1 \leq i < j \leq n)$.
Similarly, we set the matrix 
\begin{align*}
M_n^{\q} = \left(
 \begin{array}{@{\,}ccccc@{\,}}
     x_1 & q_{12} & q_{13} & \cdots & q_{1n} \\
     -1 & x_2 & q_{23} & \cdots & q_{2n} \\ 
      0 & \ddots & \ddots & \ddots & \vdots \\ 
      \vdots & \ddots & -1 & x_{n-1} & q_{n-1 \, n} \\
      0 & \cdots & 0 & -1 & x_n 
 \end{array}
 \right)
\end{align*}
and the $\q$-quantized elementary symmetric polynomial $E_i^{(n)}$ is defined to be the coefficient of $t^{n-i}$ for the characteristic polynomial of $M_n^{\q}$ multiplied by $(-1)^i$ for $1 \leq i \leq n$. 
Note that $E_i^{(n)}$ specializes to $E_i^{(n) \, \q_o}$ in setting $q_{ij} = 0$ whenever $j-i >1$ and $q_{i \, i+1} = q_i$ for each $1 \leq i \leq n-1$. 
One may generalize Fomin--Gelfand--Postnikov's $\q_o$-quantization straightforwardly (see Section~\ref{sect:quantization}). 
We call the generalization the $\q$-quantization. 
In order to construct an $S_n$-action on the polynomial ring $\Z[x_1,\ldots,x_n,q_{ij} \mid 1 \leq i < j \leq n]$, we connect it to the following polynomial ring
\begin{align*}
\Z[U]=\Z[z_{ij} \mid 1 \leq j < i \leq n]
\end{align*}
with a grading by $\deg z_{ij} = 2(i-j)$ for $1 \leq j < i \leq n$.
Geometrically, it is the coordinate ring of the open set around the identity element in the flag variety $\Fl(\C^n)$, isomorphic to the lower unipotent group of the general linear group $\GL_n(\C)$, upon tensoring with $\C$. 
Consider the quotient ring
\begin{align*} 
Q_n = \Z[x_1,\ldots,x_n,q_{ij} \mid 1 \leq i < j \leq n]/(E_1^{(n)}, \ldots, E_n^{(n)}).
\end{align*}
Note that this gives an algebraic generalization of the presentation for the quantum cohomology ring of the flag variety $\Fl(\C^n)$ in \eqref{eq:Intro_quantum_cohomology_flag}. 
The author and Shirato proved in \cite{HorShi} that $\Z[U] \otimes_\Z \C$ is isomorphic to $Q_n \otimes_\Z \C$ which sends $z_{ij}$ to $E_{i-j}^{(n-j)}$. 
Since the correspondence depends on $n$, we construct an isomorphism between $\Z[U]$ and $Q_n$ whose correspondence is independent of $n$.
For this, we define an involution $\omega$ on $Q_n$ with the usual property $\omega(E_i^{(j)}) = H_i^{(n-j)}$ (Proposition~\ref{eq:involutionEandH}) where $H_i^{(n)}$ denotes the $\q$-quantization of the $i$-th complete symmetric polynomial $h_i^{(n)}$ in the variables $x_1,\ldots,x_n$. 
Our first main theorem is as follows.

\begin{theorem} \label{theorem:Intro_isomorphism}
The map 
\begin{align*} 
\varphi: \Z[U] \to Q_n; \ \ \ z_{ij} \mapsto H_{i-j}^{(j)}
\end{align*}
is an isomorphism as graded rings. 
\end{theorem}

The isomorphism $\varphi$ is a key to construct the $S_n$-action on $\Z[x_1,\ldots,x_n,q_{ij} \mid 1 \leq i < j \leq n]$.
Before we explain it, let us introduce two polynomials $\nu_{i,j}$ and $\xi_{i,j}$ on $\Z[U]$.
Let $z=(z_{ij})_{1 \leq i,j \leq n}$ be the lower unipotent matrix and we put 
\begin{align*}  
N = 
\begin{pmatrix}
0 & 1 & &   \\
     & \ddots & \ddots & \\
     &   & 0 & 1 \\
     &  & & 0 \\ 
\end{pmatrix} \ \textrm{and} \ 
S = 
\begin{pmatrix}
1 &  & &   \\
     & 2 & & \\
     &   & \ddots &  \\
     &  & & n \\ 
\end{pmatrix}. 
\end{align*}
For $1 \leq j < i \leq n$, we define polynomials $\nu_{i,j}$ and $\xi_{i,j}$ on $\Z[U]$ by
\begin{align*} 
\nu_{i,j} = (z^{-1} N z)_{ij} \ \textrm{and} \ \xi_{i,j} = (z^{-1} S z)_{ij}. 
\end{align*} 
These polynomials geometrically mean defining functions of regular nilpotent Hessenberg varieties and regular semisimple Hessenberg varieties, respectively. 
Interestingly, the polynomials $\nu_{i,j}$ correspond to the quantum parameters by \cite{HorShi}, while the polynomials $\xi_{i,j}$ correspond to the polynomials $f_{i,j}$ by \cite{Hor25}, as defined below.
For describing an explicit presentation of the rational cohomology rings for regular nilpotent Hessenberg varieties, the polynomials $f_{i,j} \ (i \geq j \geq 1)$ are introduced in \cite{AHHM} as follows: 
\begin{align*} 
f_{i,j} = \sum_{k=1}^{j} \left(\prod_{\ell=j+1}^i (x_k-x_\ell) \right) x_k.
\end{align*}
We write $F_{i,j}$ for the $\q$-quantization of the polynomial $f_{i,j}$.

\begin{theorem} \label{theorem:Intro_nu_xi}
For $1 \leq j < i \leq n$, we obtain
\begin{align*}
\varphi(\nu_{i,j}) &= (-1)^{i-j} q_{ji}; \\ 
\varphi(\xi_{i,j}) &= F_{i-1,j}. 
\end{align*}
\end{theorem}

We remark that the correspondences in Theorem~\ref{theorem:Intro_nu_xi} are essentially given by \cite{HorShi} and \cite{Hor25}, respectively. 
In particular, the coordinate ring of regular semisimple Hessenberg varieties is related to the complex cohomology ring of regular nilpotent Hessenberg varieties by \cite{Hor25}.
In this paper we also see that the isomorphism $\varphi$ maps the Pl\"{u}cker coordinates to the $q$-quantizations of the Schur polynomials.
To be more precise, for a Young diagram $\lambda$ with at most $k$ rows and at most $n-k$ columns, one can define the Schur polynomial $s_\lambda(x_1,\ldots,x_k)$.
We denote by $S_\lambda$ the $\q$-quantization of $s_\lambda(x_1,\ldots,x_k)$. 
On the other hand, for a sequence $\underbar{{\bf i}}=(1\leq i_1 < i_2 < \dots < i_k \leq n)$, the polynomial $\Plucker_{\, \underbar{{\bf i}}}$ on $\Z[U]$ is defined to be the determinant of the submatrix of the lower unipontent matrix $z=(z_{ij})_{1 \leq i,j \leq n}$ associated to row indices $\underbar{{\bf i}}$ and column indices $\{1,2,\ldots,k\}$.
Recall that there is a one-to-one correspondence between the set of Young diagrams $\lambda$ with at most $k$ rows and at most $n-k$ columns and the set of sequences $\underbar{{\bf i}}=(1\leq i_1 < i_2 < \dots < i_k \leq n)$ (see \eqref{eq:one-to-one correspondence PP II}). 

\begin{theorem} 
For a sequence $\underbar{{\bf i}}=(1\leq i_1 < i_2 < \dots < i_k \leq n)$, we denote by $\lambda$ the corresponding Young diagram. 
Then we have
\begin{align*}
\varphi(\Plucker_{\, \underbar{{\bf i}}}) = S_\lambda.
\end{align*}
\end{theorem}

\subsection{Construction of $S_n$-action}

Now, we construct an $S_n$-action on the polynomial ring $\Z[x_1,\ldots,x_n,q_{ij} \mid 1 \leq i < j \leq n]$.
To do that, we first use the isomorphism 
\begin{align*} 
\Z[U]/(\nu_{i,j} \mid 1 \leq j <i \leq n) \cong \Z[x_1,\ldots,x_n]/(e_1^{(n)},\ldots,e_n^{(n)}); \ \ \ z_{ij} \mapsto h_{i-j}^{(j)}
\end{align*}
obtained from the isomorphism $\varphi$ by forgetting the quantum parameters (Theorems~\ref{theorem:Intro_isomorphism} and \ref{theorem:Intro_nu_xi}).
Since there exist divided difference operators on $\Z[x_1,\ldots,x_n]/(e_1^{(n)},\ldots,e_n^{(n)})$, one can translate them on $\Z[U]/(\nu_{i,j} \mid 1 \leq j <i \leq n)$. 
We know an explicit formula for $\partial_k(h_{i-j}^{(j)})$ in $\Z[x_1,\ldots,x_n]/(e_1^{(n)},\ldots,e_n^{(n)})$, so one obtains an explicit formula for $\partial_k(z_{ij})$ in the quotient ring $\Z[U]/(\nu_{i,j} \mid 1 \leq j <i \leq n)$. 
Significantly, the divided difference operators on $\Z[U]/(\nu_{i,j} \mid 1 \leq j <i \leq n)$ defined above can be lifted on $\Z[U]$. 
We next use the isomorphism
\begin{align*} 
\varphi: \Z[U] \cong Q_n=\Z[x_1,\ldots,x_n,q_{ij} \mid 1 \leq i < j \leq n]/(E_1^{(n)},\ldots,E_n^{(n)}); \ \ \ z_{ij} \mapsto H_{i-j}^{(j)}
\end{align*}
in Theorem~\ref{theorem:Intro_isomorphism}.
Then we obtain the divided difference operators $\partial_k$ on $Q_n$ induced from those on $\Z[U]$ constructed above.
By Theorem~\ref{theorem:Intro_nu_xi} the isomorphism $\varphi$ maps $(-1)^{j-i}\nu_{j,i}$ to the quantum parameter $q_{ij}$ for $1 \leq i < j \leq n$.
Thus, a computation for $\partial_k(\nu_{j,i})$ in $\Z[U]$ yields an explicit formula for $\partial_k(q_{ij})$ in $Q_n$. 
Finally, we derive the definition of $s_k(q_{ij})$ in the polynomial ring $\Z[x_1,\ldots,x_n,q_{ij} \mid 1 \leq i < j \leq n]$ from the formula for $\partial_k(q_{ij})$. 
In fact, we can explicitly write the definition as 
\begin{align} \label{eq;Intro_Sn_action_quantum_parameters}
s_k(q_{ij}) = 
\begin{cases}
q_{ij}-q_{i\,j-1}(x_{j-1}-x_j) \ \ \ &\textrm{if} \ k = j-1; \\ 
q_{ij}+q_{i+1\,j}(x_i-x_{i+1}) \ \ \ &\textrm{if} \ k = i; \\
q_{ij} \ \ \ &\textrm{otherwise}  
\end{cases}
\end{align}
in the polynomial ring $\Z[x_1,\ldots,x_n,q_{ij} \mid 1 \leq i < j \leq n]$ for each $1 \leq k \leq n-1$.
Here, we take the convention that $q_{ii}=0$ for any $1 \leq i \leq n$.
In other words, we have $s_k(q_{i \, i+1})=q_{i \, i+1}$ for any $1 \leq k \leq n-1$. 
One can verify that this formula generates a well-defined action of $S_n$ on $\Z[x_1,\ldots,x_n,q_{ij} \mid 1 \leq i < j \leq n]$ (see Lemma~\ref{lemma:Sn_action_qij}). 

\subsection{Divided difference operators}

We define the divided difference operators $\partial_k \ (1 \leq k \leq n-1)$ on the polynomial ring $\Z[x_1,\ldots,x_n,q_{ij} \mid 1 \leq i < j \leq n]$ by 
\begin{align*} 
\partial_k(F) = \frac{F-s_k(F)}{x_k-x_{k+1}} \ \ \ \textrm{for} \ F \in \Z[x_1,\ldots,x_n,q_{ij} \mid 1 \leq i < j \leq n],
\end{align*}
where we use the definition \eqref{eq;Intro_Sn_action_quantum_parameters}.
It is natural to ask whether a $\q$-analogue of \eqref{eq:Intro_divided_difference_Schubert} straightforwardly holds or not, but one can easily verify that the identity $\partial_k (\SS_w^\q) = \SS_{ws_k}^\q$ for $w(k) > w(k+1)$ is \emph{not} true in general. 
For this purpose, we introduce the following operator $\eta_k$ for $k \in [n-1]$. 
For each $E_{i_1,\ldots,i_m} \coloneqq E_{i_1}^{(1)} \cdots E_{i_m}^{(m)} \ (0 \leq i_k \leq k)$, we define $\eta_k(E_{i_1,\ldots,i_m})$ as
\begin{align*}
\sum_{p=1}^{k-1} q_{k-p \, k} E_{i_1}^{(1)} \cdots E_{i_{k-2}}^{(k-2)} \left( \sum_{\ell \geq 0} \big( E_{i_{k-1}-\ell-1}^{(k-1)} E_{i_k+\ell-p-1}^{(k-1-p)} - E_{i_{k-1}-\ell-p-1}^{(k-1-p)} E_{i_k+\ell-1}^{(k-1)} \big) \right) E_{i_{k+1}}^{(k+1)} \cdots E_{i_m}^{(m)} 
\end{align*}
with the convention that $\eta_1(E_{i_1,\ldots,i_m}) = 0$.  
If we write the Schubert polynomial $\SS_w$ as $\SS_w = \sum_{i_1,\ldots,i_{n-1}} c_{i_1,\ldots,i_{n-1}} e_{i_1,\ldots,i_{n-1}}$ for some unique $c_{i_1,\ldots,i_{n-1}} \in \Z$, then the $\q$-quantum Schubert polynomial $\SS_w^\q$ is defined to be the presentation $\SS_w^\q = \sum_{i_1,\ldots,i_{n-1}} c_{i_1,\ldots,i_{n-1}} E_{i_1,\ldots,i_{n-1}}$. 
Note that $\SS_w^{\q}$ specializes to Fomin--Gelfand--Postnikov's quantum Schubert polynomial $\SS_w^{\q_o}$ by setting $q_{ij} = 0$ for $j-i >1$ and $q_{i \, i+1} = q_i$ for each $1 \leq i \leq n-1$. 
We define 
\begin{align*}
\eta_k(\SS_w^\q) = \sum_{i_1,\ldots,i_{n-1}} c_{i_1,\ldots,i_{n-1}} \eta_k(E_{i_1,\ldots,i_{n-1}}). 
\end{align*}
We remark that $\eta_k(\SS_w^\q) = 0$ in the classical limit $q_{ij}=0$ for all $1 \leq i < j \leq n$. 

\begin{theorem} \label{theorem:Intro_divided_difference_quantumSchubert}
In the setting above, we have
\begin{align*} 
\partial_k (\SS_w^\q) = \begin{cases}
\SS_{ws_k}^\q + \eta_k(\SS_w^\q) \ \ \ &\textrm{if} \ w(k) > w(k+1); \\
0 \ \ \ &\textrm{if} \ w(k) < w(k+1).
\end{cases}
\end{align*}
\end{theorem}

As a final remark, we note that a computation for $\eta_k(\SS_w^\q)$ is easy for $k \leq 2$.
From the computation we derive a formula of $\partial_k^{\q}$ such that 
\begin{align} \label{eq:Intro_quantum_divided_difference} 
\partial_k^\q (\SS_w^\q) = \begin{cases}
\SS_{ws_k}^\q \ \ \ &\textrm{if} \ w(k) > w(k+1); \\
0 \ \ \ &\textrm{if} \ w(k) < w(k+1).
\end{cases}
\end{align}
for $k \leq 2$.
If $k \geq 3$, then a computation for $\eta_k(\SS_w^\q)$ is more complicated. 
If there exists $\partial_k^\q$ satisfying \eqref{eq:Intro_quantum_divided_difference} for $k \geq 3$, then it would be interesting to find its formula. 

The paper is organized as follows. 
After reviewing the definition and some properties for Schur polynomials in Section~\ref{sect:symmetric polynomials}, we explain the $\q$-quantization in Section~\ref{sect:quantization}. 
An involution on the polynomial ring $\Z[x_1,\ldots,x_n, q_{ij} \mid 1 \leq i < j \leq n]$ is introduced in Section~\ref{sect:involution}, which is necessary for the proof of our first main theorem. 
In Section~\ref{sect:polynomials fij} we quickly recount some background and property for the polynomials $f_{i,j}$ and their $\q$-quantizations $F_{i,j}$. 
We give an isomorphism between $\Z[U]$ and $Q_n$ in Section~\ref{sect:coordinate rings} as Theorem~\ref{theorem:iso}, and then we see that the Pl\"{u}cker coordinates correspond to the $q$-quantizations of the Schur polynomials in Section~\ref{sect:quantum Schur polynomials} as Corollary~\ref{corollary:pi_Slambda}.
Next, turning our attention to an action of the symmetric group $S_n$, we discuss $S_n$-actions on $\Z[U]$ and $\Z[x_1,\ldots,x_n, q_{ij} \mid 1 \leq i < j \leq n]$ in Sections~\ref{sect:a symmetric group action on ZU} and \ref{sect:a symmetric group action on qij}, respectively. 
Finally, we study a relation between the divided difference operators and $\q$-quantum Schubert polynomials on $\Z[x_1,\ldots,x_n, q_{ij} \mid 1 \leq i < j \leq n]$ in Section~\ref{sect:quantum Schubert polynomials} as Theorem~\ref{theorem:divided_difference_quantumSchubert}.

\bigskip
\noindent \textbf{Acknowledgements.} 
The author is supported in part by JSPS KAKENHI Grant-in-Aid for Early-Career Scientists: 23K12981.

\section{Symmetric polynomials} \label{sect:symmetric polynomials}

We review standard definitions from symmetric polynomials, such as elementary symmetric polynomials, complete symmetric polynomials, and Schur polynomials. 
We refer the reader to \cite[Chapter~6]{Ful97}.
We fix a positive integer $n$ and we use the following notation
\begin{align*}
[n]=\{1,2,\ldots,n\}
\end{align*}
throughout this paper.

The $i$-th \emph{elementary symmetric polynomial} $e_i^{(n)}$ in the variables $x_1,\ldots,x_n$ is defined by 
\begin{align*}
e_i^{(n)} =e_i(x_1,\ldots,x_n) = \sum_{1 \leq k_1 < \dots < k_i \leq n} x_{k_1} \cdots x_{k_i}. 
\end{align*} 
Here we take the convention that $e_0^{(n)}=1$ for all $n \geq 0$, and $e_i^{(n)}=0$ unless $0 \leq i \leq n$. 
One can easily see the recursive formula
\begin{align} \label{eq:recursive elementary symmetric polynomials}
e_i^{(j)} = e_i^{(j-1)} + e_{i-1}^{(j-1)} x_j.  
\end{align}
The $i$-th complete symmetric polynomial $h_i^{(n)}$ in the variables $x_1,\ldots,x_n$ is defined by  
\begin{align*}
h_i^{(n)} = h_i(x_1,\ldots,x_n)= \sum_{1 \leq k_1 \leq \dots \leq k_i \leq n} x_{k_1} \cdots x_{k_i} 
\end{align*} 
with the convention that $h_0^{(n)}=1$ for any $n \geq 0$.
We also take $h_i^{(0)}=0$ for $i > 0$, and $h_i^{(n)}=0$ for $i<0$ and $n \geq 0$. 
Similarly, one has the recursive formula
\begin{align} \label{eq:recursive complete symmetric polynomials}
h_i^{(j)} = h_i^{(j-1)} + h_{i-1}^{(j)} x_j.  
\end{align}

A sequence of positive integers $\lambda=(\lambda_1,\lambda_2,\ldots,\lambda_\ell)$ is a \emph{Young diagram} if $\lambda_1 \geq \lambda_2 \geq \dots \geq \lambda_\ell$.
A Young diagram $\lambda=(\lambda_1,\lambda_2,\ldots,\lambda_\ell)$ is often regarded as a collection of boxes arranged in left-justified rows with a weakly deacreasing numbers of boxes in each row. 
The \emph{transpose} $\lambda^t$ of a Young diagram $\lambda$ is obtained by flipping the diagram $\lambda$ over its main diagonal (from upper left to lower right).
We write $\lambda=(d_1^{a_1},\ldots,d_p^{a_p})$ to denote the Young diagram $\lambda$ that has $a_i$ copies of the positive integer $d_i$ for each $1 \leq i \leq p$. 

\begin{example}
The transpose of a Young diagram $\lambda=(5,4^2,1)$ is $\lambda^t=(4,3^3,1)$ as shown in Figure~\ref{picture:Young_diagram}.

\begin{figure}[h]
\begin{center}
\begin{picture}(175,75)
\put(0,15){\framebox(15,15)}
\put(0,30){\framebox(15,15)}
\put(15,30){\framebox(15,15)}
\put(30,30){\framebox(15,15)}
\put(45,30){\framebox(15,15)}
\put(0,45){\framebox(15,15)}
\put(15,45){\framebox(15,15)}
\put(30,45){\framebox(15,15)}
\put(45,45){\framebox(15,15)}
\put(0,60){\framebox(15,15)}
\put(15,60){\framebox(15,15)}
\put(30,60){\framebox(15,15)}
\put(45,60){\framebox(15,15)}
\put(60,60){\framebox(15,15)}

\put(-25,40){$\lambda=$}

\put(130,0){\framebox(15,15)}
\put(130,15){\framebox(15,15)}
\put(145,15){\framebox(15,15)}
\put(160,15){\framebox(15,15)}
\put(130,30){\framebox(15,15)}
\put(145,30){\framebox(15,15)}
\put(160,30){\framebox(15,15)}
\put(130,45){\framebox(15,15)}
\put(145,45){\framebox(15,15)}
\put(160,45){\framebox(15,15)}
\put(130,60){\framebox(15,15)}
\put(145,60){\framebox(15,15)}
\put(160,60){\framebox(15,15)}
\put(175,60){\framebox(15,15)}

\put(100,40){$\lambda^t=$}
\end{picture}
\end{center}
\caption{The Young diagram $\lambda=(5,4^2,1)$ and its transpose $\lambda^t=(4,3^3,1)$.}
\label{picture:Young_diagram}
\end{figure}
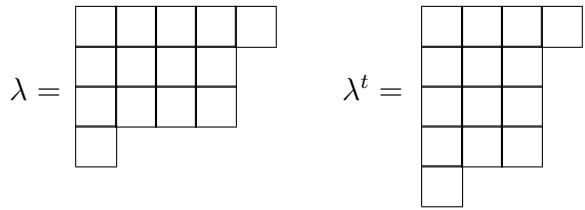
\end{example}

Given a Young diagram $\lambda=(\lambda_1 \geq \dots \geq \lambda_\ell > 0)$, it is convenient to allow more zeroes to occur at the end. 
Then we identify $\lambda$ with the sequences that differ only by such zeroes.
Fix a positive integer $m$. 
For a Young diagram $\lambda =(\lambda_1 \geq \lambda_2 \geq \dots \geq \lambda_m \geq 0)$, the \emph{Schur polynomial} $s_\lambda(x_1,\ldots,x_m)$ in the variables $x_1,\ldots,x_m$ associated to the Young diagram $\lambda$ is defined by
\begin{align*}
s_\lambda(x_1,\ldots,x_m)= \frac{\det\big(x_j^{\lambda_i+m-i} \big)_{1\leq i,j \leq m}}{\det\big(x_j^{m-i} \big)_{1\leq i,j \leq m}}.
\end{align*}
Note that we have the Schur polynomial $s_\lambda(x_1,\ldots,x_m)$ for a Young diagram $\lambda =(\lambda_1 \geq \lambda_2 \geq \dots \geq \lambda_k \geq 0)$ with $k \leq m$, i.e. $\lambda_{i} = 0$ for $k < i \leq m$.
If $\lambda =(\lambda_1 \geq \lambda_2 \geq \dots \geq \lambda_k \geq 0)$, then it is known that 
\begin{align} \label{eq:JacobiTrudi_h}
s_\lambda(x_1,\ldots,x_m) = \det\big( h_{\lambda_i+j-i}^{(m)} \big)_{1 \leq i,j \leq k}, 
\end{align}
which is called the Jacobi--Trudi identity. 
In particular, we have $s_{(p)}(x_1,\ldots,x_m)=h_p^{(m)}$.
Similarly, the following identity holds
\begin{align} \label{eq:JacobiTrudi_e}
s_{\lambda}(x_1,\ldots,x_m) = \det\big( e_{\mu_i+j-i}^{(m)} \big)_{1 \leq i,j \leq \ell}
\end{align}
where $\mu = (\mu_1 \geq \dots \geq \mu_\ell \geq 0)$ denotes the transpose of $\lambda$. 
In particular, one has $s_{(1^p)}(x_1,\ldots,x_m)=e_p^{(m)}$.

For $0 < k < n$, we write $\YY_k(n)$ for the set of Young diagrams $\lambda$ with at most $k$ rows and at most $n-k$ columns. 
Let $\binom{[n]}{k}$ be the set of sequences of positive integers $\underbar{{\bf i}}=(i_1,\ldots,i_k)$ such that $1 \leq i_1< i_2 < \cdots < i_k \leq n$.
Then there is a one-to-one correspondence between $\YY_k(n)$ and $\binom{[n]}{k}$ given by 
\begin{align} \label{eq:one-to-one correspondence PP II}
i_p-p = \lambda_{k+1-p} \ \ \ \textrm{for each} \ p \in [k].
\end{align}
The sequence $\underbar{{\bf i}}$ corresponding to a Young diagram $\lambda$ is pictorially expressed as follows. 
Draw a Young diagram $\lambda$ as shaded boxes on a square grid of size $k \times (n-k)$ by aligning the top-left corner. 
We label the vertical and horizontal steps of the lower border from $1$ to $n$ in order, starting from the bottom-left to the top-right.
Then the sequence $\underbar{{\bf i}}$ defined in \eqref{eq:one-to-one correspondence PP II} consists of the numbers labeled on the vertical line. 

\begin{example}
Let $n=12$ and $k=5$. 
Draw a Young diagram $\lambda=(5,4^2,1)$ on the square grid of size $5 \times 7$ as shown in Figure~\ref{picture:Young_diagram_on_the_square_grid}. 
Then, the corresponding sequence is $\underbar{{\bf i}}=(1,3,7,8,10)$. 

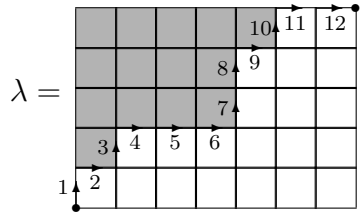
\begin{figure}[h]
\begin{center}
\begin{picture}(75,75)
\put(0,63){\colorbox{gray}}
\put(0,67){\colorbox{gray}}
\put(0,72){\colorbox{gray}}
\put(4,63){\colorbox{gray}}
\put(4,67){\colorbox{gray}}
\put(4,72){\colorbox{gray}}
\put(9,63){\colorbox{gray}}
\put(9,67){\colorbox{gray}}
\put(9,72){\colorbox{gray}}

\put(15,63){\colorbox{gray}}
\put(15,67){\colorbox{gray}}
\put(15,72){\colorbox{gray}}
\put(19,63){\colorbox{gray}}
\put(19,67){\colorbox{gray}}
\put(19,72){\colorbox{gray}}
\put(24,63){\colorbox{gray}}
\put(24,67){\colorbox{gray}}
\put(24,72){\colorbox{gray}}

\put(30,63){\colorbox{gray}}
\put(30,67){\colorbox{gray}}
\put(30,72){\colorbox{gray}}
\put(34,63){\colorbox{gray}}
\put(34,67){\colorbox{gray}}
\put(34,72){\colorbox{gray}}
\put(39,63){\colorbox{gray}}
\put(39,67){\colorbox{gray}}
\put(39,72){\colorbox{gray}}

\put(45,63){\colorbox{gray}}
\put(45,67){\colorbox{gray}}
\put(45,72){\colorbox{gray}}
\put(49,63){\colorbox{gray}}
\put(49,67){\colorbox{gray}}
\put(49,72){\colorbox{gray}}
\put(54,63){\colorbox{gray}}
\put(54,67){\colorbox{gray}}
\put(54,72){\colorbox{gray}}

\put(60,63){\colorbox{gray}}
\put(60,67){\colorbox{gray}}
\put(60,72){\colorbox{gray}}
\put(64,63){\colorbox{gray}}
\put(64,67){\colorbox{gray}}
\put(64,72){\colorbox{gray}}
\put(69,63){\colorbox{gray}}
\put(69,67){\colorbox{gray}}
\put(69,72){\colorbox{gray}}

\put(0,48){\colorbox{gray}}
\put(0,52){\colorbox{gray}}
\put(0,57){\colorbox{gray}}
\put(4,48){\colorbox{gray}}
\put(4,52){\colorbox{gray}}
\put(4,57){\colorbox{gray}}
\put(9,48){\colorbox{gray}}
\put(9,52){\colorbox{gray}}
\put(9,57){\colorbox{gray}}

\put(15,48){\colorbox{gray}}
\put(15,52){\colorbox{gray}}
\put(15,57){\colorbox{gray}}
\put(19,48){\colorbox{gray}}
\put(19,52){\colorbox{gray}}
\put(19,57){\colorbox{gray}}
\put(24,48){\colorbox{gray}}
\put(24,52){\colorbox{gray}}
\put(24,57){\colorbox{gray}}

\put(30,48){\colorbox{gray}}
\put(30,52){\colorbox{gray}}
\put(30,57){\colorbox{gray}}
\put(34,48){\colorbox{gray}}
\put(34,52){\colorbox{gray}}
\put(34,57){\colorbox{gray}}
\put(39,48){\colorbox{gray}}
\put(39,52){\colorbox{gray}}
\put(39,57){\colorbox{gray}}

\put(45,48){\colorbox{gray}}
\put(45,52){\colorbox{gray}}
\put(45,57){\colorbox{gray}}
\put(49,48){\colorbox{gray}}
\put(49,52){\colorbox{gray}}
\put(49,57){\colorbox{gray}}
\put(54,48){\colorbox{gray}}
\put(54,52){\colorbox{gray}}
\put(54,57){\colorbox{gray}}


\put(0,33){\colorbox{gray}}
\put(0,37){\colorbox{gray}}
\put(0,42){\colorbox{gray}}
\put(4,33){\colorbox{gray}}
\put(4,37){\colorbox{gray}}
\put(4,42){\colorbox{gray}}
\put(9,33){\colorbox{gray}}
\put(9,37){\colorbox{gray}}
\put(9,42){\colorbox{gray}}

\put(15,33){\colorbox{gray}}
\put(15,37){\colorbox{gray}}
\put(15,42){\colorbox{gray}}
\put(19,33){\colorbox{gray}}
\put(19,37){\colorbox{gray}}
\put(19,42){\colorbox{gray}}
\put(24,33){\colorbox{gray}}
\put(24,37){\colorbox{gray}}
\put(24,42){\colorbox{gray}}

\put(30,33){\colorbox{gray}}
\put(30,37){\colorbox{gray}}
\put(30,42){\colorbox{gray}}
\put(34,33){\colorbox{gray}}
\put(34,37){\colorbox{gray}}
\put(34,42){\colorbox{gray}}
\put(39,33){\colorbox{gray}}
\put(39,37){\colorbox{gray}}
\put(39,42){\colorbox{gray}}

\put(45,33){\colorbox{gray}}
\put(45,37){\colorbox{gray}}
\put(45,42){\colorbox{gray}}
\put(49,33){\colorbox{gray}}
\put(49,37){\colorbox{gray}}
\put(49,42){\colorbox{gray}}
\put(54,33){\colorbox{gray}}
\put(54,37){\colorbox{gray}}
\put(54,42){\colorbox{gray}}


\put(0,18){\colorbox{gray}}
\put(0,22){\colorbox{gray}}
\put(0,27){\colorbox{gray}}
\put(4,18){\colorbox{gray}}
\put(4,22){\colorbox{gray}}
\put(4,27){\colorbox{gray}}
\put(9,18){\colorbox{gray}}
\put(9,22){\colorbox{gray}}
\put(9,27){\colorbox{gray}}

\put(0,0){\framebox(15,15)}
\put(15,0){\framebox(15,15)}
\put(30,0){\framebox(15,15)}
\put(45,0){\framebox(15,15)}
\put(60,0){\framebox(15,15)}
\put(75,0){\framebox(15,15)}
\put(90,0){\framebox(15,15)}
\put(0,15){\framebox(15,15)}
\put(15,15){\framebox(15,15)}
\put(30,15){\framebox(15,15)}
\put(45,15){\framebox(15,15)}
\put(60,15){\framebox(15,15)}
\put(75,15){\framebox(15,15)}
\put(90,15){\framebox(15,15)}
\put(0,30){\framebox(15,15)}
\put(15,30){\framebox(15,15)}
\put(30,30){\framebox(15,15)}
\put(45,30){\framebox(15,15)}
\put(60,30){\framebox(15,15)}
\put(75,30){\framebox(15,15)}
\put(90,30){\framebox(15,15)}
\put(0,45){\framebox(15,15)}
\put(15,45){\framebox(15,15)}
\put(30,45){\framebox(15,15)}
\put(45,45){\framebox(15,15)}
\put(60,45){\framebox(15,15)}
\put(75,45){\framebox(15,15)}
\put(90,45){\framebox(15,15)}
\put(0,60){\framebox(15,15)}
\put(15,60){\framebox(15,15)}
\put(30,60){\framebox(15,15)}
\put(45,60){\framebox(15,15)}
\put(60,60){\framebox(15,15)}
\put(75,60){\framebox(15,15)}
\put(90,60){\framebox(15,15)}


\put(-25,40){$\lambda=$}

\put(0,0){\vector(0,1){10}}
\put(0,15){\vector(1,0){10}}
\put(15,15){\vector(0,1){10}}
\put(15,30){\vector(1,0){10}}
\put(30,30){\vector(1,0){10}}
\put(45,30){\vector(1,0){10}}
\put(60,30){\vector(0,1){10}}
\put(60,45){\vector(0,1){10}}
\put(60,60){\vector(1,0){10}}
\put(75,60){\vector(0,1){10}}
\put(75,75){\vector(1,0){10}}
\put(90,75){\vector(1,0){10}}

\put(-7,5){\tiny{$1$}}
\put(5,7){\tiny{$2$}}
\put(8,20){\tiny{$3$}}
\put(20,22){\tiny{$4$}}
\put(35,22){\tiny{$5$}}
\put(50,22){\tiny{$6$}}
\put(53,35){\tiny{$7$}}
\put(53,50){\tiny{$8$}}
\put(65,52){\tiny{$9$}}
\put(65,65){\tiny{$10$}}
\put(78,67){\tiny{$11$}}
\put(93,67){\tiny{$12$}}

\put(0,0){\circle*{3}}
 
\put(105,75){\circle*{3}}

\end{picture}
\end{center}
\caption{The Young diagram $\lambda=(5,4^2,1)$ on the square grid of size $5 \times 7$.}
\label{picture:Young_diagram_on_the_square_grid}
\end{figure}
\end{example}

The following lemma is known from \cite[Proof of Theorem~3.1]{AkyAky}, but we shall give a proof for the reader's convenience.

\begin{lemma}
Let $\lambda \in \YY_k(n)$ and $\underbar{{\bf i}} \in \binom{[n]}{k}$ the associated sequence in \eqref{eq:one-to-one correspondence PP II}.
Let $\lambda^t \in \YY_{n-k}(n)$ denotes the transpose of the Young diagram $\lambda$. 
Then the following equalities hold: 
\begin{align}
s_\lambda(x_1,\ldots,x_k) &= \left|
 \begin{array}{@{\,}cccc@{\,}}
  h_{i_1-1}^{(1)} &  h_{i_1-2}^{(2)} & \cdots & h_{i_1-k}^{(k)}  \smallskip \\
  h_{i_2-1}^{(1)} &  h_{i_2-2}^{(2)} & \cdots & h_{i_2-k}^{(k)} \\
  \vdots &  \vdots &  & \vdots \\
  h_{i_k-1}^{(1)} &  h_{i_k-2}^{(2)} & \cdots & h_{i_k-k}^{(k)}  
  \end{array}
 \right|; \label{eq:Schur_h} \\
 s_{\lambda^t}(x_1,\ldots,x_{n-k}) &= \left|
 \begin{array}{@{\,}cccc@{\,}}
  e_{i_1-1}^{(n-1)} &  e_{i_1-2}^{(n-2)} & \cdots & e_{i_1-k}^{(n-k)}   \smallskip \\
  e_{i_2-1}^{(n-1)} &  e_{i_2-2}^{(n-2)} & \cdots & e_{i_2-k}^{(n-k)} \\
  \vdots &  \vdots &  & \vdots \\
  e_{i_k-1}^{(n-1)} &  e_{i_k-2}^{(n-2)} & \cdots & e_{i_k-k}^{(n-k)}  
  \end{array}
 \right|. \label{eq:Schur_e}
\end{align}
\end{lemma}

\begin{proof}
By \eqref{eq:recursive complete symmetric polynomials}, we have 
\begin{align} \label{eq:Schur_h_proof}
h_{i}^{(j)} = h_{i}^{(j+1)} - x_{j+1} h_{i-1}^{(j+1)}.
\end{align}
By using \eqref{eq:Schur_h_proof}, the right hand side of \eqref{eq:Schur_h} equals
\begin{align*}
\left|
 \begin{array}{@{\,}cccc@{\,}}
  h_{i_1-1}^{(2)} - x_2 h_{i_1-2}^{(2)} &  h_{i_1-2}^{(2)} & \cdots & h_{i_1-k}^{(k)}  \smallskip \\
  h_{i_2-1}^{(2)} - x_2 h_{i_2-2}^{(2)} &  h_{i_2-2}^{(2)} & \cdots & h_{i_2-k}^{(k)} \\
  \vdots &  \vdots &  & \vdots \\
  h_{i_k-1}^{(2)} - x_2 h_{i_k-2}^{(2)} &  h_{i_k-2}^{(2)} & \cdots & h_{i_k-k}^{(k)}  
  \end{array}
 \right| 
 = \left|
 \begin{array}{@{\,}ccccc@{\,}}
  h_{i_1-1}^{(2)} &  h_{i_1-2}^{(2)} & h_{i_1-3}^{(3)} & \cdots & h_{i_1-k}^{(k)}  \smallskip \\
  h_{i_2-1}^{(2)} &  h_{i_2-2}^{(2)} & h_{i_2-3}^{(3)} & \cdots & h_{i_2-k}^{(k)} \\
  \vdots &  \vdots & \vdots &  & \vdots \\
  h_{i_k-1}^{(2)} &  h_{i_k-2}^{(2)} & h_{i_3-3}^{(3)} & \cdots & h_{i_k-k}^{(k)}  
  \end{array}
 \right|. 
\end{align*}
Applying \eqref{eq:Schur_h_proof} again to the second column and the first column in order for the right hand side above, one obtains
\begin{align*}
\left|
 \begin{array}{@{\,}ccccc@{\,}}
  h_{i_1-1}^{(2)} &  h_{i_1-2}^{(2)} & h_{i_1-3}^{(3)} & \cdots & h_{i_1-k}^{(k)}  \smallskip \\
  h_{i_2-1}^{(2)} &  h_{i_2-2}^{(2)} & h_{i_2-3}^{(3)} & \cdots & h_{i_2-k}^{(k)} \\
  \vdots &  \vdots & \vdots &  & \vdots \\
  h_{i_k-1}^{(2)} &  h_{i_k-2}^{(2)} & h_{i_3-3}^{(3)} & \cdots & h_{i_k-k}^{(k)}  
  \end{array}
 \right| 
 &= \left|
 \begin{array}{@{\,}ccccc@{\,}}
  h_{i_1-1}^{(2)} &  h_{i_1-2}^{(3)} & h_{i_1-3}^{(3)} & \cdots & h_{i_1-k}^{(k)}  \smallskip \\
  h_{i_2-1}^{(2)} &  h_{i_2-2}^{(3)} & h_{i_2-3}^{(3)} & \cdots & h_{i_2-k}^{(k)} \\
  \vdots &  \vdots & \vdots &  & \vdots \\
  h_{i_k-1}^{(2)} &  h_{i_k-2}^{(3)} & h_{i_3-3}^{(3)} & \cdots & h_{i_k-k}^{(k)}  
  \end{array}
 \right| \\
&= \left|
 \begin{array}{@{\,}ccccc@{\,}}
  h_{i_1-1}^{(3)} &  h_{i_1-2}^{(3)} & h_{i_1-3}^{(3)} & \cdots & h_{i_1-k}^{(k)}  \smallskip \\
  h_{i_2-1}^{(3)} &  h_{i_2-2}^{(3)} & h_{i_2-3}^{(3)} & \cdots & h_{i_2-k}^{(k)} \\
  \vdots &  \vdots & \vdots &  & \vdots \\
  h_{i_k-1}^{(3)} &  h_{i_k-2}^{(3)} & h_{i_3-3}^{(3)} & \cdots & h_{i_k-k}^{(k)}  
  \end{array}
 \right|. 
\end{align*}
Iterating this procedure, the right hand side of \eqref{eq:Schur_h} is eventually equal to
\begin{align*}
 \left|
 \begin{array}{@{\,}cccc@{\,}}
  h_{i_1-1}^{(k)} &  h_{i_1-2}^{(k)} & \cdots & h_{i_1-k}^{(k)}  \smallskip \\
  h_{i_2-1}^{(k)} &  h_{i_2-2}^{(k)} & \cdots & h_{i_2-k}^{(k)} \\
  \vdots &  \vdots &  & \vdots \\
  h_{i_k-1}^{(k)} &  h_{i_k-2}^{(k)} & \cdots & h_{i_k-k}^{(k)}  
  \end{array}
 \right|. 
\end{align*}
By using the relation in \eqref{eq:one-to-one correspondence PP II}, this is written as 
\begin{align*}
 \left|
 \begin{array}{@{\,}cccc@{\,}}
  h_{\lambda_k}^{(k)} &  h_{\lambda_k-1}^{(k)} & \cdots & h_{\lambda_k-(k-1)}^{(k)}  \smallskip \\
  h_{\lambda_{k-1}+1}^{(k)} &  h_{\lambda_{k-1}}^{(k)} & \cdots & h_{\lambda_{k-1}-(k-2)}^{(k)} \\
  \vdots &  \vdots &  & \vdots \\
  h_{\lambda_1+(k-1)}^{(k)} &  h_{\lambda_1+(k-2)}^{(k)} & \cdots & h_{\lambda_1}^{(k)}  
  \end{array}
 \right| 
 = \left|
 \begin{array}{@{\,}cccc@{\,}}
  h_{\lambda_1}^{(k)} &  h_{\lambda_1+1}^{(k)} & \cdots & h_{\lambda_1+(k-1)}^{(k)}  \smallskip \\
  h_{\lambda_2-1}^{(k)} &  h_{\lambda_2}^{(k)} & \cdots & h_{\lambda_2+(k-2)}^{(k)} \\
  \vdots &  \vdots &  & \vdots \\
  h_{\lambda_k-(k-1)}^{(k)} &  h_{\lambda_k-(k-2)}^{(k)} & \cdots & h_{\lambda_k}^{(k)}  
  \end{array}
 \right| 
\end{align*}
where the right hand side is obtained from the left hand side by flipping vertically and horizontally.
Hence, we obtain \eqref{eq:Schur_h} from the Jacobi--Trudi identity \eqref{eq:JacobiTrudi_h}. 

Using \eqref{eq:recursive elementary symmetric polynomials}, the right hand side of \eqref{eq:Schur_e} is written as
\begin{align*}
\left|
 \begin{array}{@{\,}cccc@{\,}}
  e_{i_1-1}^{(n-k)} &  e_{i_1-2}^{(n-k)} & \cdots & e_{i_1-k}^{(n-k)}   \smallskip \\
  e_{i_2-1}^{(n-k)} &  e_{i_2-2}^{(n-k)} & \cdots & e_{i_2-k}^{(n-k)} \\
  \vdots &  \vdots &  & \vdots \\
  e_{i_k-1}^{(n-k)} &  e_{i_k-2}^{(n-k)} & \cdots & e_{i_k-k}^{(n-k)}  
  \end{array}
 \right| 
\end{align*}
by a similar argument above. 
This equals 
\begin{align*}
\left|
 \begin{array}{@{\,}cccc@{\,}}
  e_{\lambda_k}^{(n-k)} &  e_{\lambda_k-1}^{(n-k)} & \cdots & e_{\lambda_k-(k-1)}^{(n-k)}   \smallskip \\
  e_{\lambda_{k-1}+1}^{(n-k)} &  e_{\lambda_{k-1}}^{(n-k)} & \cdots & e_{\lambda_{k-1}-(k-2)}^{(n-k)} \\
  \vdots &  \vdots &  & \vdots \\
  e_{\lambda_1+k-1}^{(n-k)} &  e_{\lambda_1+k-2}^{(n-k)} & \cdots & e_{\lambda_1}^{(n-k)}  
  \end{array}
 \right| 
=\left|
 \begin{array}{@{\,}cccc@{\,}}
  e_{\lambda_1}^{(n-k)} &  e_{\lambda_1+1}^{(n-k)} & \cdots & e_{\lambda_1+(k-1)}^{(n-k)}   \smallskip \\
  e_{\lambda_2-1}^{(n-k)} &  e_{\lambda_2}^{(n-k)} & \cdots & e_{\lambda_2+(k-2)}^{(n-k)} \\
  \vdots &  \vdots &  & \vdots \\
  e_{\lambda_k-(k-1)}^{(n-k)} &  e_{\lambda_k-(k-2)}^{(n-k)} & \cdots & e_{\lambda_k}^{(n-k)}  
  \end{array}
 \right|. 
\end{align*}
Therefore, the equality \eqref{eq:Schur_e} follows from \eqref{eq:JacobiTrudi_e}.
\end{proof}

\section{Quantization} \label{sect:quantization}

Throughout this paper, we consider the polynomial ring 
\begin{align*}
\Z[x_1,\ldots,x_n,q_{ij} \mid 1 \leq i < j \leq n]
\end{align*} 
with a grading defined by 
\begin{align*} 
\deg x_i &= 2 \ \ \ \textrm{for} \ i \in [n]; \\
\deg q_{ij} &= 2(j-i+1) \ \ \ \textrm{for} \ 1 \leq i < j \leq n.
\end{align*} 
We call the variables $q_{ij}$ \emph{quantum parameters}. 
We denote the quantum parameters by $\q=(q_{ij} \mid 1 \leq i < j \leq n)$.
Consider the matrix 
\begin{align*}
M_n^{\q} = \left(
 \begin{array}{@{\,}ccccc@{\,}}
     x_1 & q_{12} & q_{13} & \cdots & q_{1n} \\
     -1 & x_2 & q_{23} & \cdots & q_{2n} \\ 
      0 & \ddots & \ddots & \ddots & \vdots \\ 
      \vdots & \ddots & -1 & x_{n-1} & q_{n-1 \, n} \\
      0 & \cdots & 0 & -1 & x_n 
 \end{array}
 \right)
\end{align*}
and the \emph{$\q$-quantized elementary symmetric polynomials} $E_1^{(n)}, \ldots, E_n^{(n)}$ in the polynomial ring $\Z[x_1,\ldots,x_n,q_{ij} \mid 1 \leq i < j \leq n]$ are defined by 
\begin{align*}
\det(t I_n - M_n^{\q}) = t^n - E_1^{(n)} t^{n-1} + E_2^{(n)} t^{n-2} + \cdots + (-1)^n E_n^{(n)},
\end{align*}
where $I_n$ is the identity matrix of order $n$. 
In other words, $E_i^{(n)}$ is the coefficient of $t^{n-i}$ for the characteristic polynomial of $M_n^{\q}$ multiplied by $(-1)^i$. 
In the classical limit $q_{ij}= 0$ for all $i <j$, the $\q$-quantized elementary symmetric polynomial $E_i^{(n)}$ specializes to the $i$-th elementary symmetric polynomial $e_i^{(n)}$.

\begin{remark}
Givental--Kim and Ciocan-Fontanine give in \cite{Font95, GivKim} an explicit presentation of the quantum cohomology ring of the (full) flag variety in $\C^n$ by using the quantized elementary symmetric polynomials, which are the specializations of $E_k^{(n)} \ (k \in [n])$ in setting $q_{ij} = 0$ whenever $j-i >1$ and $q_{i \, i+1} = q_i$ for each $i \in [n-1]$. 
The $\q$-quantized elementary symmetric polynomials with the quantum parameters $q_{ij}$ are introduced in \cite{HorShi} to relate a geometry of regular nilpotent Hessenberg varieties. 
\end{remark}

One can easily see the recursive formula (\cite[Lemma~4.8]{HorShi})
\begin{align} \label{eq:recursive quantized elementary symmetric polynomials}
E_i^{(j)} = E_i^{(j-1)} + E_{i-1}^{(j-1)} x_j + \sum_{k=1}^{i-1} E_{i-1-k}^{(j-1-k)} q_{j-k \, j}  \ \ \ \textrm{for} \ 1 \leq i \leq j. \ 
\end{align}
Here, we take the convention that $E_0^{(j)}=1$ for arbitrary $j \geq 0$, and $E_i^{(j)} = 0$ unless $0 \leq i \leq j$.
Note that $E_1^{(j)} =x_1+\cdots+x_j$ for $j \geq 1$.
It is straightforward to see an explicit formula for $E_i^{(j)}$ from the recursive formula \eqref{eq:recursive quantized elementary symmetric polynomials}. 
In fact, we define $\rho_I$ for a consecutive substring $I=[k,\ell]\coloneqq \{k, k+1, \ldots, \ell \}$ of $[j]$ as follows:
\begin{align*}
\rho_I = \begin{cases}
q_{k\ell} \ &\textrm{if} \ k<\ell; \\
x_k \ &\textrm{if} \ k=\ell \ \textrm{(namely $I= \{k \}$)}.
\end{cases}
\end{align*}
Then one can write 
\begin{align} \label{eq:Ein_explicit} 
E_i^{(j)} = \sum \rho_{I_1} \rho_{I_2} \cdots \rho_{I_r}
\end{align}
for $1 \leq i \leq j$ (\cite[Lemma~4.1]{Hor25}). 
Here, the sum runs over all consecutive substrings $I_1, I_2, \ldots, I_r$ of $[j]$ such that $\bigcap_{p=1}^r I_p = \emptyset$ and $\sum_{p=1}^r |I_p| = i$. 
Note that $E_i^{(j)}$ is a homogeneous polynomial of degree $2i$ (cf. \cite[Lemma~6.2]{HorShi}).

For $i_1,\ldots,i_m$ with $0 \leq i_k \leq k$, the \emph{standard elementary monomial} is defined by 
\begin{align*}
e_{i_1,\ldots,i_m} \coloneqq e_{i_1}^{(1)} \cdots e_{i_m}^{(m)}. 
\end{align*}
Similarly, we define a \emph{$\q$-quantum standard elementary monomial} by  
\begin{align*}
E_{i_1,\ldots,i_m} \coloneqq E_{i_1}^{(1)} \cdots E_{i_m}^{(m)}. 
\end{align*}
By \cite[Proposition~3.3]{FGP} the standard elementary monomials form an additive basis of the polynomial ring $\Z[x_1,x_2,\ldots]$ in infinitely many variables.
Thus, any polynomial $f \in \Z[x_1,x_2,\ldots]$ can be uniquely written as a linear combination of standard elementary monomials 
\begin{align*}
f = \sum c_{i_1,\ldots,i_m} e_{i_1,\ldots,i_m} \ \ \ (c_{i_1,\ldots,i_m} \in \Z).
\end{align*}
By a method of Fomin--Gelfand--Postnikov in \cite{FGP}, the \emph{$\q$-quantization} $F$ of the polynomial $f$ is defined by 
\begin{align*}
F = \sum c_{i_1,\ldots,i_m} E_{i_1,\ldots,i_m}.
\end{align*}
We call the $\q$-quantization $H_i^{(n)}$ of $h_i^{(n)}$ the \emph{$\q$-quantized complete symmetric polynomial}. 
Applying \eqref{eq:Schur_e} to the case $\lambda=(1^p)$, we have $\underbar{{\bf i}} =(1,\ldots,k-p,k-p+2,\ldots,k+1)$ and
\begin{align} \label{eq:h_linear_combination_standard_elementary}
 h_{p}^{(n-k)} = \left|
 \begin{array}{@{\,}cccccc@{\,}}
  e_1^{(n-k+p-1)} &  1 & 0 & 0 & \cdots & 0 \smallskip \\
  e_2^{(n-k+p-1)} &  e_1^{(n-k+p-2)} & 1 & 0 & \cdots & 0 \\
  \vdots &  \vdots & \ddots & \ddots & \ddots & \vdots \\
  e_{p-2}^{(n-k+p-1)} &  e_{p-3}^{(n-k+p-2)} & \cdots & e_1^{(n-k+2)} & 1 & 0 \smallskip \\  
  e_{p-1}^{(n-k+p-1)} &  e_{p-2}^{(n-k+p-2)} & \cdots & e_2^{(n-k+2)} & e_1^{(n-k+1)} & 1 \smallskip \\ 
  e_p^{(n-k+p-1)} &  e_{p-1}^{(n-k+p-2)} & \cdots & e_3^{(n-k+2)} & e_2^{(n-k+1)} & e_1^{(n-k)}  
  \end{array}
 \right|. 
\end{align}
Therefore, we obtain a determinant formula for the $\q$-quantized complete symmetric polynomial as
\begin{align} \label{eq:determinant_HijEij}
 H_i^{(j)} = \left|
 \begin{array}{@{\,}cccccc@{\,}}
  E_1^{(j+i-1)} &  1 & 0 & 0 & \cdots & 0 \smallskip \\
  E_2^{(j+i-1)} &  E_1^{(j+i-2)} & 1 & 0 & \cdots & 0 \\
  \vdots &  \vdots & \ddots & \ddots & \ddots & \vdots \\
  E_{i-2}^{(j+i-1)} & E_{i-3}^{(j+i-2)} & \cdots & E_1^{(j+2)} & 1 & 0 \smallskip \\  
  E_{i-1}^{(j+i-1)} & E_{i-2}^{(j+i-2)} & \cdots & E_2^{(j+2)} & E_1^{(j+1)} & 1 \smallskip \\ 
  E_i^{(j+i-1)} & E_{i-1}^{(j+i-2)} & \cdots & E_3^{(j+2)} & E_2^{(j+1)} & E_1^{(j)}  
  \end{array}
 \right| \ \ \ \textrm{for} \ i,j \geq 1.  
\end{align}
By the cofactor expansion along the first column, we obtain
\begin{align} \label{eq:relationEandH}
H_i^{(j)} = \sum_{k=1}^i (-1)^{k-1} E_k^{(j+i-1)} H_{i-k}^{(j)}. 
\end{align}
Equivalently, one has
\begin{align} \label{eq:relationEandH2}
\sum_{k=0}^i (-1)^{k-1} E_k^{(j+i-1)} H_{i-k}^{(j)} = 0. 
\end{align}
For a closed interval $[a,b] \ (1 \leq a \leq b)$, we define $E_i^{[a,b]}$ by
\begin{align} \label{eq:Ei[a,b]}
E_i^{[a,b]}  = \sum \rho_{I_1} \rho_{I_2} \cdots \rho_{I_r} \ \ \ \textrm{for} \ 1 \leq i \leq b-a+1
\end{align}
where the sum runs over all consecutive substrings $I_1, I_2, \ldots, I_r$ of $[a,b]$ such that $\bigcap_{p=1}^r I_p = \emptyset$ and $\sum_{p=1}^r |I_p| = i$. 
Here, we take the convention that $E_0^{[a,b]}=1$ for $b-a+1 \geq 0$, and $E_i^{[a,b]}=0$ unless $0 \leq i \leq b-a+1$. 
Note that
\begin{align*}
E_i^{[1,j]} = E_i^{(j)}
\end{align*}
by \eqref{eq:Ein_explicit}. 

\begin{proposition} \label{proposition:decomposition_quantized_elementary_symmetric}
For $1 \leq i \leq j \leq n$, we have
\begin{align*} 
\sum_{k=0}^i E_k^{[1,n-j+i-1]} E_{i-k}^{[n-j+1,n]} = \sum_{\ell=0}^{i-1} E_\ell^{[n-j+1,n-j+i-1]} E_{i-\ell}^{[1,n]}
\end{align*}
in the polynomial ring $\Z[x_1,\ldots,x_n, q_{ij} \mid 1 \leq i < j \leq n]$. 
\end{proposition}

In order to prove Proposition~\ref{proposition:decomposition_quantized_elementary_symmetric}, we set 
\begin{align*}
\mathcal{C}_i^{[a,b]} = \{\mathcal{I}=\{I_1,\ldots,I_r\} \mid I_1, \ldots, I_r: \ &\textrm{consecutive substrings of} \ [a,b] \\ 
&\textrm{such that} \ \bigcap_{p=1}^r I_p = \emptyset \ \textrm{and} \ \sum_{p=1}^r |I_p| = i \} 
\end{align*}
for $0 \leq i \leq b-a+1$.
Note that we take $\mathcal{C}_0^{[a,b]} = \big\{ \{ \emptyset \} \big\}$ for $b-a+1 \geq 0$.
By \eqref{eq:Ei[a,b]}, we have 
\begin{align} \label{eq:Ei[a,b]Ci[a,b]}
E_i^{[a,b]}  = \sum_{\mathcal{I} \in \mathcal{C}_i^{[a,b]}} \rho_\mathcal{I} 
\end{align}
where $\rho_\mathcal{I} = \rho_{I_1} \rho_{I_2} \cdots \rho_{I_r}$ for $\mathcal{I}=\{I_1,\ldots,I_r\}$ and we take $\rho_{\{ \emptyset \} }=1$.
Hence, it suffices to show the following lemma.

\begin{lemma} \label{lemma:one-to-one_Ck}
Let $1 \leq i \leq j \leq n$.
Then there is a one-to-one correspondence 
\begin{align} \label{eq:one-to-one_proof}
\coprod_{k=0}^{i} \{(\mathcal{I},\mathcal{J}) \in \mathcal{C}_k^{[1,n-j+i-1]} \times \mathcal{C}_{i-k}^{[n-j+1,n]} \} \xrightarrow{1:1} \coprod_{\ell=0}^{i-1} \{(\mathcal{K},\mathcal{L}) \in \mathcal{C}_\ell^{[n-j+1,n-j+i-1]} \times \mathcal{C}_{i-\ell}^{[1,n]} \}
\end{align}
such that $\mathcal{I} \cup \mathcal{J} = \mathcal{K} \cup \mathcal{L}$ as multisets.
\end{lemma}

We first outline a proof of Lemma~\ref{lemma:one-to-one_Ck} below. 
Take a pair $(\mathcal{I},\mathcal{J}) \in \mathcal{C}_k^{[1,n-j+i-1]} \times \mathcal{C}_{i-k}^{[n-j+1,n]}$. 
Here, $\mathcal{I}$ consists of disjoint consecutive substrings of $[1,n-j+i-1]$ with total cardinality $k$.
Similarly, $\mathcal{J}$ means disjoint consecutive substrings of $[n-j+1,n]$ with total cardinality $i-k$.
See Figure~\ref{picture: a pair (I,J)}. 
By assigning consecutive substrings in $\mathcal{I}$ and $\mathcal{J}$ to elements of $\mathcal{K}$ and $\mathcal{L}$ properly, we construct a pair $(\mathcal{K},\mathcal{L}) \in \mathcal{C}_\ell^{[n-j+1,n-j+i-1]} \times \mathcal{C}_{i-\ell}^{[1,n]}$ for some $0 \leq \ell \leq i-1$. 
Before proving Lemma~\ref{lemma:one-to-one_Ck}, we explain the idea by an example below.

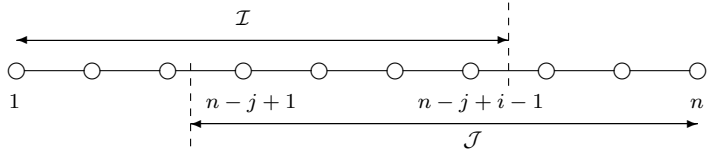
\begin{figure}[h]
\setlength{\unitlength}{1mm}
\begin{center} 
  \begin{picture}(100, 20)
  \put(0,10){\circle{2}}
  \put(1,10){\line(1,0){8}}
  \put(10,10){\circle{2}}
  \put(11,10){\line(1,0){8}}
  \put(20,10){\circle{2}}
  \put(21,10){\line(1,0){8}}
  \put(30,10){\circle{2}}
  \put(31,10){\line(1,0){8}}
  \put(40,10){\circle{2}}
  \put(41,10){\line(1,0){8}}
  \put(50,10){\circle{2}}
  \put(51,10){\line(1,0){8}}
  \put(60,10){\circle{2}}
  \put(61,10){\line(1,0){8}}
  \put(70,10){\circle{2}}
  \put(71,10){\line(1,0){8}}
  \put(80,10){\circle{2}}
  \put(81,10){\line(1,0){8}}
  \put(90,10){\circle{2}}
  \put(-1,5){\tiny{$1$}}
  \put(25,5){\tiny{$n-j+1$}}
  \put(53,5){\tiny{$n-j+i-1$}}
  \put(89,5){\tiny{$n$}}
  \put(23,11){\line(0,-1){1}}
  \put(23,9){\line(0,-1){1}}
  \put(23,7){\line(0,-1){1}}
  \put(23,5){\line(0,-1){1}}
  \put(23,3){\line(0,-1){1}}
  \put(23,1){\line(0,-1){1}}
  \put(65,8){\line(0,1){1}}
  \put(65,10){\line(0,1){1}}
  \put(65,12){\line(0,1){1}}
  \put(65,14){\line(0,1){1}}
  \put(65,16){\line(0,1){1}}
  \put(65,18){\line(0,1){1}}
  \put(0,14){\vector(1,0){65}}
  \put(65,14){\vector(-1,0){65}}
  \put(23,3){\vector(1,0){67}}
  \put(90,3){\vector(-1,0){67}}
  \put(29,16){\tiny{$\mathcal{I}$}}
  \put(59,0){\tiny{$\mathcal{J}$}}
  \end{picture}
\end{center}  
\caption{A pair $(\mathcal{I},\mathcal{J}) \in \mathcal{C}_k^{[1,n-j+i-1]} \times \mathcal{C}_{i-k}^{[n-j+1,n]}$.}
\label{picture: a pair (I,J)}
\end{figure}

\begin{example}
Let $i=j=9$ and $n=10$. 
Consider a pair consisting of 
\begin{align*}
\mathcal{I} = \big\{ \{4,5\}, \{8,9 \} \big\} \in \mathcal{C}_4^{[1,9]} \ \textrm{and} \ 
\mathcal{J} = \big\{ \{4\}, \{6\}, \{8\}, \{9,10 \} \big\} \in \mathcal{C}_5^{[2,10]},
\end{align*}
as shown in Figure~\ref{picture: a pair (I,J) example}.
Then we shall consider which pair $(\mathcal{K},\mathcal{L}) \in \mathcal{C}_\ell^{[2,9]} \times \mathcal{C}_{9-\ell}^{[1,10]}$ should correspond to $(\mathcal{I},\mathcal{J})$. 
We first note that $\{9,10 \}$ should be an element of $\mathcal{L}$. 
Since $\{8,9 \}$ intersects with $\{9,10 \}$, we must assign $\{8,9 \}$ to an element of $\mathcal{K}$.
Similarly, we need to assign $\{8 \}$ to an element of $\mathcal{L}$ since $\{8,9 \}$ and $\{8 \}$ overlap.
Hence, assignments of $\{9,10 \}, \{8,9 \}, \{8 \}$ are automatically determined in order. 
For the remaining consecutive substrings $I=\{4 \}, \{4,5 \}, \{6 \}$, we assign $I$ to an element $\mathcal{L}$ (resp. $\mathcal{K}$) if $I \in \mathcal{I}$ (resp. $I \in \mathcal{J}$). 
In summary, we have 
\begin{align*}
\mathcal{K} = \big\{ \{4\}, \{6\}, \{8,9 \} \big\} \in \mathcal{C}_4^{[2,9]} \ \textrm{and} \ 
\mathcal{L} = \big\{ \{4,5\}, \{8\}, \{9,10 \} \big\} \in \mathcal{C}_5^{[1,10]}.
\end{align*}

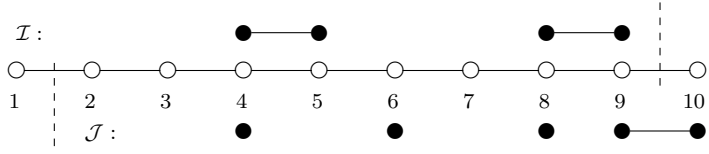
\begin{figure}[h]
\setlength{\unitlength}{1mm}
\begin{center} 
  \begin{picture}(100, 20)
  \put(0,10){\circle{2}}
  \put(1,10){\line(1,0){8}}
  \put(10,10){\circle{2}}
  \put(11,10){\line(1,0){8}}
  \put(20,10){\circle{2}}
  \put(21,10){\line(1,0){8}}
  \put(30,10){\circle{2}}
  \put(31,10){\line(1,0){8}}
  \put(40,10){\circle{2}}
  \put(41,10){\line(1,0){8}}
  \put(50,10){\circle{2}}
  \put(51,10){\line(1,0){8}}
  \put(60,10){\circle{2}}
  \put(61,10){\line(1,0){8}}
  \put(70,10){\circle{2}}
  \put(71,10){\line(1,0){8}}
  \put(80,10){\circle{2}}
  \put(81,10){\line(1,0){8}}
  \put(90,10){\circle{2}}
  \put(-1,5){\tiny{$1$}}
  \put(9,5){\tiny{$2$}}
  \put(19,5){\tiny{$3$}}
  \put(29,5){\tiny{$4$}}
  \put(39,5){\tiny{$5$}}
  \put(49,5){\tiny{$6$}}
  \put(59,5){\tiny{$7$}}
  \put(69,5){\tiny{$8$}}
  \put(79,5){\tiny{$9$}}
  \put(88,5){\tiny{$10$}}
  \put(5,11){\line(0,-1){1}}
  \put(5,9){\line(0,-1){1}}
  \put(5,7){\line(0,-1){1}}
  \put(5,5){\line(0,-1){1}}
  \put(5,3){\line(0,-1){1}}
  \put(5,1){\line(0,-1){1}}
  \put(85,8){\line(0,1){1}}
  \put(85,10){\line(0,1){1}}
  \put(85,12){\line(0,1){1}}
  \put(85,14){\line(0,1){1}}
  \put(85,16){\line(0,1){1}}
  \put(85,18){\line(0,1){1}}
  \put(0,14){\tiny{$\mathcal{I}:$}}
  \put(30,15){\circle*{2}}
  \put(31,15){\line(1,0){8}}
  \put(40,15){\circle*{2}}
  \put(70,15){\circle*{2}}
  \put(71,15){\line(1,0){8}}
  \put(80,15){\circle*{2}}
  \put(9,1){\tiny{$\mathcal{J}:$}}
  \put(30,2){\circle*{2}}
  \put(50,2){\circle*{2}}
  \put(70,2){\circle*{2}}
  \put(80,2){\circle*{2}}
  \put(81,2){\line(1,0){8}}
  \put(90,2){\circle*{2}}
  \end{picture}
\end{center}  
\caption{$\mathcal{I} = \big\{ \{4,5\}, \{8,9 \} \big\}$ and $\mathcal{J} = \big\{ \{4\}, \{6\}, \{8\}, \{9,10 \} \big\} $.}
\label{picture: a pair (I,J) example}
\end{figure}
For other examples, one has
\begin{align*}
\textrm{(a)} \ & \mathcal{I} = \big\{ \{1,2\}, \{8,9 \} \big\}, \ 
\mathcal{J} = \big\{ \{2,3\}, \{8,9,10 \} \big\} \\ 
& \hspace{10pt} \Rightarrow \mathcal{K} = \big\{ \{2,3\}, \{8,9 \} \big\}, \ 
\mathcal{L} = \big\{ \{1,2\}, \{8,9,10 \} \big\}; \\
\textrm{(b)} \ & \mathcal{I} = \big\{ \{2,3\}, \{4\}, \{8 \} \big\}, \ 
\mathcal{J} = \big\{ \{3,4\}, \{7,8,9 \} \big\} \\ 
& \hspace{10pt} \Rightarrow \mathcal{K} = \big\{ \{3,4\}, \{7,8,9 \} \big\}, \ 
\mathcal{L} = \big\{ \{2,3\}, \{4\}, \{8 \} \big\}. 
\end{align*}
In the case $\textrm{(a)}$, the pair $(\mathcal{K},\mathcal{L})$ is automatically determined by $(\mathcal{I},\mathcal{J})$ since $\{8,9,10 \}$ and $\{1,2\}$ do not belong to the closed interval $\{2,3,\ldots,9\}$, and each of $\{8,9 \}$ and $\{2,3 \}$ intersects with $\{8,9,10 \}$ and $\{1,2\}$, respectively. 
In the case $\textrm{(b)}$, all consecutive substrings $I$ in $\mathcal{I}$ and $\mathcal{J}$ belong to the closed interval $\{2,3,\ldots,9\}$, so we assign $I$ to an element $\mathcal{L}$ (resp. $\mathcal{K}$) if $I \in \mathcal{I}$ (resp. $I \in \mathcal{J}$). 
\end{example}

\begin{proof}[Proof of Lemma~\ref{lemma:one-to-one_Ck}]
Fix $0 \leq k \leq i$.
To each $(\mathcal{I},\mathcal{J}) \in \mathcal{C}_k^{[1,n-j+i-1]} \times \mathcal{C}_{i-k}^{[n-j+1,n]}$, we assign a pair $(\mathcal{K},\mathcal{L}) \in \mathcal{C}_\ell^{[n-j+1,n-j+i-1]} \times \mathcal{C}_{i-\ell}^{[1,n]}$ for some $0 \leq \ell \leq i-1$ as follows. 
Set 
\begin{align*}
\mathcal{I}' &= \{I \in \mathcal{I} \mid I \not\subset [n-j+1,n-j+i-1] \}; \\
\mathcal{I}'' &= \{I \in \mathcal{I} \mid I \subset [n-j+1,n-j+i-1] \}; \\
\mathcal{J}' &= \{J \in \mathcal{J} \mid J \not\subset [n-j+1,n-j+i-1] \}; \\
\mathcal{J}'' &= \{J \in \mathcal{J} \mid J \subset [n-j+1,n-j+i-1] \}. 
\end{align*}
We first note that 
\begin{align} \label{eq:proof_intersection_empty1}
I \in \mathcal{I}', \ J \in \mathcal{J}' \Rightarrow I \cap J = \emptyset.
\end{align}
In fact, if $I \cap J \neq \emptyset$, then $I \cup J \supset [n-j,n-j+i]$ since $I$ and $J$ are consecutive. 
Then we have $i = k+(i-k) \geq |I \cup J| \geq i+1$, which is a contradiction.

\noindent
\textbf{Step~1:} We assign all consective substrings in $\mathcal{I}'$ and $\mathcal{J}'$ to an element of $\mathcal{L}$. 

Without loss of generality, for an element $\mathcal{I}=\{I_1,\ldots,I_r \} \in \mathcal{C}_i^{[a,b]}$, we may assume that the maximal number in $I_p$ is less than the minimal number in $I_{p+1}$ for each $p \in [r-1]$.
Then we put $\mbox{left}(\mathcal{I}) \coloneqq I_1$ and $\mbox{right}(\mathcal{I}) \coloneqq I_r$. 

\noindent
\textbf{Step~2:} (i) Set $\mathcal{I}''^{(\rr)}_1 = \{I \in \mathcal{I}'' \mid I \cap \big( \mbox{left}(\mathcal{J}') \big) \neq \emptyset \}$. 
If $\mathcal{I}''^{(\rr)}_1$ is nonempty, then we add all consective substrings in $\mathcal{I}''^{(\rr)}_1$ to an element of $\mathcal{K}$. 
Next, we define $\mathcal{J}''^{(\rr)}_1 = \{J \in \mathcal{J}'' \mid J \cap \big( \mbox{left}(\mathcal{I}''^{(\rr)}_1) \big) \neq \emptyset \}$. 
If $\mathcal{J}''^{(\rr)}_1$ is nonempty, then we add all consective substrings in $\mathcal{J}''^{(\rr)}_1$ to an element of $\mathcal{L}$. 
Repeat this process until $\mathcal{I}''^{(\rr)}_p$ or $\mathcal{J}''^{(\rr)}_p$ is empty for some positive integer $p$ where 
\begin{align*}
\mathcal{I}''^{(\rr)}_p = \{I \in \mathcal{I}'' \mid I \cap \big(\mbox{left}(\mathcal{J}'^{(\rr)}_{p-1})\big) \neq \emptyset \} \ \textrm{and} \ \mathcal{J}''^{(\rr)}_p = \{J \in \mathcal{J}'' \mid J \cap \big(\mbox{left}(\mathcal{I}''^{(\rr)}_p)\big) \neq \emptyset \}.
\end{align*}

\noindent
(ii) Similarly, we put $\mathcal{J}''^{(\ll)}_1 = \{J \in \mathcal{J}'' \mid J \cap \big( \mbox{right}(\mathcal{I}') \big) \neq \emptyset \}$. 
If $\mathcal{J}''^{(\ll)}_1$ is nonempty, then we add all consective substrings in $\mathcal{J}''^{(\ll)}_1$ to an element  of $\mathcal{K}$. 
We next set $\mathcal{I}''^{(\ll)}_1 = \{I \in \mathcal{I}'' \mid I \cap \big( \mbox{right}(\mathcal{J}''^{(\ll)}_1) \big) \neq \emptyset \}$. 
If $\mathcal{I}''^{(\ll)}_1$ is nonempty, then we add all consective substrings in $\mathcal{I}''^{(\ll)}_1$ to an element of $\mathcal{L}$. 
Repeat this process until $\mathcal{J}''^{(\ll)}_p$ or $\mathcal{I}''^{(\ll)}_p$ is empty for some positive integer $p$ where  
\begin{align*}
\mathcal{J}''^{(\ll)}_p = \{J \in \mathcal{J}'' \mid J \cap \big(\mbox{right}(\mathcal{I}''^{(\ll)}_{p-1})\big) \neq \emptyset \} \ \textrm{and} \ \mathcal{I}''^{(\ll)}_p = \{I \in \mathcal{I}'' \mid I \cap \big(\mbox{right}(\mathcal{J}'^{(\ll)}_p)\big) \neq \emptyset \}. 
\end{align*}
The procedure~(i) determines assignments of some consective substrings in $\mathcal{I}''$ and $\mathcal{J}''$ from the right in $[n-j+1,n-j+i-1]$, while the procedure~(ii) determines them from the left in $[n-j+1,n-j+i-1]$.
We remark that these procedures are well-defined.
In other words, it can never happen that a consective substring $I \in \mathcal{I}''$ (resp. $\mathcal{J}''$) belongs to both $\mathcal{I}''^{(\rr)}_p$ and $\mathcal{I}''^{(\ll)}_{p'}$ (resp. $\mathcal{J}''^{(\rr)}_p$ and $\mathcal{J}''^{(\ll)}_{p'}$) for some $p$ and $p'$.
In fact, if it happens, then $\mbox{right}(\mathcal{I}') \cup (\cup_{I \in \mathcal{I}''} I) \cup (\cup_{J \in \mathcal{J}''} J) \cup \mbox{left}(\mathcal{J}')$ contains the closed interval $[n-j,n-j+i]$, which contradicts the equality $\sum_{I \in \mathcal{I}} |I| + \sum_{J \in \mathcal{J}} |J| = i$.
The discussion above implies that 
\begin{equation} \label{eq:proof_intersection_empty2}
 \begin{split}
I \in \mathcal{I}', \ J \in \mathcal{J}''^{(\rr)}_p &\Rightarrow I \cap J = \emptyset; \\
I \in \mathcal{I}''^{(\ll)}_p, \ J \in \mathcal{J}' &\Rightarrow I \cap J = \emptyset; \\
I \in \mathcal{I}''^{(\ll)}_p, \ J \in \mathcal{J}''^{(\rr)}_{p'} &\Rightarrow I \cap J = \emptyset; \\
I \in \mathcal{I}''^{(\rr)}_p, \ J \in \mathcal{J}''^{(\ll)}_{p'} &\Rightarrow I \cap J = \emptyset
 \end{split}
\end{equation}
for any $p$ and $p'$.

\noindent
\textbf{Step~3:}  We put
\begin{align*}
\mathcal{I}''' = \mathcal{I}'' \setminus \big( (\underset{p}{\amalg} \, \mathcal{I}''^{(\rr)}_p ) \amalg (\underset{p}{\amalg} \, \mathcal{I}''^{(\ll)}_p) \big) \ \textrm{and} \ \mathcal{J}''' = \mathcal{J}'' \setminus \big( (\underset{p}{\amalg} \, \mathcal{J}''^{(\rr)}_p ) \amalg (\underset{p}{\amalg} \, \mathcal{J}''^{(\ll)}_p) \big). 
\end{align*}
If $\mathcal{I}'''$ (resp. $\mathcal{J}'''$) is nonempty, then we add all consective substrings in $\mathcal{I}'''$ (resp. $\mathcal{J}'''$) to an element of $\mathcal{L}$ (resp. $\mathcal{K}$). 
By the definitions of $\mathcal{I}'''$ and $\mathcal{J}'''$, we have 
\begin{equation} \label{eq:proof_intersection_empty3}
 \begin{split}
I \in \mathcal{I}''', \ J \in \mathcal{J}' &\Rightarrow I \cap J = \emptyset; \\
I \in \mathcal{I}''', \ J \in \mathcal{J}''^{(\rr)}_p &\Rightarrow I \cap J = \emptyset; \\
I \in \mathcal{I}''^{(\rr)}_p, \ J \in \mathcal{J}''' &\Rightarrow I \cap J = \emptyset 
 \end{split}
\end{equation}
for any $p$. 

To summarize, we define  
\begin{align*}
\mathcal{K} = \big( (\underset{p}{\amalg} \, \mathcal{I}''^{(\rr)}_p ) \amalg (\underset{p}{\amalg} \, \mathcal{J}''^{(\ll)}_p) \big) \amalg \mathcal{J}''' \ \textrm{and} \ \mathcal{L} = \mathcal{I}' \amalg \mathcal{J}' \amalg \big( (\underset{p}{\amalg} \, \mathcal{J}''^{(\rr)}_p ) \amalg (\underset{p}{\amalg} \, \mathcal{I}''^{(\ll)}_p) \big) \amalg \mathcal{I}'''. 
\end{align*}
It follows from \eqref{eq:proof_intersection_empty1}, \eqref{eq:proof_intersection_empty2}, and  \eqref{eq:proof_intersection_empty3} that each of $\mathcal{K}$ and $\mathcal{L}$ consists of disjoint consective subsrings in $[n-j+1,n-j+i-1]$ and $[1,n]$, respectively. 
Hence, we obtain $(\mathcal{K},\mathcal{L}) \in \mathcal{C}_\ell^{[n-j+1,n-j+i-1]} \times \mathcal{C}_{i-\ell}^{[1,n]}$ for some $0 \leq \ell \leq i-1$. 

We next prove that one can recover $(\mathcal{I},\mathcal{J}) \in \mathcal{C}_k^{[1,n-j+i-1]} \times \mathcal{C}_{i-k}^{[n-j+1,n]}$ from $(\mathcal{K},\mathcal{L}) \in \mathcal{C}_\ell^{[n-j+1,n-j+i-1]} \times \mathcal{C}_{i-\ell}^{[1,n]}$.
Note that every $L \in \mathcal{L}$ does not include the closed interval $[n-j,n-j+i]$ since $|L| \leq i$.
We put 
\begin{align*}
\mathcal{L}^{(\ll)}_1 &= \{L \in \mathcal{L} \mid L \subset [1,n-j+i-1] \ \textrm{and} \ L \not\subset [n-j+1,n-j+i-1] \}; \\
\mathcal{L}^{(\rr)}_1 &= \{L \in \mathcal{L} \mid L \subset [n-j+1,n] \ \textrm{and} \ L \not\subset [n-j+1,n-j+i-1] \}; \\
\mathcal{L}' &= \mathcal{L} \setminus ( \mathcal{L}^{(\rr)}_1 \amalg \mathcal{L}^{(\ll)}_1).
\end{align*}

\noindent
\textbf{Step~1:} We add all consective substrings in $\mathcal{L}^{(\ll)}_1$ (resp. $\mathcal{L}^{(\rr)}_1$) to an element of $\mathcal{I}$ (resp. $\mathcal{J}$). 

\noindent
\textbf{Step~2:} (i) We set $\mathcal{K}^{(\ll)}_1 = \{K \in \mathcal{K} \mid K \cap \big( \mbox{right}(\mathcal{L}^{(\ll)}_1) \big) \neq \emptyset \}$. 
If $\mathcal{K}^{(\ll)}_1$ is nonempty, then we add all consective substrings in $\mathcal{K}^{(\ll)}_1$ to an element of $\mathcal{J}$. 
Next, we define $\mathcal{L}^{(\ll)}_2 = \{L \in \mathcal{L}' \mid L \cap \big( \mbox{right}(\mathcal{K}^{(\ll)}_1) \big) \neq \emptyset \}$. 
If $\mathcal{L}^{(\ll)}_2$ is nonempty, then we add all consective substrings in $\mathcal{L}^{(\ll)}_2$ to an element of $\mathcal{I}$. 
Repeat this process until $\mathcal{K}^{(\ll)}_p$ or $\mathcal{L}^{(\ll)}_{p+1}$ is empty for some positive integer $p$ where 
\begin{align*}
\mathcal{K}^{(\ll)}_p = \{K \in \mathcal{K} \mid K \cap \big(\mbox{right}(\mathcal{L}^{(\ll)}_p)\big) \neq \emptyset \} \ \textrm{and} \ \mathcal{L}^{(\ll)}_{p+1} = \{L \in \mathcal{L}' \mid L \cap \big(\mbox{right}(\mathcal{K}^{(\ll)}_p)\big) \neq \emptyset \}.
\end{align*}

\noindent
(ii) Similarly, we put $\mathcal{K}^{(\rr)}_1 = \{K \in \mathcal{K} \mid K \cap \big( \mbox{left}(\mathcal{L}^{(\rr)}_1) \big) \neq \emptyset \}$. 
If $\mathcal{K}^{(\rr)}_1$ is nonempty, then we add all consective substrings in $\mathcal{K}^{(\rr)}_1$ to an element of $\mathcal{I}$. 
We next put $\mathcal{L}^{(\rr)}_2 = \{L \in \mathcal{L}' \mid L \cap \big( \mbox{left}(\mathcal{K}^{(\rr)}_1) \big) \neq \emptyset \}$. 
If $\mathcal{L}^{(\rr)}_2$ is nonempty, then we add all consective substrings in $\mathcal{L}^{(\rr)}_2$ to an element of $\mathcal{J}$. 
Repeat this process until $\mathcal{K}^{(\rr)}_p$ or $\mathcal{L}^{(\rr)}_{p+1}$ is empty for some positive integer $p$ where  
\begin{align*}
\mathcal{K}^{(\rr)}_p = \{K \in \mathcal{K} \mid K \cap \big(\mbox{left}(\mathcal{L}^{(\rr)}_p)\big) \neq \emptyset \} \ \textrm{and} \ \mathcal{L}^{(\rr)}_{p+1} = \{L \in \mathcal{L}' \mid L \cap \big(\mbox{left}(\mathcal{K}^{(\rr)}_p)\big) \neq \emptyset \}.
\end{align*}

\noindent
\textbf{Step~3:}  Set
\begin{align*}
\mathcal{L}'' = \mathcal{L} \setminus \big( (\underset{p}{\amalg} \, \mathcal{L}^{(\rr)}_p ) \amalg (\underset{p}{\amalg} \, \mathcal{L}^{(\ll)}_p) \big) \ \textrm{and} \ \mathcal{K}' = \mathcal{K} \setminus \big( (\underset{p}{\amalg} \, \mathcal{K}^{(\rr)}_p ) \amalg (\underset{p}{\amalg} \, \mathcal{K}^{(\ll)}_p) \big). 
\end{align*}
If $\mathcal{L}''$ (resp. $\mathcal{K}'$) is nonempty, then we add all consective substrings in $\mathcal{L}''$ (resp. $\mathcal{K}'$) to an element of $\mathcal{I}$ (resp. $\mathcal{J}$). 

To summarize, we obtain 
\begin{align*}
\mathcal{I} = \big( (\underset{p}{\amalg} \, \mathcal{L}^{(\ll)}_p ) \amalg (\underset{p}{\amalg} \, \mathcal{K}^{(\rr)}_p) \big) \amalg \mathcal{L}'' \ \textrm{and} \ \mathcal{J} = \big( (\underset{p}{\amalg} \, \mathcal{L}^{(\rr)}_p ) \amalg (\underset{p}{\amalg} \, \mathcal{K}^{(\ll)}_p) \big) \amalg \mathcal{K}'. 
\end{align*}
Then the pair $(\mathcal{I},\mathcal{J})$ belongs to $\mathcal{C}_k^{[1,n-j+i-1]} \times \mathcal{C}_{i-k}^{[n-j+1,n]}$ for some $0 \leq k \leq i$. 
One can easily see that these correspondences are inverses of each other, and hence we obtain an one-to-one correspondence in \eqref{eq:one-to-one_proof}. 
\end{proof}

\begin{proof}[Proof of Proposition~\ref{proposition:decomposition_quantized_elementary_symmetric}]
By \eqref{eq:Ei[a,b]Ci[a,b]} it suffices to show that 
\begin{align*}
\sum_{k=0}^i \sum_{(\mathcal{I}, \mathcal{J}) \in \mathcal{C}_k^{[1,n-j+i-1]} \times \mathcal{C}_{i-k}^{[n-j+1,n]}} \rho_\mathcal{I} \rho_\mathcal{J} = \sum_{\ell=0}^{i-1} \sum_{(\mathcal{K}, \mathcal{L}) \in \mathcal{C}_\ell^{[n-j+1,n-j+i-1]} \times \mathcal{C}_{i-\ell}^{[1,n]}} \rho_\mathcal{K} \rho_\mathcal{L},
\end{align*}
which follows from Lemma~\ref{lemma:one-to-one_Ck}.
\end{proof}

\section{Involution} \label{sect:involution}

We define an involution $\omega$ on the polynomial ring $\Z[x_1,\ldots,x_n, q_{ij} \mid 1 \leq i < j \leq n]$ by
\begin{equation} \label{eq:involution_definition}
 \begin{split}
\omega(x_i) &= -x_{n+1-i} \ \ \ \textrm{for any} \ i \in [n]; \\
\omega(q_{ij}) &= (-1)^{j-i+1} q_{n+1-j \, n+1-i} \ \ \ \textrm{for any} \ 1 \leq i < j \leq n.
 \end{split}
\end{equation} 

\begin{lemma}  \label{lemma:involution_Ein}
For $1 \leq i \leq j \leq n$, we have
\begin{align*} 
\omega \big(E_i^{[1,j]}\big) = (-1)^i E_i^{[n+1-j,n]} 
\end{align*}
in the polynomial ring $\Z[x_1,\ldots,x_n, q_{ij} \mid 1 \leq i < j \leq n]$. 
In particular, we have 
\begin{align*} 
\omega \big(E_i^{(n)}\big) = (-1)^i E_i^{(n)} \ \ \ \textrm{for any} \ i \in [n]
\end{align*}
in the polynomial ring $\Z[x_1,\ldots,x_n, q_{ij} \mid 1 \leq i < j \leq n]$. 
\end{lemma}

\begin{proof}
For a consecutive substring $I=[k,\ell]$ of $[n]$, we set $I^*=[n+1-\ell, n+1-k]$.
Then one has $\omega(\rho_I)=(-1)^{|I|}\rho_{I^*}$.
By \eqref{eq:Ei[a,b]}, we obtain 
\begin{align*} 
\omega \big(E_i^{[1,j]}\big) &= \sum (-1)^{|I_1|+\cdots+|I_r|} \rho_{I_1^*}\rho_{I_2^*} \cdots \rho_{I_r^*} 
= (-1)^i \sum \rho_{I_1^*}\rho_{I_2^*} \cdots \rho_{I_r^*},
\end{align*}
where the sum runs over all consecutive substrings $I_1^*, I_2^*, \ldots, I_r^*$ of $[n+1-j,n]$ such that $\bigcap_{k=1}^r I_k^* = \emptyset$ and $\sum_{k=1}^r |I_k^*| = i$. 
Hence, this equals $(-1)^i E_i^{[n+1-j,n]}$ as desired.
\end{proof}

We set the quotient ring
\begin{align*} 
Q_n \coloneqq \Z[x_1,\ldots,x_n,q_{ij} \mid 1 \leq i < j \leq n]/(E_1^{(n)}, \ldots, E_n^{(n)}).
\end{align*}
It follows from Lemma~\ref{lemma:involution_Ein} that the involution $\omega$ induces an involution on $Q_n$. 
By a slight abuse of notation, we denote the image of $F \in \Z[x_1,\ldots,x_n,q_{ij} \mid 1 \leq i < j \leq n]$ under the natural projection $\Z[x_1,\ldots,x_n,q_{ij} \mid 1 \leq i < j \leq n] \twoheadrightarrow Q_n$ by the same symbol $F \in Q_n$. 

\begin{proposition} \label{eq:involutionEandH}
For any $1 \leq i \leq j \leq n$, we have
\begin{align*}
\omega(E_i^{(j)}) = H_i^{(n-j)} 
\end{align*}
in the quotient ring $Q_n$. 
\end{proposition}
 
\begin{proof}
By Lemma~\ref{lemma:involution_Ein}, it is enough to prove that 
\begin{align*}
H_i^{(n-j)} = (-1)^i E_i^{[n-j+1,n]} \ \ \ \textrm{in} \ Q_n.
\end{align*}
We prove this by induction on $i$.
The base case is $i=1$, which is clear since $H_1^{(n-j)} + E_1^{[n-j+1,n]} = E_1^{(n)} =0$ in $Q_n$. 
Suppose that $i>1$ and the claim is true for arbitrary $i'$ with $i' \leq i-1$.
From \eqref{eq:relationEandH} we have 
\begin{align*} 
H_i^{(n-j)} &= \sum_{k=1}^i (-1)^{k-1} E_k^{(n-j+i-1)} H_{i-k}^{(n-j)} \\
&= \sum_{k=1}^i (-1)^{k-1} E_k^{(n-j+i-1)} \big((-1)^{i-k} E_{i-k}^{[n-j+1,n]}\big) \ \ \ (\textrm{by the inductive hypothesis}) \\ 
&= (-1)^{i-1} \sum_{k=1}^i E_k^{[1,n-j+i-1]} E_{i-k}^{[n-j+1,n]} \\ 
&= (-1)^{i-1} \left( -E_i^{[n-j+1,n]} + \sum_{\ell=0}^{i-1} E_\ell^{[n-j+1,n-j+i-1]} E_{i-\ell}^{[1,n]} \right) \ \ \ (\textrm{by Proposition~\ref{proposition:decomposition_quantized_elementary_symmetric}})\\ 
&= (-1)^i E_i^{[n-j+1,n]} \ \ \ (\textrm{since} \ E_\ell^{(n)}=0 \ \textrm{for} \ \ell \in [n] \ \textrm{in} \ Q_n)
\end{align*}
in the quotient ring $Q_n$.
This completes the proof. 
\end{proof}

\begin{remark}
Let $V$ be an $n$-dimensional vector space over $\C$. 
We write $V^*$ for the dual space of $V$. 
For a subspace $W \subset V$, we define 
\begin{align*}
W^{\perp} \coloneqq \{f \in V^* \mid f|_W = 0 \}.
\end{align*}
We denote by $\Fl(V)$ the flag variety in $V$ consisting of all nested complex linear subspaces of $V$. 
Then there is a duality isomorphism 
\begin{align} \label{eq:duality isomorphism}
\Fl(V) \xrightarrow{\cong} \Fl(V^*)
\end{align} 
which sends $W_\bullet \coloneqq (\{0 \} = W_0 \subset W_1 \subset \dots \subset W_n =V)$ to $W_\bullet^{\perp} \coloneqq (\{0 \} = W_n^\perp \subset W_{n-1}^\perp \subset \dots \subset W_0^\perp =V^*)$. 
As is well-known, the integral cohomology ring $H^*(\Fl(\C^n);\Z)$ is isomorphic to
\begin{align} \label{eq:cohomology_flag}
H^*(\Fl(\C^n);\Z) \cong \Z[x_1,\ldots,x_n]/(e_1^{(n)},\ldots,e_n^{(n)})
\end{align}
as graded rings (\cite{Bor53}).
Then it is known that the involution on $\Z[x_1,\ldots,x_n]/(e_1^{(n)},\ldots,e_n^{(n)})$ induced from the duality isomorphism \eqref{eq:duality isomorphism} under the identification \eqref{eq:cohomology_flag} maps $e_i^{(j)}$ to $h_i^{(n-j)}$ for all $1 \leq i \leq j \leq n$. 
Proposition~\ref{eq:involutionEandH} gives its $\q$-analogue.
\end{remark}

\section{Polynomials $f_{i,j}$} \label{sect:polynomials fij}

In \cite{AHHM}, polynomials $f_{i,j} \ (i \geq j \geq 1)$ are introduced as follows: 
\begin{align} \label{eq:fij}
f_{i,j} = \sum_{k=1}^{j} \left(\prod_{\ell=j+1}^i (x_k-x_\ell) \right) x_k
\end{align}
where we take the convention $\prod_{\ell=j+1}^i (x_k-x_\ell)=1$ whenever $i=j$.

\begin{remark}
Let $A: V \to V$ be a linear operator of an $n$-dimensional complex vector space $V$.
Consider a weakly increasing function $\hh: [n] \to [n]$ with $\hh(j) \geq j$ for any $j \in [n]$, which is called a \emph{Hessenberg function}.
A \emph{Hessenberg variety} $\Hess(A,\hh)$ associated to $A$ and $\hh$ is defined by 
\begin{align*}
\Hess(A,\hh) \coloneqq \{W_\bullet \in \Fl(V) \mid A(W_i) \subset W_{\hh(i)} \ \textrm{for any} \ i \in [n] \}.
\end{align*}
By using polynomials $f_{i,j}$, \cite{AHHM} gives the following explicit presentation of the rational cohomology rings of Hessenberg varieties $\Hess(N,\hh)$ in $\Fl(\C^n)$ when $N$ is regular nilpotent, i.e. a nilpotent matrix whose Jordan form consists of exactly one Jordan block:
\begin{align*} 
H^*(\Hess(N,\hh);\Q) \cong \Q[x_1,\ldots,x_n]/(f_{\hh(1),1},\ldots,f_{\hh(n),n}).
\end{align*}
\end{remark}

The $\q$-quantization $F_{i,j}$ of $f_{i,j}$ is studied in \cite{Hor25}. 
By \cite[Proposition~6.4]{Hor25} we have
\begin{align} \label{eq:feh}
f_{i,j} = \sum_{k=0}^{i-j+1} (-1)^k (i-k) e_k^{(i)} h_{i-j+1-k}^{(j)}.
\end{align}
for any $i \geq j \geq 1$. 

\begin{lemma} \label{lemma:FEH}
Let $F_{i,j}$ be the $\q$-quantization of $f_{i,j}$. 
Then the following holds:
\begin{align*}
F_{i,j} = \sum_{k=0}^{i-j+1} (-1)^k (i-k) E_k^{(i)} H_{i-j+1-k}^{(j)}.
\end{align*}
\end{lemma}

\begin{proof}
By \eqref{eq:feh}, we have 
\begin{align*} 
f_{i,j} = ih_{i-j+1}^{(j)} + \sum_{k=1}^{i-j+1} (-1)^k (i-k) e_k^{(i)} h_{i-j+1-k}^{(j)}.
\end{align*}
Since $h_{i-j+1-k}^{(j)}$ is a linear combination of standard elementary monomials of the form $e_{*}^{(j)} \cdots e_{*}^{(i-k)}$ by \eqref{eq:h_linear_combination_standard_elementary}, one has 
\begin{align*} 
F_{i,j} = iH_{i-j+1}^{(j)} + \sum_{k=1}^{i-j+1} (-1)^k (i-k) E_k^{(i)} H_{i-j+1-k}^{(j)},
\end{align*}
as desired.
\end{proof}

\begin{lemma}
For $j \in [n]$, we have $F_{n,j}=0$ in the quotient ring $Q_n$.
\end{lemma}

\begin{proof}
It follows from Lemma~\ref{lemma:FEH} that $F_{n,j}=n H_{n-j+1}^{(j)}$ in $Q_n$. 
By \eqref{eq:relationEandH} one has
\begin{align*} 
H_{n-j+1}^{(j)} = \sum_{k=1}^{n-j+1} (-1)^{k-1} E_k^{(n)} H_{n-j+1-k}^{(j)} =0 \ \ \ \textrm{in} \ Q_n,
\end{align*}
which proves the claim.
\end{proof}

\begin{proposition} \label{proposition:involution_Fij}
For $1 \leq j \leq i \leq n-1$, we have
\begin{align*}
\omega(F_{i,j}) = (-1)^{i-j} F_{n-j,n-i}
\end{align*}
in the quotient ring $Q_n$.  
\end{proposition}

\begin{proof}
By Lemma~\ref{lemma:FEH} and Proposition~\ref{eq:involutionEandH} we have
\begin{align*}
\omega(F_{i,j}) = \sum_{k=0}^{i-j+1} (-1)^k (i-k) H_k^{(n-i)} E_{i-j+1-k}^{(n-j)} = \sum_{k=0}^{i-j+1} (-1)^{i-j+1-k} (j+k-1) H_{i-j+1-k}^{(n-i)} E_k^{(n-j)}. 
\end{align*}
By setting $i'= n-j$ and $j'=n-i$, the right hand side above is 
\begin{align*}
&\sum_{k=0}^{i'-j'+1} (-1)^{i'-j'+1-k} ((n-1)-(i'-k)) E_k^{(i')} H_{i'-j'+1-k}^{(j')} \\ 
=&(-1)^{i'-j'} \left( F_{i',j'} + (n-1) \sum_{k=0}^{i'-j'+1} (-1)^{k-1} E_k^{(i')} H_{i'-j'+1-k}^{(j')} \right) \\
=&(-1)^{i-j} F_{i',j'} \ \ \ \textrm{(from \eqref{eq:relationEandH2})},
\end{align*}
as desired.
\end{proof}

\section{Coordinate rings} \label{sect:coordinate rings}

We set 
\begin{align*}
\Z[U]=\Z[z_{ij} \mid 1 \leq j < i \leq n]
\end{align*}
equipped with a grading by 
\begin{align*} 
\deg z_{ij} = 2(i-j)
\end{align*}
for $1 \leq j < i \leq n$.

\begin{remark}
Let $U_\C$ be the lower unipotent group of the general linear group $\GL_n(\C)$.
Then $U_\C$ is naturally identified with an open set around the identity element in the flag variety $\Fl(\C^n)$ and its coordinate ring is $\Z[U] \otimes_\Z \C$.
For this reason, we use the symbol $\Z[U]$ above and we call it the coordinate ring. 
\end{remark}

We put 
\begin{align*}  
z=\left(
 \begin{array}{@{\,}ccccc@{\,}}
     1 & 0 & \cdots & \cdots & 0 \\
     z_{21} & 1 & 0 &  & \vdots \\ 
     z_{31} & z_{32} & 1 & \ddots & \vdots \\ 
     \vdots& \vdots & \ddots & \ddots & 0 \\
     z_{n1} & z_{n2} & \cdots & z_{n \, n-1} & 1 
 \end{array}
\right), \quad
N = 
\begin{pmatrix}
0 & 1 & &   \\
     & \ddots & \ddots & \\
     &   & 0 & 1 \\
     &  & & 0 \\ 
\end{pmatrix}, \quad 
S = 
\begin{pmatrix}
1 &  & &   \\
     & 2 & & \\
     &   & \ddots &  \\
     &  & & n \\ 
\end{pmatrix}. 
\end{align*}
Throughout this paper, we take $z_{ij} = 0$ if either $i > n$ or $j < 1$, or $i<j$, and $z_{ii} = 1$ for $i \in n$.
For $1 \leq j < i \leq n$, we define polynomials $\nu_{i,j}$ and $\xi_{i,j}$ on $\Z[U]$ by
\begin{align} 
\nu_{i,j} \coloneqq (z^{-1} N z)_{ij}; \label{eq:nuij} \\
\xi_{i,j} \coloneqq (z^{-1} S z)_{ij}. \label{eq:xiij}
\end{align} 
It is known that $\nu_{i,j}$ is homogeneous of degree $2(i-j+1)$ and $\xi_{i,j}$ is homogeneous of degree $2(i-j)$ in the coordinate ring $\Z[U]$ by \cite[Lemma~6.4]{HorShi} and \cite[Equation~(7.2)]{Hor25}.

We define the map 
\begin{align} \label{eq:varphi}
\varphi: \Z[U] \to Q_n; \ \ \ z_{ij} \mapsto H_{i-j}^{(j)}.
\end{align}

\begin{theorem} \label{theorem:iso}
The map $\varphi$ in \eqref{eq:varphi} is an isomorphism as graded rings. 
Moreover, we have
\begin{align*}
\varphi(\nu_{i,j}) &= (-1)^{i-j} q_{ji} \ \ \ (1 \leq j < i \leq n) \\
\varphi(\xi_{i,j}) &= F_{i-1,j} \ \ \ \ \ \ \ \ \, (1 \leq j < i \leq n)
\end{align*}
where $F_{i,j}$ denotes the $\q$-quantization of the polynomial $f_{i,j}$ in \eqref{eq:fij}.
\end{theorem}

\begin{proof}
Define the map 
\begin{align*}
\tilde\varphi: \Z[U] \to Q_n; \ \ \ z_{ij} \mapsto E_{i-j}^{(n-j)}. 
\end{align*}
Then the induced map obtained by tensoring with $\C$ 
\begin{align*}
\tilde\varphi_\C: \C[z_{ij} \mid 1 \leq j < i \leq n] \to \C[x_1,\ldots,x_n,q_{ij} \mid 1 \leq i < j \leq n]/(E_1^{(n)}, \ldots, E_n^{(n)})
\end{align*}
is an isomorphism of graded $\C$-algebras by \cite[Theorem~4.13]{HorShi}. 
Hence, we see that $\tilde\varphi$ is injective.
It follows from \cite[Proposition~5.2]{HorShi} that 
\begin{align*}
\tilde\varphi_\C(z_{n-i+1 \, n-i}-z_{n-i+2 \, n-i+1}) &= x_i \ \ \ \textrm{for} \ i \in [n]; \\
\tilde\varphi_\C(-\nu_{n+1-i, n+1-j}) &= q_{ij} \ \ \ \textrm{for} \ 1 \leq i < j \leq n,
\end{align*}
so these equalities hold for $\tilde\varphi$ instead of $\tilde\varphi_\C$, which means a surjectivity of $\tilde\varphi$. 
Thus, $\tilde\varphi$ is an isomorphism as graded rings. 
The map $\varphi$ is the composition of the involution $\omega$ and $\tilde\varphi$ by Proposition~\ref{eq:involutionEandH}, so we conclude that $\varphi$ is also an isomorphism as graded rings.

For the latter part, since we have 
\begin{align*}
\tilde\varphi(\nu_{i,j}) &= -q_{n+1-i \, n+1-j}; \\
\tilde\varphi(\xi_{i,j}) &= (-1)^{i-j+1}F_{n-j, n+1-i} 
\end{align*}
for $1 \leq j < i \leq n$ by \cite[Proposition~5.2]{HorShi} and \cite[Theorem~5.3]{Hor25}, we obtain from \eqref{eq:involution_definition} and Proposition~\ref{proposition:involution_Fij} that
\begin{align*}
\varphi(\nu_{i,j}) &= -\omega(q_{n+1-i \, n+1-j}) = (-1)^{i-j}q_{ji}; \\
\varphi(\xi_{i,j}) &= (-1)^{i-j+1} \omega(F_{n-j, n+1-i}) = F_{i-1,j}. 
\end{align*}
This completes the proof.
\end{proof}

\begin{remark}
In setting $q_{ij}=0$ for all $1 \leq i < j \leq n$, the isomorphism $\varphi$ in \eqref{eq:varphi} is described as 
\begin{align*}
\Z[U]/(\nu_{i,j} \mid 1 \leq j < i \leq n) \cong \Z[x_1,\ldots,x_n]/(e_1^{(n)}, \ldots, e_n^{(n)}). 
\end{align*}
A geometric meaning of this isomorphism is given in \cite[Proposition~2.1]{AkyAky}.
Here we note that our generators $\nu_{i,j}$ appeared in the left hand side above are different from those given in \cite{AkyAky}. 
\end{remark}

\begin{remark}
Recall that $U_\C$ is the lower unipotent group of $\GL_n(\C)$, which is identified with the open set around the identity element in the flag variety $\Fl(\C^n)$.
Geometrically, the functions $\nu_{i,j}$ and $\xi_{i,j}$ mean defining functions for Hessenberg varieties $\Hess(N,\hh)$ and $\Hess(S,\hh)$ in $U_\C$, respectively. 
A relation between the coordinate ring of $\Hess(S,\hh) \cap U_\C$ and the (complex) cohomology ring of $\Hess(N,\hh)$ is studied in \cite{Hor25}.
Here, the dual $\hh^*: [n] \to [n]$ of a Hessenberg function $\hh: [n] \to [n]$ appears in the result in \cite{Hor25} where $\hh^*$ is defined by
\begin{align*} 
\hh^*(i) \coloneqq |\{j \in [n] \mid \hh(j) \geq n+1-i \}| \ \ \ \textrm{for} \ i \in [n]. 
\end{align*}
We may interpret why the dual Hessenberg function $\hh^*$ appears in \cite{Hor25} as follows.
The duality isomorphism in \eqref{eq:duality isomorphism} maps a Hessenberg variety $\Hess(A,\hh)$ onto the Hessenberg variety $\Hess(A^*,\hh^*)$ where $A^*: V^* \to V^*$ is the dual of a linear operator $A: V \to V$. 
This fact would affect the reason why the dual Hessenberg function $\hh^*$ appears in \cite{Hor25}.
In fact, the isomorphism $\varphi$ in \eqref{eq:varphi} induces an isomorphism
\begin{align*}
\Z[U]/(\xi_{i,j} \mid i > \hh(j)) \cong Q_n/(F_{i,j} \mid \hh(j) \leq i \leq n-1). 
\end{align*}
By tensoring with $\C$, we see that the coordinate ring of $\Hess(S,\hh) \cap U_\C$ has a presentation $\C[x_1,\ldots,x_n, q_{ij} \mid 1 \leq i < j \leq n]/(F_{i,j} \mid i \geq \hh(j))$, which specializes to the complex cohomology ring of the regular nilpotent Hessenberg variety $\Hess(N,\hh)$ in the classical limit $q_{ij}=0$ for all $1 \leq i < j \leq n$. 
\end{remark}

\section{$\q$-quantum Schur polynomials} \label{sect:quantum Schur polynomials}

Recall that $\YY_k(n)$ is the set of Young diagrams $\lambda$ with at most $k$ rows and at most $n-k$ columns. 
For $\lambda \in \YY_k(n)$, we write $S_\lambda$ for the $\q$-quantization of the Schur polynomial $s_\lambda(x_1,\ldots,x_k)$.
We call $S_\lambda$ the \emph{$\q$-quantum Schur polynomial}. 
Let $\underbar{{\bf j}}=(j_1,\ldots,j_{n-k}) \in \binom{[n]}{n-k}$ be the sequence in \eqref{eq:one-to-one correspondence PP II} corresponding to the transpose $\lambda^t$ of $\lambda$. 
By \eqref{eq:Schur_e} we have
\begin{align} \label{eq:QuantumSchurE}
 S_{\lambda} = \left|
 \begin{array}{@{\,}cccc@{\,}}
  E_{j_1-1}^{(n-1)} &  E_{j_1-2}^{(n-2)} & \cdots & E_{j_1-(n-k)}^{(k)}   \smallskip \\
  E_{j_2-1}^{(n-1)} &  E_{j_2-2}^{(n-2)} & \cdots & E_{j_2-(n-k)}^{(k)} \\
  \vdots &  \vdots &  & \vdots \\
  E_{j_{n-k}-1}^{(n-1)} &  E_{j_{n-k}-2}^{(n-2)} & \cdots & E_{j_{n-k}-(n-k)}^{(k)} 
  \end{array}
 \right|. 
\end{align}

\begin{proposition} 
Let $A_n$ and $B_n$ be the following $n \times n$ matrices:
\begin{align*}  
A_n=\left(
 \begin{array}{@{\,}ccccc@{\,}}
     1 & 0 & \cdots & \cdots & 0 \\
     H_1^{(1)} & 1 & 0 &  & \vdots \\ 
     H_2^{(1)} & H_1^{(2)} & 1 & \ddots & \vdots \\ 
     \vdots& \vdots & \ddots & \ddots & 0 \\
     H_{n-1}^{(1)} & H_{n-2}^{(2)} & \cdots & H_1^{(n-1)} & 1 
 \end{array}
\right), \quad
B_n=\left(
 \begin{array}{@{\,}ccccc@{\,}}
     1 & 0 & \cdots & \cdots & 0 \\
     E_1^{(n-1)} & 1 & 0 &  & \vdots \\ 
     E_2^{(n-1)} & E_1^{(n-2)} & 1 & \ddots & \vdots \\ 
     \vdots& \vdots & \ddots & \ddots & 0 \\
     E_{n-1}^{(n-1)} & E_{n-2}^{(n-2)} & \cdots & E_1^{(1)} & 1 
 \end{array}
\right). 
\end{align*}
For any sequences $\underbar{{\bf i}}=(i_1,\ldots,i_k) \in \binom{[n]}{k}$ and $\underbar{{\bf j}}=(j_1,\ldots,j_{n-k}) \in \binom{[n]}{n-k}$ such that 
\begin{align*}
\{j_1,\ldots,j_{n-k}, n+1-i_1,\ldots,n+1-i_k \} = [n],
\end{align*}
the determinant of the submatrix of $A_n$ associated to row indices $\underbar{{\bf i}}$ and column indices $[k]$ coincides with the determinant of the submatrix of $B_n$ associated to row indices $\underbar{{\bf j}}$ and column indices $[n-k]$.
In other words, we have 
\begin{align} \label{eq:AnBn}
 \left|
 \begin{array}{@{\,}cccc@{\,}}
  H_{i_1-1}^{(1)} &  H_{i_1-2}^{(2)} & \cdots & H_{i_1-k}^{(k)}  \smallskip \\
  H_{i_2-1}^{(1)} &  H_{i_2-2}^{(2)} & \cdots & H_{i_2-k}^{(k)} \\
  \vdots &  \vdots &  & \vdots \\
  H_{i_k-1}^{(1)} &  H_{i_k-2}^{(2)} & \cdots & H_{i_k-k}^{(k)}  
  \end{array}
 \right|
= \left|
 \begin{array}{@{\,}cccc@{\,}}
  E_{j_1-1}^{(n-1)} &  E_{j_1-2}^{(n-2)} & \cdots & E_{j_1-(n-k)}^{(k)}   \smallskip \\
  E_{j_2-1}^{(n-1)} &  E_{j_2-2}^{(n-2)} & \cdots & E_{j_2-(n-k)}^{(k)} \\
  \vdots &  \vdots &  & \vdots \\
  E_{j_{n-k}-1}^{(n-1)} &  E_{j_{n-k}-2}^{(n-2)} & \cdots & E_{j_{n-k}-(n-k)}^{(k)} 
  \end{array}
 \right| 
\end{align}
in the polynomial ring $\Z[x_1,\ldots,x_n,q_{ij} \mid 1 \leq i < j \leq n]$. 
\end{proposition}

\begin{proof}
We prove \eqref{eq:AnBn} by induction on $n$.
We first consider the base case $n=2$. 
In this case, the possible sequences are $\underbar{{\bf i}}=(1)$ and $\underbar{{\bf j}}=(1)$, or $\underbar{{\bf i}}=(2)$ and $\underbar{{\bf j}}=(2)$. 
Since $H_0^{(1)} = E_0^{(1)} = 1$ and $H_1^{(1)} = E_1^{(1)} = x_1$, we proved the base case.

We now suppose that $n > 2$ and that the claim holds for $n-1$ and any sequences $\underbar{{\bf i}}'=(i'_1,\ldots,i'_k) \in \binom{[n-1]}{k}$ and $\underbar{{\bf j}}'=(j'_1,\ldots,j'_{n-1-k}) \in \binom{[n-1]}{n-1-k}$ with 
$\{j'_1,\ldots,j'_{n-1-k}, n-i'_1,\ldots,n-i'_k \} = [n-1]$. 
We take cases as follows.

\noindent
\textbf{Case~(i):} Suppose that $i_k \neq n$. 
In this case we have $j_1=1$. 
Since the first row of the right hand side in \eqref{eq:AnBn} is $(1,0,\dots,0)$, the right hand side of \eqref{eq:AnBn} is 
\begin{align} \label{eq:proofAnBn_1} 
 \left|
 \begin{array}{@{\,}ccc@{\,}}
  E_{j_2-2}^{(n-2)} & \cdots & E_{j_2-(n-k)}^{(k)} \\
 \vdots &  & \vdots \\
  E_{j_{n-k}-2}^{(n-2)} & \cdots & E_{j_{n-k}-(n-k)}^{(k)} 
  \end{array}
 \right|. 
\end{align}
Note that $\{i_1,\ldots,i_k, n+1-j_2,\ldots,n+1-j_{n-k} \} = [n-1]$, which is equivalently $\{j_2-1,\ldots,j_{n-k}-1, n-i_1,\ldots,n-i_k \} = [n-1]$. 
By setting $\underbar{{\bf i}}'=(i_1,\ldots,i_k) \in \binom{[n-1]}{k}$ and $\underbar{{\bf j}}'=(j_2-1,\ldots,j_{n-k}-1) \in \binom{[n-1]}{n-1-k}$, \eqref{eq:proofAnBn_1} is equal to 
\begin{align*}
 \left|
 \begin{array}{@{\,}cccc@{\,}}
  H_{i_1-1}^{(1)} &  H_{i_1-2}^{(2)} & \cdots & H_{i_1-k}^{(k)}  \smallskip \\
  H_{i_2-1}^{(1)} &  H_{i_2-2}^{(2)} & \cdots & H_{i_2-k}^{(k)} \\
  \vdots &  \vdots &  & \vdots \\
  H_{i_k-1}^{(1)} &  H_{i_k-2}^{(2)} & \cdots & H_{i_k-k}^{(k)}  
  \end{array}
 \right|
\end{align*}
by our inductive assumption on $n$.

\noindent
\textbf{Case~(ii):} Suppose that $i_k = n$. 
Then we have $\{i_1,\ldots,i_{k-1}, n+1-j_1,\ldots,n+1-j_{n-k} \} = [n-1]$, which is equivalently $\{j_1-1,\ldots,j_{n-k}-1, n-i_1,\ldots,n-i_{k-1} \} = [n-1]$. 
The right hand side of \eqref{eq:AnBn} is  
\begin{align} \label{eq:proofAnBn_2}
\sum_{\ell=1}^{n-k} (-1)^{\ell-1} E_{j_\ell-1}^{(n-1)} \left|
 \begin{array}{@{\,}cccc@{\,}}
  E_{j_1-2}^{(n-2)} &  E_{j_1-3}^{(n-3)} & \cdots & E_{j_1-(n-k)}^{(k)}  \\
  \vdots &  \vdots &  & \vdots \\
  E_{j_{\ell-1}-2}^{(n-2)} &  E_{j_{\ell-1}-3}^{(n-3)} & \cdots & E_{j_{\ell-1}-(n-k)}^{(k)}   \smallskip \\
  E_{j_{\ell+1}-2}^{(n-2)} &  E_{j_{\ell+1}-3}^{(n-3)} & \cdots & E_{j_{\ell+1}-(n-k)}^{(k)} \\
  \vdots &  \vdots &  & \vdots \\
  E_{j_{n-k}-2}^{(n-2)} &  E_{j_{n-k}-2}^{(n-3)} & \cdots & E_{j_{n-k}-(n-k)}^{(k)} 
  \end{array}
 \right| 
\end{align}
by the cofactor expansion along the first column. 
For each $\ell \in [n-k]$, there is a unique positive integer $p$ such that $i_{p-1} < n+1-j_\ell < i_p$ with the convention $i_0=0$. 
Then it follows from the inductive assumption on $n$ that  
\begin{align} \label{eq:proofAnBn_3} 
 \left|
 \begin{array}{@{\,}cccc@{\,}}
  E_{j_1-2}^{(n-2)} &  E_{j_1-3}^{(n-3)} & \cdots & E_{j_1-(n-k)}^{(k)} \\
  \vdots &  \vdots &  & \vdots \\
  E_{j_{\ell-1}-2}^{(n-2)} &  E_{j_{\ell-1}-3}^{(n-3)} & \cdots & E_{j_{\ell-1}-(n-k)}^{(k)}   \smallskip \\
  E_{j_{\ell+1}-2}^{(n-2)} &  E_{j_{\ell+1}-3}^{(n-3)} & \cdots & E_{j_{\ell+1}-(n-k)}^{(k)} \\
  \vdots &  \vdots &  & \vdots \\
  E_{j_{n-k}-2}^{(n-2)} &  E_{j_{n-k}-2}^{(n-3)} & \cdots & E_{j_{n-k}-(n-k)}^{(k)} 
  \end{array}
 \right| =
 \left|
 \begin{array}{@{\,}cccc@{\,}}
  H_{i_1-1}^{(1)} &  H_{i_1-2}^{(2)} & \cdots & H_{i_1-k}^{(k)}  \\
  \vdots &  \vdots &  & \vdots \\
  H_{i_{p-1}-1}^{(1)} &  H_{i_{p-1}-2}^{(2)} & \cdots & H_{i_{p-1}-k}^{(k)}  \smallskip \\
  H_{n-j_\ell}^{(1)} &  H_{n-j_\ell-1}^{(2)} & \cdots & H_{n-j_\ell-k+1}^{(k)}  \smallskip \\
  H_{i_p-1}^{(1)} &  H_{i_p-2}^{(2)} & \cdots & H_{i_p-k}^{(k)}  \\
  \vdots &  \vdots &  & \vdots \\
  H_{i_{k-1}-1}^{(1)} &  H_{i_{k-1}-2}^{(2)} & \cdots & H_{i_{k-1}-k}^{(k)}  
  \end{array}
 \right|. 
\end{align}
Here we note that $\{n+1-j_\ell,\ldots,n+1-j_1,i_p,\ldots,i_k \} = [n+1-j_\ell, n]$.
By comparing the cardinality, we have $k-p=j_\ell-\ell-1$.
Thus, one can see from \eqref{eq:proofAnBn_2} and \eqref{eq:proofAnBn_3} that the right hand side of \eqref{eq:AnBn} is written as
\begin{align} \label{eq:proofAnBn_4} 
\sum_{\ell=1}^{n-k} (-1)^{j_\ell} E_{j_\ell-1}^{(n-1)} \left|
 \begin{array}{@{\,}cccc@{\,}}
  H_{i_1-1}^{(1)} &  H_{i_1-2}^{(2)} & \cdots & H_{i_1-k}^{(k)}  \\
  \vdots &  \vdots &  & \vdots \\
  H_{i_{k-1}-1}^{(1)} &  H_{i_{k-1}-2}^{(2)} & \cdots & H_{i_{k-1}-k}^{(k)} \smallskip \\ 
  H_{n-j_\ell}^{(1)} &  H_{n-j_\ell-1}^{(2)} & \cdots & H_{n-j_\ell-k+1}^{(k)} 
  \end{array}
 \right|. 
\end{align}
On the other hand, for the left hand side of \eqref{eq:AnBn}, the entry of the $j$-th column in the last row is $H_{n-j}^{(j)}$ since $i_k=n$, which is written as 
\begin{align*}
H_{n-j}^{(j)} = \sum_{r=1}^{n-j} (-1)^{r-1} E_r^{(n-1)} H_{n-j-r}^{(j)} = \sum_{r=1}^{n-1} (-1)^{r-1} E_r^{(n-1)} H_{n-j-r}^{(j)}
\end{align*}
by \eqref{eq:relationEandH} with the convention $H_i^{(j)}=0$ for $i<0$.
Hence, the left hand side of \eqref{eq:AnBn} is 
\begin{align} \label{eq:proofAnBn_5} 
\sum_{r=1}^{n-1} (-1)^{r-1} E_{r}^{(n-1)} \left|
 \begin{array}{@{\,}cccc@{\,}}
  H_{i_1-1}^{(1)} &  H_{i_1-2}^{(2)} & \cdots & H_{i_1-k}^{(k)}  \\
  \vdots &  \vdots &  & \vdots \\
  H_{i_{k-1}-1}^{(1)} &  H_{i_{k-1}-2}^{(2)} & \cdots & H_{i_{k-1}-k}^{(k)} \smallskip \\ 
  H_{n-r-1}^{(1)} &  H_{n-r-2}^{(2)} & \cdots & H_{n-r-k}^{(k)} 
  \end{array}
 \right|. 
\end{align}
The determinant appeared in \eqref{eq:proofAnBn_5} is equal to $0$ if $n-r \in \{i_1,\ldots,i_{k-1} \}$, so it is enough to consider $r$ such that $n-r \notin \{i_1,\ldots,i_{k-1} \}$ in \eqref{eq:proofAnBn_5}.
Such an $r$ is of the form $j_\ell-1 \ (\ell \in [n-k])$ since $\{j_1-1,\ldots,j_{n-k}-1, n-i_1,\ldots,n-i_{k-1} \} = [n-1]$.
Therefore, we obtain that \eqref{eq:proofAnBn_5} is equal to \eqref{eq:proofAnBn_4}, which means the desired equality \eqref{eq:AnBn}. 
This completes the proof.
\end{proof}

\begin{theorem} \label{theorem:QuantumSchurH}
Let $\lambda \in \YY_k(n)$ and $\underbar{{\bf i}} \in \binom{[n]}{k}$ the associated sequence in \eqref{eq:one-to-one correspondence PP II}.
Then the $\q$-quantum Schur polynomial $S_\lambda$ is described as
\begin{align} \label{eq:QuantumSchurH}
S_\lambda = \left|
 \begin{array}{@{\,}cccc@{\,}}
  H_{i_1-1}^{(1)} &  H_{i_1-2}^{(2)} & \cdots & H_{i_1-k}^{(k)}  \smallskip \\
  H_{i_2-1}^{(1)} &  H_{i_2-2}^{(2)} & \cdots & H_{i_2-k}^{(k)} \\
  \vdots &  \vdots &  & \vdots \\
  H_{i_k-1}^{(1)} &  H_{i_k-2}^{(2)} & \cdots & H_{i_k-k}^{(k)}  
  \end{array}
 \right|
\end{align}
in the polynomial ring $\Z[x_1,\ldots,x_n,q_{ij} \mid 1 \leq i < j \leq n]$. 
\end{theorem}

\begin{proof}
Let $\underbar{{\bf j}}=(j_1,\ldots,j_{n-k}) \in \binom{[n]}{n-k}$ be the sequence in \eqref{eq:one-to-one correspondence PP II} corresponding to the transpose $\lambda^t$ of $\lambda$. 
Then we have $\{j_1,\ldots,j_{n-k}, n+1-i_1,\ldots,n+1-i_k \} = [n]$, so the desired identity \eqref{eq:QuantumSchurH}
follows from \eqref{eq:QuantumSchurE} and \eqref{eq:AnBn}. 
\end{proof}

\begin{corollary} \label{corollary:involution_quantumSchur}
For $\lambda \in \YY_k(n)$, one has
\begin{align*}
\omega(S_\lambda) = S_{\lambda^t}
\end{align*}
in the quotient ring $Q_n$.
\end{corollary}

\begin{proof}
It follows from \eqref{eq:QuantumSchurE}, \eqref{eq:QuantumSchurH}, and Proposition~\ref{eq:involutionEandH}.
\end{proof}

For a sequence $\underbar{{\bf i}}=(i_1,\ldots,i_k) \in \binom{[n]}{k}$, we define $\Plucker_{\, \underbar{{\bf i}}}$ by
the determinant of the submatrix of the lower unipontent matrix $(z_{ij})$ associated to row indices $\underbar{{\bf i}}$ and column indices $[k]$.
Namely, we have 
\begin{align*} 
\Plucker_{\, \underbar{{\bf i}}} = \left|
 \begin{array}{@{\,}cccc@{\,}}
  z_{i_1 1} &  z_{i_1 2} & \cdots & z_{i_1 k}  \smallskip \\
  z_{i_2 1} &  z_{i_2 2} & \cdots & z_{i_2 k} \\
  \vdots &  \vdots &  & \vdots \\
  z_{i_k 1} &  z_{i_k 2} & \cdots & z_{i_k k}  
  \end{array}
 \right| \ \ \ \textrm{in} \ \Z[U].
\end{align*}
Note that $z_{ij} = \Plucker_{(1,2,\ldots,j-1,i)}$ for any $1 \leq j \leq i \leq n$.

\begin{corollary} \label{corollary:pi_Slambda}
Let $\varphi$ be the isomorphism in \eqref{eq:varphi}.
For any sequence $\underbar{{\bf i}}=(i_1,\ldots,i_k) \in \binom{[n]}{k}$, we have
\begin{align*}
\varphi(\Plucker_{\, \underbar{{\bf i}}}) = S_\lambda
\end{align*}
where $\lambda \in \YY_k(n)$ is the Young diagram corresponding to $\underbar{{\bf i}}$ in \eqref{eq:one-to-one correspondence PP II}.
\end{corollary}

\begin{proof}
The result follows from Theorem~\ref{theorem:QuantumSchurH}.
\end{proof}

\begin{remark}
In the classical limit $q_{ij}=0$ for $1 \leq i < j \leq n$, Corollary~\ref{corollary:pi_Slambda} corresponds to \cite[Theorem~3.1]{AkyAky}. 
\end{remark}

Under the isomorphism $\varphi$ in \eqref{eq:varphi}, the involution $\omega$ on $Q_n$ induces the involution $\omega$ on $\Z[U]$.
We denote the involution on $\Z[U]$ by the same symbol $\omega$.
Then the following holds.

\begin{corollary}
Let $\underbar{{\bf i}}=(i_1,\ldots,i_k) \in \binom{[n]}{k}$ be a  sequence and $\underbar{{\bf j}}=(j_1,\ldots,j_{n-k}) \in \binom{[n]}{n-k}$ the sequence defined by $\{j_1,\ldots,j_{n-k}, n+1-i_1,\ldots,n+1-i_k \} = [n]$.
Then we have
\begin{align*}
\omega(\Plucker_{\, \underbar{{\bf i}}}) = \Plucker_{\, \underbar{{\bf j}}}
\end{align*}
in the coordinate ring $\Z[U]$.
In particular, we have $\omega(z_{ij}) = \Plucker_{(1,2,\ldots,n-i,n-i+2,n-i+3,\ldots,n-j+1)}$.
\end{corollary}

\begin{proof}
It follows from Corollaries~\ref{corollary:involution_quantumSchur} and \ref{corollary:pi_Slambda}.
\end{proof}

\section{A symmetric group action on $\Z[U]$} \label{sect:a symmetric group action on ZU}

In this section we construct a symmetric group action on the coordinate ring $\Z[U]$, which is lifted and induced from a usual action of the symmetric group on the quotient ring $\Z[x_1,\ldots,x_n]/(e_1^{(1)},\ldots,e_n^{(n)})$.
For this purpose, we review the definition and some basic properties for devided difference operators on $\Z[x_1,\ldots,x_n]$. 

Let $S_n$ be the symmetric group on $n$ letters $[n]$. 
As is well-known, the symmetric group $S_n$ is generated by the adjacent transpositions $s_k$ of $k$ and $k+1$ for $k \in [n-1]$ and their relations are given by
\begin{align*}
s_k^2 &= 1; \\
s_k s_\ell &= s_\ell s_k  \ \ \ \textrm{if} \ |k-\ell|>1; \\
s_k s_{k+1} s_k &= s_{k+1} s_k s_{k+1}. 
\end{align*}
For a permutation $w \in S_n$, the length $\ell(w)$ of $w$ is the minimal number $r$ of the adjacent transpositions needed to write $w = s_{k_1} s_{k_2} \dots s_{k_r}$.
An expression $w = s_{k_1} s_{k_2} \dots s_{k_r}$ is called reduced if $r = \ell(w)$.

The symmetric group $S_n$ naturally acts on the polynomial ring $\Z[x_1,\ldots,x_n]$ by the following rule:
\begin{align} \label{eq:Sn_action_xi}
w\big(f(x_1,\ldots,x_n)\big) = f(x_{w(1)},\ldots,x_{w(n)})
\end{align}
for $w \in S_n$ and $f(x_1,\ldots,x_n) \in \Z[x_1,\ldots,x_n]$.
For each $k \in [n-1]$, the divided difference operator $\partial_k$ on the polynomial ring $\Z[x_1,\ldots,x_n]$ is defined by
\begin{align} \label{eq:definition_divided_difference_operator}
\partial_k(f) = \frac{f-s_k(f)}{x_k-x_{k+1}} \ \ \ \textrm{for} \ f \in \Z[x_1,\ldots,x_n].
\end{align}
Note that $f-s_k(f)$ is divisible by $x_k-x_{k+1}$, so $\partial_k(f)$ is a polynomial again. 
For any polynomials $f, g \in \Z[x_1,\ldots,x_n]$ and any $k \in [n-1]$, we have the following ``Leibniz formula":
\begin{align} \label{eq:Leibniz formula}
\partial_k(f \cdot g) = \partial_k(f) \cdot g + s_k(f) \cdot \partial_k(g). 
\end{align}
One can check that the divided difference operators satisfy the following relations
\begin{align*}
\partial_k^2 &= 0; \\
\partial_k \partial_\ell &= \partial_\ell \partial_k  \ \ \ \textrm{if} \ |k-\ell|>1; \\
\partial_k \partial_{k+1} \partial_k &= \partial_{k+1} \partial_k \partial_{k+1}. 
\end{align*}
Thus, if we set $\partial_w = \partial_{k_1} \partial_{k_2} \cdots \partial_{k_r}$ for a reduced expression $w = s_{k_1} s_{k_2} \dots s_{k_r}$, then the operator $\partial_w$ is independent of the choice of reduced expressions for $w$.

By the definition of $\partial_k$, a polynomial $f \in \Z[x_1,\ldots,x_n]$ is symmetric in $x_k$ and $x_{k+1}$ if and only if $\partial_k(f)=0$.
In particular, the divided difference operator $\partial_k$ on the polynomial ring $\Z[x_1,\ldots,x_n]$ induces that on the quotient ring $\Z[x_1,\ldots,x_n]/(e_1^{(n)},\ldots,e_n^{(n)})$.
From Theorem~\ref{theorem:iso} there is the following isomorphism of graded rings:
\begin{align} \label{eq:isomorphism_forget_quantum}
\Z[U]/(\nu_{i,j} \mid 1 \leq j <i \leq n) \xrightarrow{\cong} \Z[x_1,\ldots,x_n]/(e_1^{(n)},\ldots,e_n^{(n)}); \ \ \ z_{ij} \mapsto h_{i-j}^{(j)}.
\end{align}
Thus, we obtain the divided difference operator $\partial_k$ on $\Z[U]/(\nu_{i,j} \mid 1 \leq j <i \leq n)$ for each $k \in [n-1]$. 
In what follows, we will see that the divided difference operator $\partial_k$ on $\Z[U]/(\nu_{i,j} \mid 1 \leq j <i \leq n)$ can be lifted on the coordinate ring $\Z[U]$.
For this purpose, we first note that 
\begin{align*}
\partial_{k}(h_{i-j}^{(j)}) = \begin{cases}
h_{i-j-1}^{(j+1)} \ \ \ &\textrm{if} \ k=j; \\
0 \ \ \ &\textrm{if} \ k \neq j 
\end{cases}
\end{align*}
in the polynomial ring $\Z[x_1,\ldots,x_n]$ (cf. \cite[Lemma~6.3]{Hor25}). 
Under the isomorphism \eqref{eq:isomorphism_forget_quantum}, one can translate this formula to 
\begin{align*}
\partial_{k}(z_{ij}) = \begin{cases}
z_{i \, j+1} \ \ \ &\textrm{if} \ k=j; \\
0 \ \ \ &\textrm{if} \ k \neq j 
\end{cases}
\end{align*}
in the quotient ring $\Z[U]/(\nu_{i,j} \mid 1 \leq j <i \leq n)$.
Motivated by the formula above with \eqref{eq:definition_divided_difference_operator}, we introduce an $S_n$-action on the coordinate ring $\Z[U]$ as follows.

\begin{definition} \label{definition:Sn_action_Z[U]}
Let $1 \leq j < i \leq n$.
For each $k \in [n-1]$, we define 
\begin{align*}
s_k(z_{ij}) \coloneqq 
\begin{cases}
z_{ij} + (z_{j \, j-1} -2 z_{j+1 \, j} + z_{j+2 \, j+1}) z_{i \, j+1} \ \ \ &\textrm{if} \ k=j; \\
z_{ij} \ \ \ &\textrm{if} \ k \neq j 
\end{cases}
\end{align*}
in the coordinate ring $\Z[U]$.
In particular, one has 
\begin{align*}
s_j(z_{j+1 \, j}) = z_{j \, j-1} - z_{j+1 \, j} + z_{j+2 \, j+1} \ \ \ \textrm{for} \ j \in [n-1].
\end{align*}
\end{definition}

\begin{lemma} \label{lemma:Sn_action_ZU}
The following equalities hold:
\begin{align*}
&{\rm (i)} \, s_k^2(z_{ij}) = z_{ij}; \\
&{\rm (ii)} \, s_k s_\ell (z_{ij}) = s_\ell s_k (z_{ij}) \ \ \ \textrm{if} \ |k-\ell|>1; \\
&{\rm (iii)} \, s_k s_{k+1} s_k (z_{ij}) = s_{k+1} s_k s_{k+1} (z_{ij}). 
\end{align*}
\end{lemma}

\begin{proof}
(i) If $k \neq j$, then one has $s_k(z_{ij}) = z_{ij}$, so $s_k^2(z_{ij}) = z_{ij}$.
If $k = j$, then we have 
\begin{align*}
s_j^2(z_{ij}) &= s_j \big(z_{ij} + (z_{j \, j-1} -2 z_{j+1 \, j} + z_{j+2 \, j+1}) z_{i \, j+1} \big) \\
&= z_{ij} + (z_{j \, j-1} -2 z_{j+1 \, j} + z_{j+2 \, j+1}) z_{i \, j+1} + (-z_{j \, j-1} +2 z_{j+1 \, j} - z_{j+2 \, j+1}) z_{i \, j+1} \\
&= z_{ij}. 
\end{align*}

\noindent
(ii) It is enough to check the case when $k=j$ and $\ell \neq j$. 
Since $|j-\ell|>1$, one has 
\begin{align*}
s_j s_\ell (z_{ij}) &= s_j (z_{ij}) = z_{ij} + (z_{j \, j-1} -2 z_{j+1 \, j} + z_{j+2 \, j+1}) z_{i \, j+1}; \\
s_\ell s_j (z_{ij}) &= s_\ell \big( z_{ij} + (z_{j \, j-1} -2 z_{j+1 \, j} + z_{j+2 \, j+1}) z_{i \, j+1} \big) \\
&= z_{ij} + (z_{j \, j-1} -2 z_{j+1 \, j} + z_{j+2 \, j+1}) z_{i \, j+1}. 
\end{align*}

\noindent
(iii) It suffices to show the two cases when $k=j-1$ and $k = j$. 
If $k=j-1$, then we have
\begin{align*}
s_{j-1} s_j s_{j-1} (z_{ij}) &= s_{j-1} s_j (z_{ij}) = s_{j-1} \big( z_{ij} + (z_{j \, j-1} -2 z_{j+1 \, j} + z_{j+2 \, j+1}) z_{i \, j+1} \big) \\
&= z_{ij} + (z_{j-1 \, j-2} - z_{j \, j-1} - z_{j+1 \, j} + z_{j+2 \, j+1}) z_{i \, j+1} ; \\
s_j s_{j-1} s_j (z_{ij}) &= s_j s_{j-1} \big( z_{ij} + (z_{j \, j-1} -2 z_{j+1 \, j} + z_{j+2 \, j+1}) z_{i \, j+1} \big) \\
&= s_j \big( z_{ij} + (z_{j-1 \, j-2} - z_{j \, j-1} - z_{j+1 \, j} + z_{j+2 \, j+1}) z_{i \, j+1} \big) \\
&= z_{ij} + (z_{j \, j-1} -2 z_{j+1 \, j} + z_{j+2 \, j+1}) z_{i \, j+1}+(z_{j-1 \, j-2} - 2z_{j \, j-1} + z_{j+1 \, j}) z_{i \, j+1} \\
&= z_{ij} + (z_{j-1 \, j-2} - z_{j \, j-1} - z_{j+1 \, j} + z_{j+2 \, j+1}) z_{i \, j+1}.
\end{align*}
If $k=j$, then $s_j s_{j+1} s_j (z_{ij})$ equals 
\begin{align*}
&s_j s_{j+1} \big( z_{ij} + (z_{j \, j-1} -2 z_{j+1 \, j} + z_{j+2 \, j+1}) z_{i \, j+1} \big) \\
=& s_j \big( z_{ij} + ( z_{j \, j-1} - z_{j+1 \, j} - z_{j+2 \, j+1} + z_{j+3 \, j+2})(z_{i \, j+1} + (z_{j+1 \, j} -2 z_{j+2 \, j+1} + z_{j+3 \, j+2}) z_{i \, j+2}) \big) \\
=& z_{ij} + (z_{j \, j-1} -2 z_{j+1 \, j} + z_{j+2 \, j+1}) z_{i \, j+1} \\
& \hspace{20pt} +(z_{j+1 \, j} -2 z_{j+2 \, j+1} + z_{j+3 \, j+2})\big(z_{i \, j+1} + ( z_{j \, j-1} - z_{j+1 \, j} - z_{j+2 \, j+1} + z_{j+3 \, j+2}) z_{i \, j+2} \big) \\
=& z_{ij} + (z_{j \, j-1} - z_{j+1 \, j} - z_{j+2 \, j+1} + z_{j+3 \, j+2}) z_{i \, j+1} \\
& \hspace{20pt} +(z_{j+1 \, j} -2 z_{j+2 \, j+1} + z_{j+3 \, j+2})( z_{j \, j-1} - z_{j+1 \, j} - z_{j+2 \, j+1} + z_{j+3 \, j+2}) z_{i \, j+2}. 
\end{align*}
On the other hand, $s_{j+1} s_j s_{j+1} (z_{ij})$ is computed as 
\begin{align*}
& s_{j+1} \big( z_{ij} + (z_{j \, j-1} -2 z_{j+1 \, j} + z_{j+2 \, j+1}) z_{i \, j+1} \big) \\
=& z_{ij} + (z_{j \, j-1} - z_{j+1 \, j} - z_{j+2 \, j+1} + z_{j+3 \, j+2}) \big(z_{i \, j+1} + ( z_{j+1 \, j} -2 z_{j+2 \, j+1} + z_{j+3 \, j+2}) z_{i \, j+2} \big).
\end{align*}
This shows the statement~(iii).
\end{proof}

For $w \in S_n$, we write $w = s_{k_1} s_{k_2} \dots s_{k_r}$, which is not necessarily a reduced decomposition.
Then we define
\begin{align*} 
w(z_{ij}) = s_{k_1} s_{k_2} \dots s_{k_r} (z_{ij}).
\end{align*}
This is well-defined from Lemma~\ref{lemma:Sn_action_ZU}. 
We now define the divided difference operator $\partial_k$ on $\Z[U]$ as follows:
\begin{align} \label{eq:definition_divided_difference_operator_ZU}
\partial_k(P) = \frac{P-s_k(P)}{-z_{k \, k-1} +2 z_{k+1 \, k} - z_{k+2 \, k+1}} \ \ \ \textrm{for} \ P \in \Z[U].
\end{align}
It is straightforward to see that $P-s_k(P)$ is divisible by $-z_{k \, k-1} +2 z_{k+1 \, k} - z_{k+2 \, k+1}$.
Thus, $\partial_k(P)$ is also a polynomial on $\Z[U]$. 

\begin{lemma} \label{lemma:divided_difference_zij}
For any $1 \leq j < i \leq n$, one has
\begin{align*}
\partial_{k}(z_{ij}) = \begin{cases}
z_{i \, j+1} \ \ \ &\textrm{if} \ k=j; \\
0 \ \ \ &\textrm{if} \ k \neq j 
\end{cases}
\end{align*}
in the coordinate ring $\Z[U]$. 
\end{lemma}

\begin{proof}
Since we have
\begin{align*}
\partial_{k}(z_{ij}) = \frac{z_{ij}-s_k(z_{ij})}{-z_{k \, k-1} +2 z_{k+1 \, k} - z_{k+2 \, k+1}},
\end{align*}
the result follows from Definition~\ref{definition:Sn_action_Z[U]}.
\end{proof}

\begin{proposition} \label{proposition:divided_difference_nuij}
Let $\nu_{i,j}$ be the polynomial on $\Z[U]$ defined in \eqref{eq:nuij}. 
\begin{enumerate}
\item[(1)] For any $k, j \in [n-1]$, we have 
\begin{align*}
\partial_k(\nu_{j+1,j}) = 0.
\end{align*}
\item[(2)] If $i > j+1$, then we have
\begin{align*}
\partial_k(\nu_{i,j}) = 
\begin{cases}
-\nu_{i-1,j} \ \ \ &\textrm{if} \ k = i-1; \\
\nu_{i,j+1} \ \ \ &\textrm{if} \ k = j; \\
0 \ \ \ &\textrm{otherwise}.
\end{cases}
\end{align*}
\end{enumerate}
\end{proposition}

\begin{proof}
We first note that
\begin{align} \label{eq:nu_i,j_determinant} 
\nu_{i,j} = \left|
 \begin{array}{@{\,}ccccc@{\,}}
     1 & 0 & \cdots & 0 & 1 \\
     z_{j\,j-1} & 1 & \ddots & \vdots & z_{j+1\,j} \\ 
     z_{j+1\,j-1} & z_{j+1\,j} & \ddots & 0 & \vdots \\ 
     \vdots& \vdots & \ddots & 1 & \vdots \\
     z_{i\,j-1} & z_{ij} & \cdots & z_{i \, i-1} & z_{i+1 \, j} 
 \end{array}
 \right|  
\end{align}
for any $1 \leq j < i \leq n$ by \cite[Equations~(3.3) and (3.4)]{HorShi}. 

(1) Suppose that $i = j+1$. 
By \eqref{eq:nu_i,j_determinant} one has 
\begin{align} \label{eq:nu_j+1,j_determinant}  
\nu_{j+1,j} = \left|
 \begin{array}{@{\,}ccc@{\,}}
     1 & 0 & 1 \\
     z_{j\,j-1} & 1 & z_{j+1\,j} \\ 
     z_{j+1\,j-1} & z_{j+1\,j} & z_{j+2 \, j} 
 \end{array}
 \right|. 
\end{align}
It follows from Lemma~\ref{lemma:divided_difference_zij} that $\partial_k(\nu_{j+1,j}) = 0$ when $k \neq j-1, j$. 
Consider the case $k = j-1$. 
Apply the cofactor expansion along the first column in \eqref{eq:nu_j+1,j_determinant}, and then we compute $\partial_{j-1}(\nu_{j+1,j})$ by using Lemma~\ref{lemma:divided_difference_zij} as follows: 
\begin{align*}
\partial_{j-1}(\nu_{j+1,j}) = \left|
 \begin{array}{@{\,}ccc@{\,}}
     \partial_{j-1}(1) & 0 & 1 \\
     \partial_{j-1}(z_{j\,j-1}) & 1 & z_{j+1\,j} \\ 
     \partial_{j-1}(z_{j+1\,j-1}) & z_{j+1\,j} & z_{j+2 \, j} 
 \end{array}
 \right| = \left|
 \begin{array}{@{\,}ccc@{\,}}
     0 & 0 & 1 \\
     1 & 1 & z_{j+1\,j} \\ 
     z_{j+1\,j} & z_{j+1\,j} & z_{j+2 \, j} 
 \end{array}
 \right| = 0. 
\end{align*}
If $k=j$, then we note that 
\begin{align*} 
\nu_{j+1,j} = z_{j+2\,j} - z_{j+1\,j}z_{j+1\,j} + \left|
 \begin{array}{@{\,}cc@{\,}}
     z_{j\,j-1} & 1 \\ 
     z_{j+1\,j-1} & z_{j+1\,j}  
 \end{array}
 \right|
\end{align*}
by the cofactor expansion along the last column in \eqref{eq:nu_j+1,j_determinant}. 
Hence, it follows from Lemma~\ref{lemma:divided_difference_zij} that 
\begin{align*} 
\partial_j(\nu_{j+1,j}) &= z_{j+2\,j+1} - \big(s_j(z_{j+1\,j}) \cdot 1 + 1 \cdot z_{j+1\,j} \big) + \left|
 \begin{array}{@{\,}cc@{\,}}
     z_{j\,j-1} & \partial_{j}(1) \\ 
     z_{j+1\,j-1} & \partial_{j}(z_{j+1\,j})  
 \end{array}
 \right| \\  
 &= z_{j+2\,j+1} - (z_{j\,j-1} -z_{j+1\,j} +z_{j+2\,j+1} + z_{j+1\,j} ) + \left|
 \begin{array}{@{\,}cc@{\,}}
     z_{j\,j-1} & 0 \\ 
     z_{j+1\,j-1} & 1  
 \end{array}
 \right| \\  
 &= 0.
\end{align*}

(2) Suppose that $i>j+1$. 
One can easily see from the formula \eqref{eq:nu_i,j_determinant} that $\partial_k(\nu_{i,j})=0$ unless $j-1 \leq k \leq i-1$.  
If $j-1 \leq k \leq i-1$ and $k \neq i-1, j$, then the $(k-j+2)$-th column of $\nu_{i,j}$ in \eqref{eq:nu_i,j_determinant} is given by $(0,\dots,0,1,z_{k+1 \, k},z_{k+2 \, k}, \dots, z_{ik})^t$. 
By a similar argument to (1), $\partial_k(\nu_{i,j})$ is obtained by applying $\partial_k$ to the $(k-j+2)$-th column of $\nu_{i,j}$.
Applying $\partial_k$ to the $(k-j+2)$-th column, the result coincides with the $(k-j+3)$-th column of $\nu_{i,j}$ in \eqref{eq:nu_i,j_determinant} by Lemma~\ref{lemma:divided_difference_zij}.
Hence, we have $\partial_k(\nu_{i,j}) =0$ in this case.
If $k=i-1$, then $\partial_{i-1}(\nu_{i,j})$ is obtained by applying $\partial_{i-1}$ to the $(i-j+1)$-th column of $\nu_{i,j}$ in \eqref{eq:nu_i,j_determinant}.
By Lemma~\ref{lemma:divided_difference_zij} $\partial_{i-1}(\nu_{i,j})$ is computed as 
\begin{align*} 
 \left|
 \begin{array}{@{\,}cccccc@{\,}}
     1 & 0 & \cdots & 0 & 0 & 1 \\
     z_{j\,j-1} & 1 & \ddots & \vdots & \vdots & z_{j+1\,j} \\ 
     z_{j+1\,j-1} & z_{j+1\,j} & \ddots & 0 & \vdots & \vdots \\ 
     \vdots& \vdots & \ddots & 1 & 0 & \vdots \\
     z_{i-1\,j-1} & z_{i-1 \, j} & \cdots & z_{i-1 \, i-2} & 0 & z_{i \, j} \\
     z_{i\,j-1} & z_{ij} & \cdots & z_{i \, i-2} & 1 & z_{i+1 \, j} 
 \end{array}
 \right| = -\left|
 \begin{array}{@{\,}cccccc@{\,}}
     1 & 0 & \cdots & 0 & 1 \\
     z_{j\,j-1} & 1 & \ddots & \vdots & z_{j+1\,j} \\ 
     z_{j+1\,j-1} & z_{j+1\,j} & \ddots & 0 & \vdots \\ 
     \vdots& \vdots & \ddots & 1 & \vdots \\
     z_{i-1\,j-1} & z_{i-1 \, j} & \cdots & z_{i-1 \, i-2} & z_{i \, j} 
 \end{array}
 \right|,
\end{align*}
which is $-\nu_{i-1,j}$ as desired.
Lastly, we consider the case when $k=j$.
By the cofactor expansion along the last column in \eqref{eq:nu_i,j_determinant}, we obtain
\begin{align*}
\nu_{i,j} = z_{i+1\,j} + \sum_{k=1}^{i-j+1} (-1)^k z_{i-k+1\,j} \left|
 \begin{array}{@{\,}ccccc@{\,}}
     z_{i-k+1\,i-k} & 1 & 0 & \cdots & 0 \\
     z_{i-k+2\,i-k} & z_{i-k+2\,i-k+1} & 1 & \ddots & \vdots \\ 
     \vdots & \vdots & \ddots & \ddots & 0 \\ 
     \vdots & \vdots &  & \ddots & 1 \\
     z_{i\,i-k} & z_{i\,i-k+1} & \cdots & \cdots & z_{i \, i-1} 
 \end{array}
 \right|.
\end{align*}
Applying the divided difference operator $\partial_j$ to the determinants appeared in the right hand side above, 
the results are $0$ by a similar argument above with $i-1>j$.
Hence, we obtain 
\begin{align*}
\partial_j(\nu_{i,j}) &= z_{i+1\,j+1} + \sum_{k=1}^{i-j} (-1)^k z_{i-k+1\,j+1} \left|
 \begin{array}{@{\,}ccccc@{\,}}
     z_{i-k+1\,i-k} & 1 & 0 & \cdots & 0 \\
     z_{i-k+2\,i-k} & z_{i-k+2\,i-k+1} & 1 & \ddots & \vdots \\ 
     \vdots & \vdots & \ddots & \ddots & 0 \\ 
     \vdots & \vdots &  & \ddots & 1 \\
     z_{i\,i-k} & z_{i\,i-k+1} & \cdots & \cdots & z_{i \, i-1} 
 \end{array}
 \right| 
= \nu_{i,j+1} 
\end{align*}
by Lemma~\ref{lemma:divided_difference_zij}.
This completes the proof.
\end{proof}

Since we have 
\begin{align*} 
\xi_{i,j} = (-1)^{i-j+1}\left|
 \begin{array}{@{\,}ccccc@{\,}}
     z_{j+1\,j} & 1 & 0 & \cdots & 0 \\
     2z_{j+2\,j} & z_{j+2\,j+1} & 1 & \ddots & \vdots \\ 
     3z_{j+3\,j} & z_{j+3\,j+1} & z_{j+3\,j+2} & \ddots & 0 \\ 
     \vdots& \vdots & \vdots & \ddots & 1 \\
     (i-j)z_{ij} & z_{i\,j+1} & z_{i\,j+2} & \cdots & z_{i \, i-1} 
 \end{array}
 \right| 
\end{align*}
by \cite[Equation~(7.2)]{Hor25}, the following proposition can be proved by a similar argument to Proposition~\ref{proposition:divided_difference_nuij}. 
For brevity, we omit the details.

\begin{proposition}
Let $\xi_{i,j}$ be the polynomial on $\Z[U]$ defined in \eqref{eq:xiij}. 
\begin{enumerate}
\item[(1)] For arbitrary $k, j \in [n-1]$, it holds that
\begin{align*}
\partial_k(\xi_{j+1,j}) = \delta_{kj}
\end{align*}
where $\delta_{kj}$denotes the Kronecker delta. 
\item[(2)] For $i > j+1$, we have
\begin{align*}
\partial_k(\xi_{i,j}) = 
\begin{cases}
-\xi_{i-1,j} \ \ \ &\textrm{if} \ k = i-1; \\
\xi_{i,j+1} \ \ \ &\textrm{if} \ k = j; \\
0 \ \ \ &\textrm{otherwise}.
\end{cases}
\end{align*}
\end{enumerate}
\end{proposition}

\section{A symmetric group action on $\Z[x_1,\ldots,x_n,q_{ij} \mid 1 \leq i < j \leq n]$} \label{sect:a symmetric group action on qij}

As an application of Theorem~\ref{theorem:iso}, we construct an action of the symmetric group $S_n$ on the polynomial ring $\Z[x_1,\ldots,x_n,q_{ij} \mid 1 \leq i < j \leq n]$. 
Namely, we use the following isomorphism of graded rings 
\begin{align*} 
\varphi: \Z[U] \xrightarrow{\cong} Q_n=\Z[x_1,\ldots,x_n,q_{ij} \mid 1 \leq i < j \leq n]/(E_1^{(n)},\ldots,E_n^{(n)}); \ \ \ z_{ij} \mapsto H_{i-j}^{(j)}.
\end{align*}
In Section~\ref{sect:a symmetric group action on ZU} we constructed the divided difference operator $\partial_k$ on the coordinate ring $\Z[U]$, so we obtain the divided difference operator $\partial_k$ on the quotient ring $Q_n$ under the isomorphism $\varphi$.
Recalling from Theorem~\ref{theorem:iso} that $\varphi$ maps $(-1)^{j-i}\nu_{j,i}$ to the quantum parameter $q_{ij}$, we can translate Proposition~\ref{proposition:divided_difference_nuij} to the following formula:
\begin{align*}
\partial_{k}(q_{i\, i+1}) &= 0 \ \ \ \textrm{for any} \ k \in [n-1]; \\
\textrm{if} \ j> i+1, \ \textrm{then} \ \partial_{k}(q_{ij}) &= \begin{cases}
q_{i \, j-1} \ \ \ &\textrm{if} \ k=j-1; \\
-q_{i+1 \, j} \ \ \ &\textrm{if} \ k=i; \\
0 \ \ \ &\textrm{otherwise}  
\end{cases}
\end{align*}
in the quotient ring $Q_n$.
Motivated by the formula above, we introduce an $S_n$-action on the polynomial ring $\Z[x_1,\ldots,x_n,q_{ij} \mid 1 \leq i < j \leq n]$.
For the rest of paper, we take the convention that 
\begin{align*}
q_{ij} = 0 \ \textrm{unless} \ 1 \leq i < j \leq n
\end{align*}
in the polynomial ring $\Z[x_1,\ldots,x_n,q_{ij} \mid 1 \leq i < j \leq n]$.

\begin{definition} \label{definition:Sn_action_qij}
Let $1 \leq i < j \leq n$.
For each $k \in [n-1]$, we define 
\begin{align*}
s_k(q_{ij}) \coloneqq 
\begin{cases}
q_{ij}-q_{i\,j-1}(x_{j-1}-x_j) \ \ \ &\textrm{if} \ k = j-1; \\ 
q_{ij}+q_{i+1\,j}(x_i-x_{i+1}) \ \ \ &\textrm{if} \ k = i; \\
q_{ij} \ \ \ &\textrm{otherwise}  
\end{cases}
\end{align*}
in the polynomial ring $\Z[x_1,\ldots,x_n,q_{ij} \mid 1 \leq i < j \leq n]$.
In particular, we have 
\begin{align*}
s_k(q_{i \, i+1})=q_{i \, i+1}
\end{align*} 
for any $i, k \in [n-1]$. 
\end{definition}

\begin{lemma} \label{lemma:Sn_action_qij}
We have the following identities:
\begin{align*}
&{\rm (i)} \, s_k^2(q_{ij}) = q_{ij}; \\
&{\rm (ii)} \, s_k s_\ell (q_{ij}) = s_\ell s_k (q_{ij}) \ \ \ \textrm{if} \ |k-\ell|>1; \\
&{\rm (iii)} \, s_k s_{k+1} s_k (q_{ij}) = s_{k+1} s_k s_{k+1} (q_{ij}). 
\end{align*}
\end{lemma}

\begin{proof}
(i) If $k \neq j-1, i$, then one has $s_k(q_{ij}) = q_{ij}$, so $s_k^2(q_{ij}) = q_{ij}$.
If $k = j-1$, then we have 
\begin{align*}
s_{j-1}^2(q_{ij}) = s_{j-1} \big(q_{ij}-q_{i\,j-1}(x_{j-1}-x_j) \big) 
= q_{ij}-q_{i\,j-1}(x_{j-1}-x_j) -q_{i\,j-1}(x_j-x_{j-1}) 
= q_{ij}. 
\end{align*}
If $k = i$, then one has 
\begin{align*}
s_i^2(q_{ij}) = s_i \big(q_{ij}+q_{i+1\,j}(x_i-x_{i+1}) \big) 
= q_{ij}+q_{i+1\,j}(x_i-x_{i+1}) +q_{i+1\,j}(x_{i+1}-x_i) 
= q_{ij}. 
\end{align*}

\noindent
(ii) Without loss of generality, we may assume that $k < \ell$. 
If both $k$ and $\ell$ are neither $i$ nor $j-1$, then one has $s_k s_\ell (q_{ij}) = s_\ell s_k (q_{ij})$ since $s_k(q_{ij}) = s_\ell(q_{ij})=q_{ij}$.
We take cases.

\noindent
\textbf{Case~(ii-a):} Suppose that $(k,\ell)=(i,j-1)$. 
Note that $i < j-2$ since $|k-\ell|>1$.
Then we obtain
\begin{align*}
s_i s_{j-1} (q_{ij}) &= s_i \big(q_{ij}-q_{i\,j-1}(x_{j-1}-x_j)\big) \\
&= q_{ij}+q_{i+1\,j}(x_i-x_{i+1})-\big(q_{i\,j-1}+q_{i+1\,j-1}(x_i-x_{i+1})\big)(x_{j-1}-x_j); \\
s_{j-1} s_i (q_{ij}) &= s_{j-1}\big(q_{ij}+q_{i+1\,j}(x_i-x_{i+1})\big) \\
&= q_{ij}-q_{i\,j-1}(x_{j-1}-x_j) + \big(q_{i+1\,j}-q_{i+1\,j-1}(x_{j-1}-x_j)\big)(x_i-x_{i+1}). 
\end{align*}

\noindent
\textbf{Case~(ii-b):} Suppose that $k \neq i$ and $\ell=j-1$. Note that $k < j-2$ and $k \neq i$. 
Then one has 
\begin{align*}
s_k s_{j-1} (q_{ij}) &= s_k \big(q_{ij}-q_{i\,j-1}(x_{j-1}-x_j)\big) = q_{ij}-q_{i\,j-1}(x_{j-1}-x_j); \\
s_{j-1} s_k (q_{ij}) &= s_{j-1}(q_{ij}) = q_{ij}-q_{i\,j-1}(x_{j-1}-x_j). 
\end{align*}

\noindent
\textbf{Case~(ii-c):} Suppose that $k = i$ and $\ell \neq j-1$. Note that $\ell > i+1$ and $\ell \neq j-1$. 
Then we have
\begin{align*}
s_i s_{\ell} (q_{ij}) &= s_i(q_{ij}) = q_{ij}+q_{i+1\,j}(x_i-x_{i+1}); \\
s_{\ell} s_i (q_{ij}) &= s_{\ell}\big(q_{ij}+q_{i+1\,j}(x_i-x_{i+1})\big) = q_{ij}+q_{i+1\,j}(x_i-x_{i+1}). 
\end{align*}

\noindent
(iii) It is enough to prove the four cases when $k=j-2$, $k=j-1$, $k=i-1$, and $k=i$. 

\noindent
\textbf{Case~(iii-a):} Suppose that $k=j-2$. 
Without loss of generality, we may assume that $i \neq j-2$ since the case when $i = j-2$ is covered in Case~(iii-d) below. 
We calculate as follows:
\begin{align*}
s_{j-2} s_{j-1} s_{j-2} (q_{ij}) &= s_{j-2} s_{j-1} (q_{ij}) = s_{j-2} \big( q_{ij}-q_{i\,j-1}(x_{j-1}-x_j) \big) \\
&= q_{ij}-\big(q_{i\,j-1}-q_{i\,j-2}(x_{j-2}-x_{j-1})\big)(x_{j-2}-x_j) ; \\
s_{j-1} s_{j-2} s_{j-1} (q_{ij}) &= s_{j-1} s_{j-2} \big( q_{ij}-q_{i\,j-1}(x_{j-1}-x_j) \big) \\
&= s_{j-1} \big( q_{ij}-\big(q_{i\,j-1}-q_{i\,j-2}(x_{j-2}-x_{j-1})\big)(x_{j-2}-x_j) \big) \\
&= q_{ij}-q_{i\,j-1}(x_{j-1}-x_j) - \big(q_{i\,j-1}-q_{i\,j-2}(x_{j-2}-x_j)\big)(x_{j-2}-x_{j-1}) \\
&= q_{ij}-q_{i\,j-1}(x_{j-2}-x_j) +q_{i\,j-2}(x_{j-2}-x_{j-1})(x_{j-2}-x_j). 
\end{align*}

\noindent
\textbf{Case~(iii-b):} Suppose that $k=j-1$. 
Then one has
\begin{align*}
s_{j-1} s_j s_{j-1} (q_{ij}) &= s_{j-1} s_j \big(q_{ij}-q_{i\,j-1}(x_{j-1}-x_j)\big) 
= s_{j-1} \big( q_{ij}-q_{i\,j-1}(x_{j-1}-x_{j+1}) \big) \\
&= q_{ij}-q_{i\,j-1}(x_{j-1}-x_j)-q_{i\,j-1}(x_j-x_{j+1}) = q_{ij}-q_{i\,j-1}(x_{j-1}-x_{j+1}); \\
s_j s_{j-1} s_j (q_{ij}) & = s_j s_{j-1}(q_{ij}) = s_j \big( q_{ij}-q_{i\,j-1}(x_{j-1}-x_j) \big) 
= q_{ij}-q_{i\,j-1}(x_{j-1}-x_{j+1}).
\end{align*}

\noindent
\textbf{Case~(iii-c):} Suppose that $k=i-1$. 
Then we have
\begin{align*}
s_{i-1} s_i s_{i-1} (q_{ij}) &= s_{i-1} s_i (q_{ij}) = s_{i-1} \big( q_{ij}+q_{i+1\,j}(x_i-x_{i+1}) \big) 
= q_{ij}+q_{i+1\,j}(x_{i-1}-x_{i+1}); \\
s_i s_{i-1} s_i (q_{ij}) & = s_i s_{i-1}\big( q_{ij}+q_{i+1\,j}(x_i-x_{i+1}) \big) = s_i \big( q_{ij}+q_{i+1\,j}(x_{i-1}-x_{i+1}) \big) \\
&= q_{ij}+q_{i+1\,j}(x_i-x_{i+1}) + q_{i+1\,j}(x_{i-1}-x_i) 
= q_{ij}+q_{i+1\,j}(x_{i-1}-x_{i+1}). 
\end{align*}

\noindent
\textbf{Case~(iii-d):} Suppose that $k=i$. 
Without loss of generality, we may assume that $j \neq i+2$ since the case when $j = i+2$ is covered in Case~(iii-a). 
We compute as follows:
\begin{align*}
s_i s_{i+1} s_i (q_{ij}) &= s_i s_{i+1} \big(q_{ij}+q_{i+1\,j}(x_i-x_{i+1})\big) \\
&= s_i \big(q_{ij}+\big(q_{i+1\,j}+q_{i+2\,j}(x_{i+1}-x_{i+2})\big)(x_i-x_{i+2})\big) \\
&= q_{ij}+q_{i+1\,j}(x_i-x_{i+1}) + \big(q_{i+1\,j}+q_{i+2\,j}(x_i-x_{i+2})\big)(x_{i+1}-x_{i+2}) \\
&= q_{ij}+q_{i+1\,j}(x_i-x_{i+2}) + q_{i+2\,j}(x_i-x_{i+2})(x_{i+1}-x_{i+2}); \\
s_{i+1} s_i s_{i+1} (q_{ij}) &= s_{i+1} s_i (q_{ij}) = s_{i+1} \big(q_{ij}+q_{i+1\,j}(x_i-x_{i+1})\big) \\
&= q_{ij}+\big(q_{i+1\,j}+q_{i+2\,j}(x_{i+1}-x_{i+2})\big)(x_i-x_{i+2}). 
\end{align*}
This completes the proof. 
\end{proof}

For $w = s_{k_1} s_{k_2} \dots s_{k_r} \in S_n$ (not necessarily reduced), we define
\begin{align*} 
w(q_{ij}) = s_{k_1} s_{k_2} \dots s_{k_r} (q_{ij}),
\end{align*}
which is well-defined by Lemma~\ref{lemma:Sn_action_qij}. 
This with \eqref{eq:Sn_action_xi} yields an action of $S_n$ on the polynomial ring $\Z[x_1,\ldots,x_n,q_{ij} \mid 1 \leq i < j \leq n]$.
Now, we define the divided difference operator $\partial_k \ (k \in [n-1])$ on the polynomial ring $\Z[x_1,\ldots,x_n,q_{ij} \mid 1 \leq i < j \leq n]$ by
\begin{align*} 
\partial_k(F) = \frac{F-s_k(F)}{x_k-x_{k+1}} 
\end{align*}
for $F \in \Z[x_1,\ldots,x_n,q_{ij} \mid 1 \leq i < j \leq n]$.
By the definition of the $S_n$-action on $\Z[x_1,\ldots,x_n,q_{ij} \mid 1 \leq i < j \leq n]$, the numerator $F-s_k(F)$ is divisible by $x_k-x_{k+1}$, and hence $\partial_k(F)$ is a polynomial on $\Z[x_1,\ldots,x_n,q_{ij} \mid 1 \leq i < j \leq n]$. 

\begin{lemma} \label{lemma:divided_difference_qij}
For all $1 \leq i < j \leq n$, we have
\begin{align*}
\partial_{k}(q_{ij}) = \begin{cases}
q_{i \, j-1} \ \ \ &\textrm{if} \ k=j-1; \\
-q_{i+1 \, j} \ \ \ &\textrm{if} \ k=i; \\
0 \ \ \ &\textrm{otherwise} 
\end{cases}
\end{align*}
in the polynomial ring $\Z[x_1,\ldots,x_n,q_{ij} \mid 1 \leq i < j \leq n]$. 
In particular, we obtain
\begin{align*}
\partial_k(q_{i \, i+1})=0
\end{align*} 
for any $i, k \in [n-1]$. 
\end{lemma}

\begin{proof}
By definition one has 
\begin{align*}
\partial_{k}(q_{ij}) = \frac{q_{ij}-s_k(q_{ij})}{x_k-x_{k+1}}.
\end{align*}
The result follows from Definition~\ref{definition:Sn_action_qij}.
\end{proof}

The following equality is a $\q$-analogue of \cite[Lemma~3.1]{FGP}. 

\begin{proposition} \label{proposition:divided_difference_E}
For any $1 \leq i \leq j \leq n$, we have 
\begin{align*}
\partial_k \big(E_i^{(j)}\big) = \begin{cases}
E_{i-1}^{(j-1)} \ \ \ &\textrm{if} \ k = j; \\
0 \ \ \ &\textrm{if} \ k \neq j
\end{cases}
\end{align*}
in the polynomial ring $\Z[x_1,\ldots,x_n,q_{ij} \mid 1 \leq i < j \leq n]$.
Moreover, $\partial_k$ commutes with multiplication by $E_i^{(j)}$ whenever $k \neq j$.
\end{proposition}

\begin{proof}
We prove this by induction on $j$.
The base case $j=1$ is clear since $E_1^{(1)} = x_1$. 
Suppose that $j>1$ and that the claim holds for arbitrary $j' \leq j-1$, with any allowable choices of $i'$ with $1 \leq i' \leq j'$.
We use the recursive formula in \eqref{eq:recursive quantized elementary symmetric polynomials}, i.e. 
\begin{align*} 
E_i^{(j)} = E_i^{(j-1)} + E_{i-1}^{(j-1)} x_j + \sum_{\ell=1}^{i-1} E_{i-1-\ell}^{(j-1-\ell)} q_{j-\ell \, j}. 
\end{align*}
It follows from the inductive hypothesis and Lemma~\ref{lemma:divided_difference_qij} that $\partial_j(E_i^{(j)}) = E_{i-1}^{(j-1)}$. 
In what follows, we will prove that $\partial_k(E_i^{(j)}) = 0$ whenever $k \neq j$.
It is clear that $\partial_k(E_i^{(j)}) = 0$ if $k > j$ from the inductive assumption. 
If $k < j-1$, then we have 
\begin{align*} 
\partial_k\big(E_i^{(j)}\big) = \partial_k\big(E_{i-j+k}^{(k)}\big) q_{k+1 \, j} + E_{i-j+k-1}^{(k-1)} \partial_k(q_{kj}) = E_{i-j+k-1}^{(k-1)} q_{k+1 \, j} + E_{i-j+k-1}^{(k-1)} (-q_{k+1\,j}) = 0 
\end{align*}
by the inductive hypothesis and Lemma~\ref{lemma:divided_difference_qij}.
If $k = j-1$, then one has 
\begin{align*} 
\partial_{j-1}\big(E_i^{(j)}\big) &= E_{i-1}^{(j-2)} + \partial_{j-1}\big(E_{i-1}^{(j-1)} x_j \big) + \sum_{\ell=1}^{i-1} E_{i-1-\ell}^{(j-1-\ell)} \partial_{j-1}(q_{j-\ell \, j}) \\
&= E_{i-1}^{(j-2)} + E_{i-2}^{(j-2)} x_{j-1} - E_{i-1}^{(j-1)} + \sum_{\ell=2}^{i-1} E_{i-1-\ell}^{(j-1-\ell)} q_{j-\ell \, j-1} \\
& \hspace{30pt} \textrm{(by ``Leibniz formula'' in \eqref{eq:Leibniz formula})} \\
&= E_{i-1}^{(j-2)} + E_{i-2}^{(j-2)} x_{j-1} - E_{i-1}^{(j-1)} + \sum_{\ell=1}^{i-2} E_{i-2-\ell}^{(j-2-\ell)} q_{j-\ell-1 \, j-1}.
\end{align*}
The recursive formula $E_{i-1}^{(j-1)} = E_{i-1}^{(j-2)} + E_{i-2}^{(j-2)} x_{j-1} + \sum_{\ell=1}^{i-2} E_{i-2-\ell}^{(j-2-\ell)} q_{j-\ell-1 \, j-1}$ yields that $\partial_{j-1}\big(E_i^{(j)}\big) = 0$ as desired.

Finally, if $k \neq j$, then we have
\begin{align*}
\partial_k(F \cdot E_i^{(j)}) = \partial_k(F) \cdot E_i^{(j)} + s_k(F) \cdot \partial_k \big(E_i^{(j)}\big) = \partial_k(F) \cdot E_i^{(j)}
\end{align*}
for any polynomial $F \in \Z[x_1,\ldots,x_n,q_{ij} \mid 1 \leq i < j \leq n]$ and we are done. 
\end{proof}

By a similar argument to Proposition~\ref{proposition:divided_difference_nuij} with \eqref{eq:determinant_HijEij}, one can prove the following result. 
So, the details are omitted for the sake of brevity.

\begin{proposition} \label{proposition:divided_difference_H}
For any $i, j \geq 0$, one has 
\begin{align*}
\partial_{k}\big(H_{i}^{(j)}\big) = \begin{cases}
H_{i-1}^{(j+1)} \ \ \ &\textrm{if} \ k=j; \\
0 \ \ \ &\textrm{if} \ k \neq j
\end{cases}
\end{align*}
in the polynomial ring $\Z[x_1,\ldots,x_n,q_{ij} \mid 1 \leq i < j \leq n]$.
\end{proposition}

The following result is a $\q$-analogue of \cite[Lemma~6.1]{Hor25}. 

\begin{proposition} \label{proposition:Fij_divided_difference_operator} 
Let $F_{i,j}$ be the $\q$-quantization of the polynomial $f_{i,j}$ in \eqref{eq:fij}. 
For any $i > j \geq 1$, we have
\begin{align*}
\partial_k (F_{i,j})=\begin{cases}
-F_{i-1,j} \ & {\rm if} \ k=i; \\
F_{i,j+1} \ & {\rm if} \ k=j; \\
0 \ & {\rm otherwise} 
\end{cases}
\end{align*}
in the polynomial ring $\Z[x_1,\ldots,x_n,q_{ij} \mid 1 \leq i < j \leq n]$.
\end{proposition}

\begin{proof}
By Lemma~\ref{lemma:FEH} we have 
\begin{align*}
F_{i,j} = \sum_{\ell=0}^{i-j+1} (-1)^\ell (i-\ell) E_\ell^{(i)} H_{i-j+1-\ell}^{(j)}.
\end{align*}
The result follows from Propositions~\ref{proposition:divided_difference_E} and \ref{proposition:divided_difference_H}. 
\end{proof}

\section{$\q$-quantum Schubert polynomials} \label{sect:quantum Schubert polynomials}

In this section we begin with the definition of Schubert polynomials.
We write $w_0$ for the longest element in $S_n$, i.e. $w_0(i) = n+1-i$ for all $i \in [n]$.
The Schubert polynomial $\SS_w=\SS_w(x_1,\ldots,x_n)$ associated with a permutation $w \in S_n$ is recursively defined as follows. 
For the longest element $w_0 \in S_n$, we define 
\begin{align} \label{eq:Schubert_longest}
\SS_{w_0} = x_1^{n-1}x_2^{n-2} \cdots x_{n-1} = e_1^{(1)}e_2^{(2)} \cdots e_{n-1}^{(n-1)}.
\end{align}
In general, for a permutation $v \in S_n$ with $v \neq w_0$, there exists $k \in [n-1]$ such that $v(k) < v(k+1)$.
By setting $w = v s_k$, the Schubert polynomial is inductively defined by
\begin{align} \label{eq:divided_difference_Schubert}
\SS_v = \partial_k (\SS_w).
\end{align}
In other words, if we write $w=w_0 s_{i_1}s_{i_2} \dots s_{i_m}$ with $\ell(w_0s_{i_1}s_{i_2}\dots s_{i_p}) = \ell(w_0)-p$ for all $p \in [m]$, then we have 
\begin{align*}
\SS_w = \partial_{i_m} \circ \dots \circ \partial_{i_2} \circ \partial_{i_1} (\SS_{w_0}) = \partial_{w^{-1}w_0}(\SS_{w_0}).
\end{align*}
Note that $\partial_k (\SS_v) = 0$ if $v(k) < v(k+1)$ since $\partial_k^2 = 0$.
We here call the $\q$-quantization of the Schubert polynomial $\SS_w$ the \emph{$\q$-quantum Schubert polynomial}, denoted by $\SS_w^\q$.

\begin{remark}
Fomin--Gelfand--Postnikov originally introduced in \cite{FGP} the quantum Schubert polynomial which is the specialization of $\SS_w^\q$ in setting $q_{ij} = 0$ whenever $j-i >1$ and $q_{i \, i+1} = q_i$ for each $i \in [n-1]$. 
\end{remark}

By definition we have 
\begin{align} \label{eq:quantumSchubert_longest}
\SS_{w_0}^\q = E_1^{(1)}E_2^{(2)} \cdots E_{n-1}^{(n-1)}.
\end{align}
We study a $\q$-analogue of \eqref{eq:divided_difference_Schubert} below.

\begin{lemma}
Let $m \in [n-1]$ and $w_m=w_0 s_{1}s_{2} \dots s_{m}$, i.e. 
\begin{align*}
w_m =  n-1 \ n-2 \ \cdots \ n-m \ n \ n-m-1 \ n-m-2 \ \cdots \ 1
\end{align*} 
in one-line notation.
Then we have 
\begin{align*}
\SS_{w_m}^\q = \partial_{m} \circ \dots \circ \partial_{2} \circ \partial_{1} (\SS_{w_0}^\q) = E_1^{(1)}E_2^{(2)} \cdots E_{m-1}^{(m-1)} E_{m+1}^{(m+1)} E_{m+2}^{(m+2)} \cdots E_{n-1}^{(n-1)}. 
\end{align*}
\end{lemma}

\begin{proof}
It follows from \cite[Lemma~3.1]{FGP} with \eqref{eq:Schubert_longest} and \eqref{eq:divided_difference_Schubert} that 
\begin{align}
\SS_{w_m} = \partial_{m} \circ \dots \circ \partial_{2} \circ \partial_{1} (\SS_{w_0}) = e_1^{(1)}e_2^{(2)} \cdots e_{m-1}^{(m-1)} e_{m+1}^{(m+1)} e_{m+2}^{(m+2)} \cdots e_{n-1}^{(n-1)}. 
\end{align}
Hence, we have 
\begin{align*}
\SS_{w_m}^\q = E_1^{(1)}E_2^{(2)} \cdots E_{m-1}^{(m-1)} E_{m+1}^{(m+1)} E_{m+2}^{(m+2)} \cdots E_{n-1}^{(n-1)}. 
\end{align*}
The right hand side above is exactly $\partial_{m} \circ \dots \circ \partial_{2} \circ \partial_{1} (\SS_{w_0}^\q)$
by Proposition~\ref{proposition:divided_difference_E} with \eqref{eq:quantumSchubert_longest}.
\end{proof}

The following lemma is a $\q$-analogue of \cite[Lemma~3.2]{FGP}. 

\begin{lemma}
For positive integers $i,j,k$, we have the following identities:
\begin{align}
&\big(E_{i}^{(k+1)} - E_{i}^{(k)} \big) E_{j-1}^{(k)} - E_{i-1}^{(k)} \big( E_{j}^{(k+1)} - E_{j}^{(k)} \big)  
= \sum_{p=1}^k q_{k+1-p \, k+1} \big( E_{i-1-p}^{(k-p)} E_{j-1}^{(k)} - E_{i-1}^{(k)} E_{j-1-p}^{(k-p)} \big); \label{eq:propertyE1} \\
&E_i^{(k)}E_{j}^{(k)} - \left( E_i^{(k+1)}E_{j}^{(k)} + \sum_{\ell \geq 1} E_{i-\ell}^{(k+1)}E_{j+\ell}^{(k)} - \sum_{\ell \geq 1} E_{i-\ell}^{(k)}E_{j+\ell}^{(k+1)} \right) \label{eq:propertyE2} \\
= & \sum_{p=1}^k q_{k+1-p \, k+1} \left( \sum_{\ell \geq 0} \big( E_{i-\ell-1}^{(k)} E_{j+\ell-p}^{(k-p)} - E_{i-\ell-p-1}^{(k-p)} E_{j+\ell}^{(k)} \big) \right) \notag
\end{align}
in the polynomial ring $\Z[x_1,\ldots,x_n,q_{ij} \mid 1 \leq i < j \leq n]$.
\end{lemma}

\begin{proof}
By using the recursive formula \eqref{eq:recursive quantized elementary symmetric polynomials}, one can write
\begin{align} 
E_{i-1}^{(k)} \big( E_{j}^{(k+1)} - E_{j}^{(k)} \big) &= E_{i-1}^{(k)} \left( E_{j-1}^{(k)} x_{k+1} + \sum_{p=1}^k E_{j-1-p}^{(k-p)} q_{k+1-p \, k+1} \right); \label{eq:propertyE1_proof1} \\ 
E_{j-1}^{(k)} \big( E_{i}^{(k+1)} - E_{i}^{(k)} \big) &= E_{j-1}^{(k)} \left( E_{i-1}^{(k)} x_{k+1} + \sum_{p=1}^k E_{i-1-p}^{(k-p)} q_{k+1-p \, k+1} \right). \label{eq:propertyE1_proof2}
\end{align}
By subtracting \eqref{eq:propertyE1_proof1} from \eqref{eq:propertyE1_proof2}, we obtain \eqref{eq:propertyE1}. 

We next prove \eqref{eq:propertyE2} by induction on $i$. 
The base case is $i=1$. 
In this case we wish to show that
\begin{align*}
\big( E_1^{(k)} - E_1^{(k+1)} \big) E_j^{(k)} - (E_{j+1}^{(k)} - E_{j+1}^{(k+1)}) = \sum_{p=1}^k q_{k+1-p \, k+1} E_{j-p}^{(k-p)}.
\end{align*}
The equality above is immediate from \eqref{eq:propertyE1}. 
Now suppose that $i >1$ and assume by induction that \eqref{eq:propertyE2} is true for $i-1$. 
Then it follows from \eqref{eq:propertyE1} that the left hand side of \eqref{eq:propertyE2} is computed as
\begin{align*}
&-E_{i-1}^{(k)} \big( E_{j+1}^{(k+1)} - E_{j+1}^{(k)} \big) - \sum_{p=1}^k q_{k+1-p \, k+1} \big( E_{i-1-p}^{(k-p)} E_{j}^{(k)} - E_{i-1}^{(k)} E_{j-p}^{(k-p)} \big) \\
&\hspace{10pt} - \sum_{\ell \geq 0} E_{i-\ell-1}^{(k+1)} E_{j+\ell+1}^{(k)} + \sum_{\ell \geq 0} E_{i-\ell-1}^{(k)} E_{j+\ell+1}^{(k+1)} \\
=&E_{i-1}^{(k)}E_{j+1}^{(k)} - \left( E_{i-1}^{(k+1)}E_{j+1}^{(k)} + \sum_{\ell \geq 1} E_{i-1-\ell}^{(k+1)}E_{j+1+\ell}^{(k)} - \sum_{\ell \geq 1} E_{i-1-\ell}^{(k)}E_{j+1+\ell}^{(k+1)} \right) \\
&\hspace{10pt} - \sum_{p=1}^k q_{k+1-p \, k+1} \big( E_{i-1-p}^{(k-p)} E_{j}^{(k)} - E_{i-1}^{(k)} E_{j-p}^{(k-p)} \big). 
\end{align*}
By our inductive assumption on $i$, this equals 
\begin{align*}
&\sum_{p=1}^k q_{k+1-p \, k+1} \left( \sum_{\ell \geq 0} \big( E_{i-\ell-2}^{(k)} E_{j+1+\ell-p}^{(k-p)} - E_{i-\ell-p-2}^{(k-p)} E_{j+1+\ell}^{(k)} \big) \right) \\
&\hspace{10pt} + \sum_{p=1}^k q_{k+1-p \, k+1} \big( E_{i-1}^{(k)} E_{j-p}^{(k-p)} - E_{i-1-p}^{(k-p)} E_{j}^{(k)} \big) \\
=&\sum_{p=1}^k q_{k+1-p \, k+1} \left( \sum_{\ell \geq 0} \big( E_{i-\ell-1}^{(k)} E_{j+\ell-p}^{(k-p)} - E_{i-\ell-p-1}^{(k-p)} E_{j+\ell}^{(k)} \big) \right), 
\end{align*}
which is the right hand side of \eqref{eq:propertyE2} as desired. 
This completes the proof.
\end{proof}

In general, it is \emph{not} true that $\partial_k (\SS_w^\q) = \SS_{ws_k}^\q$ for a permutation $w \in S_n$ with $w(k) > w(k+1)$.
The following example gives a counter example. 

\begin{example}
Let $n=4$ and $w=4312$. 
Then we have
\begin{align*}
\SS_{4312} &= \partial_3 (\SS_{4321}) = \partial_3 \big(e_1^{(1)}e_2^{(2)}e_3^{(3)}\big) =e_1^{(1)}e_2^{(2)}  \partial_3 \big(e_3^{(3)}\big) = e_1^{(1)}e_2^{(2)}e_2^{(2)} \\
&=e_1^{(1)} \big( e_2^{(3)} e_2^{(2)} -e_1^{(2)}e_3^{(3)} \big) = e_{1,2,2} - e_{1,1,3}
\end{align*}
by \cite[Lemmas~3.1 and 3.2]{FGP}.
By definition we obtain 
\begin{align*}
\SS_{4312}^\q = E_{1,2,2} - E_{1,1,3}. 
\end{align*}
On the other hand, we have 
\begin{align*}
E_2^{(2)}E_2^{(2)} = E_2^{(3)} E_2^{(2)} -E_1^{(2)}E_3^{(3)} +q_{23}\big(E_1^{(2)}E_1^{(1)}-E_2^{(2)}\big) +q_{13}E_1^{(2)}
\end{align*}
by \eqref{eq:propertyE2}.
Hence, it follows from Proposition~\ref{proposition:divided_difference_E} that 
\begin{align*}
\partial_3 (\SS_{4321}^\q) &= \partial_3 \big(E_1^{(1)}E_2^{(2)}E_3^{(3)}\big) =E_1^{(1)}E_2^{(2)}  \partial_3 \big(E_3^{(3)}\big) = E_1^{(1)}E_2^{(2)}E_2^{(2)} \\
&=E_1^{(1)} \big( E_2^{(3)} E_2^{(2)} -E_1^{(2)}E_3^{(3)} +q_{23}\big(E_1^{(2)}E_1^{(1)}-E_2^{(2)}\big) +q_{13}E_1^{(2)} \big). \\
&= \SS_{4312}^\q + q_{23} \big(E_1^{(1)}E_1^{(1)}E_1^{(2)} - E_1^{(1)}E_2^{(2)} \big)+q_{13}E_1^{(1)}E_1^{(2)}.
\end{align*}
\end{example}

We can ask what is a polynomial $F_w \in \Z[x_1,\ldots,x_n,q_{ij} \mid 1 \leq i < j \leq n]$ such that $\partial_k(F_w) = \SS_{ws_k}^\q$. 
In what follows, we construct an operator to obtain $F_w$ from $\SS_{w}^\q$ by using a presentation of a $\Z$-linear combination of standard elementary monomials for $\SS_w$. 
The following is a key lemma. 

\begin{lemma} \label{lemma:divided_difference_Leibniz_formula}
For any positive integers $i,j,k$, one has
\begin{align*}
\partial_k\big( s_k \big(E_i^{(k)}\big) \cdot E_j^{(k)} \big) = E_i^{(k)} E_{j-1}^{(k-1)} - E_{i-1}^{(k-1)} E_j^{(k)}. 
\end{align*}
\end{lemma}

\begin{proof}
By using ``Leibniz formula" \eqref{eq:Leibniz formula}, the left hand side equals 
\begin{align*}
\partial_k\big( s_k \big(E_i^{(k)} \big) \big) \cdot E_j^{(k)} + s_ks_k \big(E_i^{(k)} \big) \cdot \partial_k \big( E_j^{(k)} \big) 
&= -\partial_k\big( E_i^{(k)} \big) \cdot E_j^{(k)} + E_i^{(k)} \cdot \partial_k \big( E_j^{(k)} \big) \\
&=- E_{i-1}^{(k-1)} E_j^{(k)} + E_i^{(k)} E_{j-1}^{(k-1)} 
\end{align*}
where we used Proposition~\ref{proposition:divided_difference_E} for the last equality.
\end{proof}

Motivated by Lemma~\ref{lemma:divided_difference_Leibniz_formula}, we introduce the following operator $\tau_k $ for $k \in [n-1]$. 
For a $\q$-quantum standard elementary monomial $E_{i_1,\ldots,i_m} \ (0 \leq i_k \leq k)$, we define  
\begin{align*}
\tau_k(E_{i_1,\ldots,i_m}) = \sum_{\ell \geq 0} E_{i_1}^{(1)} \cdots E_{i_{k-2}}^{(k-2)} \left( s_k\big( E_{i_{k-1}-\ell}^{(k)} \big) E_{i_k+\ell}^{(k)} \right) E_{i_{k+1}}^{(k+1)} \cdots E_{i_m}^{(m)}. 
\end{align*}
In general, we write a polynomial $f \in \Z[x_1,\ldots,x_n]$ as a unique linear combination of standard elementary monomials 
\begin{align*}
f = \sum_{i_1,\ldots,i_m} c_{i_1,\ldots,i_m} e_{i_1,\ldots,i_m} \ \ \ (c_{i_1,\ldots,i_m} \in \Z).
\end{align*}
For the $\q$-quantization $F \in \Z[x_1,\ldots,x_n, q_{ij} \mid 1 \leq i<j \leq n]$ of $f$, we define $\tau_k(F)$ by
\begin{align*}
\tau_k(F) = \sum_{i_1,\ldots,i_m} c_{i_1,\ldots,i_m} \tau_k(E_{i_1,\ldots,i_m}). 
\end{align*}

\begin{example}
We compute two special cases of $\tau_k(E_{i_1, \ldots, i_m})$. 
\begin{enumerate}
\item[(1)] If $i_{k-1} =0$, then we have
\begin{align*}
\tau_k(E_{i_1, \ldots, i_m}) = E_{i_1, \ldots, i_m}. 
\end{align*}
In particular, $\tau_1(E_{i_1, \ldots, i_m}) = E_{i_1, \ldots, i_m}$. 
\item[(2)] If $m=n-1$ and $i_p =p$ for all $p \in [n-1]$, then we have
\begin{align*}
\tau_k \big(\SS_{w_0}^\q \big) = \tau_k \big( E_1^{(1)}E_2^{(2)} \cdots E_{n-1}^{(n-1)} \big) 
= E_1^{(1)} \cdots E_{k-2}^{(k-2)} \cdot s_k\big( E_{k-1}^{(k)} \big) \cdot E_k^{(k)} E_{k+1}^{(k+1)} \cdots E_{n-1}^{(n-1)} 
\end{align*}
with the convention that $\tau_1 \big(\SS_{w_0}^\q \big) = \SS_{w_0}^\q$.
\end{enumerate}
\end{example}

\begin{proposition} \label{proposition:divided_difference_quantumSchubert}
For a permutation $w \in S_n$ with $w(k) > w(k+1)$, we have 
\begin{align*} 
\partial_k (\tau_k (\SS_w^\q)) = \SS_{ws_k}^\q. 
\end{align*}
\end{proposition}

\begin{proof}
If we write the Schubert polynomial $\SS_w$ as a unique linear combination of standard elementary monomials 
\begin{align*}
\SS_w = \sum_{i_1,\ldots,i_{n-1}} c_{i_1,\ldots,i_{n-1}} e_{i_1,\ldots,i_{n-1}} \ \ \ (c_{i_1,\ldots,i_{n-1}} \in \Z),
\end{align*}
then one has 
\begin{align*}
\SS_{ws_k} &=\partial_k(\SS_w) = \sum_{i_1,\ldots,i_{n-1}} c_{i_1,\ldots,i_{n-1}} e_{i_1}^{(1)} \cdots e_{i_{k-1}}^{(k-1)} e_{i_k-1}^{(k-1)}e_{i_{k+1}}^{(k+1)}\cdots e_{i_{n-1}}^{(n-1)}   
\end{align*}
by \eqref{eq:divided_difference_Schubert} and \cite[Lemma~3.1]{FGP}. 
By using \cite[Lemma~3.2]{FGP}, we have 
\begin{align*}
e_{i_{k-1}}^{(k-1)} e_{i_k-1}^{(k-1)} &= e_{i_{k-1}}^{(k)} e_{i_k-1}^{(k-1)} + \sum_{\ell \geq 1} e_{i_{k-1}-\ell}^{(k)} e_{i_k-1+\ell}^{(k-1)} - \sum_{\ell \geq 1} e_{i_{k-1}-\ell}^{(k-1)} e_{i_k-1+\ell}^{(k)} \\
&= \sum_{\ell \geq 0} \left(e_{i_{k-1}-\ell}^{(k)} e_{i_k-1+\ell}^{(k-1)} - e_{i_{k-1}-1-\ell}^{(k-1)} e_{i_k+\ell}^{(k)} \right). 
\end{align*}
By the definition of the $\q$-quantization, the $\q$-quantum Schubert polynomial $\SS_{ws_k}^\q$ equals  
\begin{align*}
\sum_{i_1,\ldots,i_{n-1}} \sum_{\ell \geq 0} c_{i_1,\ldots,i_{n-1}} E_{i_1}^{(1)} \cdots E_{i_{k-2}}^{(k-2)} \left( E_{i_{k-1}-\ell}^{(k)} E_{i_k-1+\ell}^{(k-1)} - E_{i_{k-1}-1-\ell}^{(k-1)} E_{i_k+\ell}^{(k)} \right) E_{i_{k+1}}^{(k+1)}\cdots E_{i_{n-1}}^{(n-1)}.
\end{align*}
Since we have 
\begin{align*}
\partial_k \big( s_k \big( E_{i_{k-1}-\ell}^{(k)} \big) \cdot E_{i_k+\ell}^{(k)} \big) = E_{i_{k-1}-\ell}^{(k)} E_{i_k-1+\ell}^{(k-1)} - E_{i_{k-1}-1-\ell}^{(k-1)} E_{i_k+\ell}^{(k)}
\end{align*}
by Lemma~\ref{lemma:divided_difference_Leibniz_formula}, we conclude that 
\begin{align*} 
\SS_{ws_k}^\q = \partial_k (\tau_k (\SS_w^\q)) 
\end{align*}
from Proposition~\ref{proposition:divided_difference_E}.
\end{proof}

\begin{example}
Consider $n=4$ and $w=4312$. 
Since one has
\begin{align*}
\tau_3(\SS_{4321}^\q) &= E_{1}^{(1)} s_3\big(E_2^{(3)}) E_{3}^{(3)} = E_1^{(1)} s_3(x_1x_2+x_1x_3+x_2x_3+q_{12}+q_{23}) E_{3}^{(3)} \\
&=E_1^{(1)} (x_1x_2+x_1x_4+x_2x_4+q_{12}+q_{23}) E_{3}^{(3)}, 
\end{align*}
we compute 
\begin{align*}
\partial_3(\tau_3(\SS_{4321}^\q)) &= E_1^{(1)} (-x_1-x_2) E_{3}^{(3)} + E_1^{(1)} s_3(x_1x_2+x_1x_4+x_2x_4+q_{12}+q_{23}) E_{2}^{(2)} \\
&= -E_1^{(1)} E_1^{(2)} E_{3}^{(3)} + E_1^{(1)} E_2^{(3)} E_{2}^{(2)} =E_{1,2,2} - E_{1,1,3} =\SS_{4312}^\q.
\end{align*}
\end{example}

\begin{remark}
In order to compute $\tau_k (\SS_w^\q)$, we need to know a $\Z$-linear combination of standard elementary monomials for $\SS_w$. 
However, if we know it, then one can compute a $\Z$-linear combination of standard elementary monomials for $\SS_{ws_k}$ by using \cite[Lemmas~3.1 and 3.2]{FGP}, which yields a computation for $\SS_{ws_k}^\q$. 
\end{remark}

In the classical limit $q_{ij}=0$ for all $1 \leq i < j \leq n$, we have $\partial_k(\tau_k(\SS_w)) = \SS_{ws_k}$ when $w(k) > w(k+1)$.
However, $\tau_k(\SS_w)$ is not equal to $\SS_w$ in general, so it does not seem to be a $\q$-analogue of \eqref{eq:divided_difference_Schubert}. 
For this, we introduce another operator $\eta_k$ for $k \in [n-1]$. 
For each $E_{i_1,\ldots,i_m} \ (0 \leq i_k \leq k)$, we define $\eta_k(E_{i_1,\ldots,i_m})$ by the following formula 
\begin{align*}
\sum_{p=1}^{k-1} q_{k-p \, k} E_{i_1}^{(1)} \cdots E_{i_{k-2}}^{(k-2)} \left( \sum_{\ell \geq 0} \big( E_{i_{k-1}-\ell-1}^{(k-1)} E_{i_k+\ell-p-1}^{(k-1-p)} - E_{i_{k-1}-\ell-p-1}^{(k-1-p)} E_{i_k+\ell-1}^{(k-1)} \big) \right) E_{i_{k+1}}^{(k+1)} \cdots E_{i_m}^{(m)}, 
\end{align*}
where we take $\eta_1(E_{i_1,\ldots,i_m}) = 0$.  
In general, if we write $f \in \Z[x_1,\ldots,x_n]$ as $f = \sum_{i_1,\ldots,i_m} c_{i_1,\ldots,i_m} e_{i_1,\ldots,i_m}$ for some unique $c_{i_1,\ldots,i_m} \in \Z$, then we define $\eta_k(F)$ for the $\q$-quantization $F$ of $f$ by 
\begin{align*}
\eta_k(F) = \sum_{i_1,\ldots,i_m} c_{i_1,\ldots,i_m} \eta_k(E_{i_1,\ldots,i_m}). 
\end{align*}

\begin{theorem} \label{theorem:divided_difference_quantumSchubert}
Let $w$ be a permutation in $S_n$.
Then we have 
\begin{align*} 
\partial_k (\SS_w^\q) = \begin{cases}
\SS_{ws_k}^\q + \eta_k(\SS_w^\q) \ \ \ &\textrm{if} \ w(k) > w(k+1); \\
0 \ \ \ &\textrm{if} \ w(k) < w(k+1).
\end{cases}
\end{align*}
Note that $\eta_k(\SS_w^\q) = 0$ in setting $q_{ij}=0$ for all $1 \leq i < j \leq n$. 
\end{theorem}

\begin{proof}
The statement for $w(k) < w(k+1)$ follows from Proposition~\ref{proposition:divided_difference_quantumSchubert} with $\partial_k^2=0$. 
Consider the case $w(k) > w(k+1)$ below.
If we write $\SS_w = \sum_{i_1,\ldots,i_{n-1}} c_{i_1,\ldots,i_{n-1}} e_{i_1,\ldots,i_{n-1}}$ for some unique $c_{i_1,\ldots,i_{n-1}} \in \Z$, then one has 
\begin{align*}
\partial_k(\SS_w^\q) = \sum_{i_1,\ldots,i_{n-1}} c_{i_1,\ldots,i_{n-1}} E_{i_1}^{(1)} \cdots E_{i_{k-1}}^{(k-1)} E_{i_k-1}^{(k-1)}E_{i_{k+1}}^{(k+1)}\cdots E_{i_{n-1}}^{(n-1)}   
\end{align*}
from Proposition~\ref{proposition:divided_difference_E}.
By using \eqref{eq:propertyE2} we obtain
\begin{align*}
E_{i_{k-1}}^{(k-1)}E_{i_k-1}^{(k-1)} = & \left( E_{i_{k-1}}^{(k)}E_{i_k-1}^{(k-1)} + \sum_{\ell \geq 1} E_{i_{k-1}-\ell}^{(k)}E_{i_k-1+\ell}^{(k-1)} - \sum_{\ell \geq 1} E_{i_{k-1}-\ell}^{(k-1)}E_{i_k-1+\ell}^{(k)} \right) \\
& \hspace{10pt} + \sum_{p=1}^{k-1} q_{k-p \, k} \left( \sum_{\ell \geq 0} \big( E_{i_{k-1}-\ell-1}^{(k-1)} E_{i_k+\ell-p-1}^{(k-1-p)} - E_{i_{k-1}-\ell-p-1}^{(k-1-p)} E_{i_k+\ell-1}^{(k-1)} \big) \right). 
\end{align*}
As seen in the proof of Proposition~\ref{proposition:divided_difference_quantumSchubert}, we know an explicit presentation for $\SS_{ws_k}^\q$. 
Hence, we obtain 
\begin{align*}
\partial_k(\SS_w^{\q}) = \SS_{ws_k}^\q + \sum_{i_1,\ldots,i_{n-1}} c_{i_1,\ldots,i_{n-1}} \eta_k(E_{i_1,\ldots,i_{n-1}}) =  \SS_{ws_k}^\q + \eta_k(\SS_w^{\q}),
\end{align*}
as desired.
\end{proof}

Consider the algebra $\Z[\q]\langle \partial_1, \ldots,\partial_{n-1} \rangle$ generated by the divided difference operators $\partial_1,\ldots, \partial_{n-1}$ over $\Z[\q]=\Z[q_{ij} \mid 1 \leq i < j \leq n]$. 
We define $\partial_1^\q, \partial_2^\q \in \Z[\q]\langle \partial_1, \ldots,\partial_{n-1} \rangle$ by 
\begin{align*}
\partial_1^\q &= \partial_1; \\ 
\partial_2^\q &= \partial_2(1-q_{12}\partial_1\partial_2).
\end{align*}

\begin{corollary}
Let $w \in S_n$. For $k=1,2$, we have 
\begin{align} \label{eq:quantum_divided_difference} 
\partial_k^\q (\SS_w^\q) = \begin{cases}
\SS_{ws_k}^\q \ \ \ &\textrm{if} \ w(k) > w(k+1); \\
0 \ \ \ &\textrm{if} \ w(k) < w(k+1).
\end{cases}
\end{align}
\end{corollary}

\begin{proof}
Assume that $w(k) < w(k+1)$. 
Then the equality $\partial_k^\q (\SS_w^\q) =0$ is straightforward from Theorem~\ref{theorem:divided_difference_quantumSchubert}.
We now suppose that $w(k) > w(k+1)$.
By Theorem~\ref{theorem:divided_difference_quantumSchubert} we compute $\eta_k(\SS_w^\q)$. 
If $k=1$, then $\eta_1(\SS_w^\q)=0$ by the definition. 
Put $\SS_w = \sum_{i_1,\ldots,i_{n-1}} c_{i_1,\ldots,i_{n-1}} e_{i_1,\ldots,i_{n-1}} \ (c_{i_1,\ldots,i_{n-1}} \in \Z)$.
If $k=2$, then we compute
\begin{align*}
\eta_2(\SS_w^\q) &= q_{12} \sum_{i_1,\ldots,i_{n-1}} c_{i_1,\ldots,i_{n-1}} \sum_{\ell \geq 0} \big( E_{i_1-\ell-1}^{(1)} E_{i_2+\ell-2}^{(0)} - E_{i_1-\ell-2}^{(0)} E_{i_2+\ell-1}^{(1)} \big) E_{i_3}^{(3)} \cdots E_{i_{n-1}}^{(n-1)} \\
&= q_{12} \sum_{i_1,\ldots,i_{n-1}} c_{i_1,\ldots,i_{n-1}} E_{i_1+i_2-3}^{(1)} E_{i_3}^{(3)} \cdots E_{i_{n-1}}^{(n-1)}. 
\end{align*}
Since $i_1 \leq 1$ and $i_2 \leq 2$, one has $E_{i_1+i_2-3}^{(1)}=1$ if $(i_1,i_2)=(1,2)$ and $E_{i_1+i_2-3}^{(1)}=0$ otherwise. 
Thus, we obtain $\partial_1\partial_2\partial_1(E_{i_1}^{(1)}E_{i_2}^{(2)}) = E_{i_1+i_2-3}^{(1)}$ which yields that $\eta_2(\SS_w^\q) = q_{12} \partial_2\partial_1\partial_2(\SS_w^\q)$.  
Then it follows from Theorem~\ref{theorem:divided_difference_quantumSchubert} that $\partial_k^\q(\SS_w^\q) = \SS_{ws_k}^\q$ if $w(k) > w(k+1)$ for $k \leq 2$.
\end{proof}

A computation of $\eta_k(\SS_w^\q)$ for $k \geq 3$ is more complicated than the case when $k \leq 2$. 
It would be interesting to find a formula of $\partial_k^\q$ for $k \geq 3$ satisfying \eqref{eq:quantum_divided_difference}.

\end{document}